\documentclass[11pt]{article}
\usepackage[utf8]{inputenc}

\usepackage[numbers,sort&compress]{natbib}

\usepackage{algorithm}
\usepackage{pgf}
\usepackage{tikz}
\usetikzlibrary{arrows, decorations.pathmorphing, backgrounds, positioning, fit, petri, automata}
\definecolor{yellow1}{rgb}{1,0.8,0.2}

\usepackage{amsmath,amsthm,amsfonts,amssymb,amsbsy,amsopn,amstext}
\usepackage{graphicx,color,subfigure}
\usepackage{mathbbol,mathrsfs}
\usepackage{float,ccaption}
\usepackage{indentfirst}
\usepackage{appendix}

\usepackage{authblk}
\usepackage{lipsum}

\newtheorem{thm}{Theorem}

\newtheorem{lem}{Lemma}

\newtheorem{ass}{Assumption}

\renewcommand{\d}{\ensuremath{\mathrm{d}}} 
\renewcommand{\P}{\ensuremath{\mathbb{P}}} 
\newcommand{\Cov}{\ensuremath{\mathrm{Cov}}} 
\newcommand{\define}{:=}

\DeclareMathOperator*{\ri}{ri}
\DeclareMathOperator*{\st}{s.t.}

\DeclareMathOperator*{\amax}{argmax}
\DeclareMathOperator*{\amin}{argmin}

\newcommand{\al}{\ensuremath{\alpha}}

\newcommand{\q}{\mathrm{Q}}

\begin{document}
\title{Asymptotic properties of dual averaging algorithm for constrained distributed stochastic optimization\footnote{The work of the first author and the third author is supported by NSFC 11971090 and Fundamental Research Funds for the Central Universities under grant DUT19LK24. The research of the second author is supported by NSFC under grant 61203118, and the Fundamental Research Funds for the Central Universities under grant DUT20LK03.}}

\author[a]{Shengchao Zhao\thanks{email: zhaoshengchao@mail.dlut.edu.cn}}
\author[a]{Xing-Min Chen\thanks{email: xmchen@dlut.edu.cn}}
\author[a]{Yongchao Liu\thanks{email: lyc@dlut.edu.cn}}
\affil[a]{School of Mathematical Sciences, Dalian University of Technology, Dalian 116024, China}


\maketitle
\begin{abstract}
Considering the constrained stochastic optimization problem over a time-varying random network, where the agents are to collectively minimize a sum of objective functions subject to a common constraint set, we investigate asymptotic properties of a distributed algorithm based on dual averaging of gradients. Different from most existing works on distributed dual averaging algorithms that mainly concentrating on their non-asymptotic properties, we not only prove almost sure convergence and the rate of almost sure convergence, but also asymptotic normality and asymptotic efficiency of the algorithm. Firstly, for general constrained convex optimization problem distributed over a random network, we prove that almost sure consensus can be archived and the estimates of agents converge to the same optimal point. For the case of linear constrained convex optimization, we show that the mirror map of the averaged dual sequence identifies the active constraints of the optimal solution with probability 1, which helps us to prove the almost sure convergence rate and then establish asymptotic normality of the algorithm. Furthermore, we also verify that the algorithm is asymptotically optimal. To the best of our knowledge, it seems to be the first asymptotic normality result for constrained distributed optimization algorithms. Finally, a numerical example is provided to justify the theoretical analysis.
\end{abstract}




\section{Introduction}
Distributed algorithms for solving optimization problems that are defined over  networks have been receiving increasing attention from researchers since the earlier seminal work \cite{tsitsiklis1984problems,tsitsiklis1986distributed,bertsekas1989parallel}.
The most concerned problem among which is to optimize a sum of local objective functions of agents subject to the intersection of their local constraint sets, where the agents are connected through a communication network with each objective and constraint held privately. A large number of problems, such as multi-agent coordination \cite{ren2008distributed}, wireless networks \cite{naghshineh1996distributed,fitzek2006cooperation}, machine learning \cite{lian2017can}, can be transformed into distributed optimization problems. In practice these problems are often random or large-scale, so they are very suitable to be solved by stochastic approximation (SA) based distributed algorithms. Over the last decades, numerous algorithms for distributed stochastic optimization have been developed and various scenarios have been considered, such as stochastic sub-gradient \cite{ram2010distributed}, distributed dual averaging \cite{duchi2011dual}, random gradient-free \cite{yuan2014randomized,chen2017strong}, push-sum method \cite{nedic2016push-sum}; or, with the same local constraint \cite{bianchi2012convergence}, with the different local constraint \cite{shah2018distributed,lee2013distributed}, with asynchronous communications \cite{ram2009asynchronous}. In most of the mentioned works, asymptotic convergence such as convergence in mean (and further the rate in mean) or almost sure convergence, or non-asymptotic properties in expectation, are commonly concerned.

Asymptotic normality and asymptotic efficiency are important topics of stochastic algorithms, which have been studied in SA for a long time. For centralized problem, the asymptotic normality of one-dimensional and multi-dimensional SA was provided in \cite{chung1954stochastic,sacks1958asymptotic}
and \cite{fabian1968asymptotic}, respectively. To archive asymptotic efficiency the so-called adaptive SA may be concerned, see e.g. \cite{ruppert1985newton}, but it requires rather restrictive conditions to guarantee its convergence and optimality. On the other hand, the averaging technique introduced in \cite{polyak1992acceleration} has been widely used. Recently, \cite{duchi2016asymptotic} gave the asymptotic efficiency of the dual average algorithm for solving linear constrained and nonlinear constrained optimization problems respectively. For decentralized problem, however, asymptotic normality and asymptotic efficiency results are rather limited. The asymptotic normality and asymptotic efficiency of a distributed stochastic approximation algorithm were proven in \cite{bianchi2013performance}; a distributed stochastic primal dual algorithm was proposed, and then whose asymptotic normality and asymptotic efficiency were provided in \cite{lei2018asymptotic}. However, all of the above works on distributed optimization are concentrated on unconstrained problems. Inspired by \cite{duchi2016asymptotic,duchi2011dual}, we provide the asymptotic normality of distributed dual averaging algorithm for linear constrained problem.

The dual averaging algorithm was introduced by \cite{nesterov2009primal} in deterministic settings, and further analyzed and developed by many authors. For instance, \cite{xiao2010DA} extended it to stochastic settings and composite optimization problem. \cite{lee2012manifold} proved that the dual averaging algorithm can identify the optimal manifold with a high probability before finding the optimal solution, and provided a strategy to search for the optimal solution in the optimal manifold after identifying the active set. \cite{duchi2016asymptotic} showed that variants of Nesterov's dual averaging algorithm guarantee almost sure finite time identification of active constraints in constrained stochastic optimization problems. The reason why the optimal manifold identification property is so concerned is that it contributes to prove algorithm’s asymptotic normality from a theoretical viewpoint, while it is also helpful to reduce the amount of computation and save storage space of data from a practical viewpoint, especially for sparsity problem.

The dual averaging algorithm was developed to solve distributed optimization problems in \cite{duchi2011dual,agarwal2011distributed}, where it was shown how do the network size and topology influence sharp bounds on convergence rates in \cite{duchi2011dual}, and how do the delays in stochastic gradient information affect the convergence results in \cite{agarwal2011distributed}. Applying the dual averaging algorithm to distributed optimization problems was concerned by many authors. For example, the effects of deterministic and probabilistic massage quantization on distributed dual averaging algorithms for multi-agent optimization problem was considered in \cite{yuan2012distributed}. \cite{hosseini2013online} extended the distributed algorithm based on dual subgradient averaging to the online setting and provided an upper bound on regret as a function of connectivity in the underlying network. Recently, \cite{liang2019distributed} proposed a distributed quasi-monotone sub-gradient algorithm, and proved this algorithm’s asymptotic convergence, where quasi-monotone algorithm introduced in \cite{nesterov2015quasi} is a modification of dual averaging algorithm. However, these works are mostly focused on the non-asymptotic convergence analysis and asymptotic properties such as asymptotic normality have not been resolved for the distributed dual averaging algorithm.

In this paper, we investigate a dual averaging algorithm for the distributed stochastic optimization problem subject to a common constraint set over a time-varying random network. We first establish the almost sure consensus and almost sure convergence of the algorithm. And then in the linear constraint case we provide the almost sure active set identification, and with whose help we are able to analyze the almost sure convergence rate and prove the asymptotic normality as well as asymptotic efficiency of the algorithm. The main contributions of the paper are summarized as follows.
\begin{itemize}
	\item [(a)]	Different from most existing works on distributed dual averaging algorithms that mainly focus on their non-asymptotic properties, we prove all agents' estimates converge to the same optimal solution almost surely for general constrained optimization problem over time-varying random networks.
	In particular, the weight matrices are not restricted to be doubly stochastic, which are only required to be column stochastic in mean sense except for row stochasticity.
	\item [(b)] Motivated by the idea of active set identification, we extend the method in \cite{duchi2016asymptotic} to distributed scenario, and show that the mirror map of the averaged dual sequence identifies the active set of the optimal solution after finite steps almost surely. As explained earlier, once the estimates enter into the optimal manifold, asymptotic convergence properties of the algorithm can be proved as unconstrained stochastic approximation algorithms. On this basis, we provide a novel result on almost convergence rate of the distributed dual averaging algorithm for the case of linear constrained convex distributed stochastic optimization.
	\item[(c)] 	Different from \cite{bianchi2013performance,lei2018asymptotic} that concentrate on unconstrained distributed optimization problem, we provide asymptotic normality and asymptotic efficiency of distributed dual averaging algorithms for linear constrained distributed optimization, which seems to be the first asymptotic normality result for constrained distributed optimization algorithms as far as we know.
\end{itemize}

The remainder of this paper is organized as follows. Section 2 introduces the distributed optimization problem model and a distributed dual averaging (DDA for short) algorithm. Section 3 gives not only the almost sure convergence of DDA algorithm for the convex optimization problem with general constraints, but also the almost sure convergence rate of DDA algorithm in the case objective function is restricted strong convex and constraints are linear. Section 4 proves the asymptotic normality and asymptotic efficiency of DDA algorithm. Section 5 presents a numerical example to justify these theoretic results.


\textbf{Notations and basic definitions}: Throughout this paper, we use the following notation.
$\mathbb{R}^d$ denotes the $d$-dimension Euclidean space with norm $\|\cdot\|$ and $\mathbb{R}^{d}_+:=\{x\in \mathbb{R}^d :x\geq 0\}$.
$\textbf{1}:=(1~1\dots1)^T\in\mathbb{R}^m$, $I_d\in \mathbb{R}^{d\times d}$ denotes the identity matrix and $\textbf{0}$ denotes the zero matrix of compatible dimension, respectively. For a matrix $A$, $A^\dag$ is its Moore-Penrose inverse and $\|A\|=\sup_{\|x\|=1}\|Ax\|$ is the spectral norm. For two matrices  $A$ and $B$,
$A\otimes B$ stands for the Kronecker product. Given a set $\mathcal{X}\subseteq\mathbb{R}^d$,	$1_{\mathcal{X}}$ denotes the characteristic function of set $\mathcal{X}$, which means that it equals $1$ if $x\in\mathcal{X}$, and $0$ otherwie. $\ri(\mathcal{X})$ denotes the set of relative interior of a non-empty convex set $\mathcal{X}$.
For a closed convex set $\mathcal{X}\subseteq\mathbb{R}^d$,  $\mathcal{N}_{\mathcal{X}}(x)$ denotes the normal cone and $P_\mathcal{X}(z)$ denotes the projection operator,  that is,
\begin{equation*}
\mathcal{N}_{\mathcal{X}}(x)\define \{v\in\mathbb{R}^d:\langle v,y-x\rangle\le 0,\forall y\in\mathcal{X}\},\quad P_\mathcal{X}(z)=\arg\min_{x\in\mathcal{X}}\|x-z\|.
\end{equation*}
For a sequence of random vectors $\{\xi_k\}$ and a random vector $\xi$,  $\xi_k\stackrel{\text{a.s.}}{\rightarrow} \xi$ and  $\xi_k\stackrel{d}{\rightarrow} \xi$  stand for  $\{\xi_k\}$ converges to $\xi$ almost surely (a.s. for short) and in distribution, respectively.

\section{Distributed optimization problem and dual averaging method}
Consider the following distributed constrained stochastic optimization problem:
\begin{equation}\label{problem model}
\min ~f(x)=\sum_{j=1}^mf_j(x)\quad \st ~x\in \mathcal{X},
\end{equation}
where $f_j(x)\define\mathbb{E}[F_j(x;\xi_j)], \; j=1,\cdots, m$ with $\xi_j,\, j=1,\cdots, m$ being a random vector defined on a probability space $(\Omega,\mathcal{F},\P)$ with support set $\Xi_j$,  $\mathbb{E}\left[\cdot\right]$  denotes the expected value with respect to
probability measure $\P$ and  $\mathcal{X}\subset\mathbb{R}^d$ is a closed convex  set.

In problem (\ref{problem model}), each agent $j$ shares the common constraint set $\mathcal{X}$ but holds the private information  on objective function $f_j(x)$, such as, the value of sampled function or the corresponding gradient. But each agent can communicate with its immediate neighbors to cooperatively solve the constrained optimization problem (\ref{problem model}). For convenience, denote by $f^*=\inf_{x\in \mathcal{X}} f(x)$ the optimal value of problem (\ref{problem model}), and by $\mathcal{X}^*=\{x\in \mathcal{X}:f(x)=f^*\}$ the optimal solution set.

The network over which the agents communicate at time $k$ is represented by a directed graph $G_k=(V,E_k)$, where $V=\{1,2,\ldots,m\}$ is the node set, and $E_k\subset V\times V$ is associated with the weight matrix $A_k\in\mathbb{R}^{m \times m}$ through
\begin{equation*}
E_k\define\{(j, i): [A_k]_{ij}> 0, i,j\in V\},
\end{equation*}
where $[A_k]_{ij}$ is the $(i,j)$-th entry of matrix $A_k$. At time $k$,
$N_{j,k}\define\{i\in V:(i,j)\in E_k\}$
denotes the neighbors of agent $j$.

The dual averaging  method is proposed by Nesterov \cite{nesterov2009primal}. 
Consider the following optimization problem
\begin{equation*}
\min_{x\in \mathcal{X}} ~h(x),
\end{equation*}
where $h(x):\mathbb{R}^d\to\mathbb{R}$ is a differentiable convex function, $\mathcal{X}\subset\mathbb{R}^d$ is a closed convex set. 
The dual averaging  method
involves two alternate processes:
\begin{align*}
z_k&=z_{k-1}-\alpha_k \nabla h(x_k),\\
x_{k+1}&= \amax_{x\in\mathcal{X}}\{\langle z_k, x\rangle-\psi(x)\},
\end{align*}
where $\psi:\mathcal{X}\to\mathbb{R}$ is called \emph{regularizer}, which is a continuous and strongly convex function on $\mathcal{X}$,  that is, there exists some $\sigma>0$ such that
\begin{equation*}
\psi(\lambda x+(1-\lambda)y)\le\lambda \psi(x)+(1-\lambda)\psi(y)-\frac{\sigma}{2}\lambda(1-\lambda)\|x-y\|^2
\end{equation*}
for all $x,y\in\mathcal{X}$ and $\lambda\in[0,1]$.
For example, the Euclidean regularization is mostly common used in the literature.

Recently, the dual averaging algorithm has been developed to solve distributed optimization problems in \cite{duchi2011dual,yuan2012distributed,agarwal2011distributed,hosseini2013online,liang2019distributed}. In this paper,  we investigate a variant of the distributed dual averaging algorithm proposed in \cite{duchi2011dual} and focus on its asymptotic properties, which reads as the following.
\begin{algorithm}[H]
	\caption{Distributed dual averaging algorithm}
	\label{alg:DDA}
	\textbf{Initialization:} For any $1\leq j\leq m$, agent $j$ initializes its dual variable $z_{j,0}\in\mathbb{R}^d$ (possibly randomly).\\
	\textbf{General step:} At time $k=1,2,\cdots$, update weighted matrix $A_k$ and stepsize $\al_{k}>0$; agent $j$ maintains a pair of vectors $\{x_{j,k},z_{j,k}\}$, exchanges $z_{j,k}$ between agents, and performs the following primal-dual iteration locally.
	\begin{itemize}
		\item[1.] \textbf{Primal step}: Update the primal estimate by a projection defined by $\psi(x)$
		\begin{equation}\label{DDA-primal}
		x_{j,k}=\amax_{x\in\mathcal{X}}\{\langle z_{j,k-1},x\rangle-\psi(x)\}.
		\end{equation}
		\item[2.] \textbf{Dual step:} Draw $\xi_{j, k}\stackrel{i.i.d.}{\sim}\P$, compute $\nabla F_j(x_{j,k};\xi_{j,k})$, update the dual estimate by
		\begin{equation}\label{DDA-dual}
		z_{j,k}=\sum_{i\in N_{j,k}}[A_k]_{ji}z_{i,k-1}-\alpha_{k}\nabla F_j(x_{j,k};\xi_{j,k}).
		\end{equation}
	\end{itemize}
\end{algorithm}
Throughout the paper, we define the filtration
\begin{equation*}
\mathcal{F}_k=\sigma\{z_{j,0},\xi_{j,t},~A_t:j\in V, 1\leq t\leq k-1\},\
\mathcal{F}_1 =\sigma\{z_{j,0},j\in V\}.
\end{equation*}
It is obvious that $z_{j,k-1},x_{j,k}$ is adapted to $\mathcal{F}_{k}$.

\section{Almost sure convergence and convergence rate}
In this section, we study the almost sure convergence of Algorithm \ref{alg:DDA}. We show that each iteration $x_{j,k}$ converges almost surely to the same solution in $\mathcal{X}^*$, for the case where $f_j(\cdot)$ is convex for any $1\le j\le m$. If $f(\cdot)$ is further restricted strong convex, we may provide an estimation of the almost sure convergence rate, which will be used to analyze the asymptotic normality of each estimate $x_{j,k}$ to the optimal solution.

\subsection{Almost sure convergence}

We first introduce the conditions on objective functions, constraint set, network topology, step-size and  sample. 

\begin{ass}[\textbf{objective function}]\label{ass:lip func} For any $1\leq j\leq m$, \\
	(i) $F_j(\cdot;\xi_j)$ is differentiable convex function on $\mathcal{X}$ for any $\xi_j$; \\
	(ii) $F_j(\cdot;\xi_j)$ is Lipschitz continuous on $\mathcal{X}$, that is,
	\begin{equation}\label{zeroth order massage}
	|F_j(x;\xi_j)-F_j(y;\xi_j)|\leq L_{0,j}(\xi_j)\|x-y\|, \quad \forall x,y\in\mathcal{X},
	\end{equation}
	where $L_{0,j}(\xi_j)$ is  measurable and  $\mathbb{E}[L_{0,j}^p(\xi_j)]<\infty$ for some $p\geq2$.
\end{ass}
The Lipschitz continuity of $F_j(\cdot;\xi_j)$ implies that $f_j(\cdot)$ is Lipschitz continuous, and that the gradient $\nabla F_j(x;\xi_j),\nabla f_j(x)$ are bounded by $L_{0,j}(\xi_j)$ and $\mathbb{E}[L_{0,j}(\xi_j)]$, respectively.
For the convergence of Algorithm \ref{alg:DDA}, the condition $p=2$ in part (ii) of Assumption \ref{ass:lip func} is sufficient. When studying the asymptotic normality of the algorithm, $p>2$ is needed to verify Lindeberg's condition. Moreover, for easy of the notation, we denote the observation noise of gradient $\nabla f_j(x_{j,k})$ by
\begin{equation}
s_{j,k}\define \nabla F_j(x_{j,k};\xi_{j,k})-\nabla f_j(x_{j,k}),
\end{equation}
and
\begin{equation}
\label{eq: L0}
L_0=\max_{1\leq j \leq m}\mathbb{E}[L_{0,j}(\xi_j)],\quad
L_0^p=\max_{1\leq j \leq m} \mathbb{E}[L^p_{0,j}(\xi_j)]
\end{equation}
throughout the paper.

We now turn to assumptions on the weight matrices $A_k$, which are commonly assumed to be doubly stochastic in most works (for instance \cite{ram2010distributed,duchi2011dual,lee2013distributed,shah2018distributed}).  However, in practice it is rather easy to implement row-stochasticity ($A_k\textbf{1}=\textbf{1}$) but hard to ensure  column-stochasticity ($\textbf{1}^TA_k=\textbf{1}^T$) since which implies more stringent restrictions on the network. Motivated by \cite[Assumption 1]{bianchi2012convergence}, we investigate Algorithm \ref{alg:DDA} under the relatively weaker conditions.
\begin{ass}[\textbf{weight matrices}]\label{ass:A rho} Let $A_k$ be the  weight matrix at step $k$. Assume that\\
	(i) $A_k$ is a sequences of matrix-valued random variables with nonnegative components and
	\begin{equation*}
	A_k\textbf{1}=\textbf{1}, \quad \textbf{1}^T\mathbb{E}\left[A_k\right]=\textbf{1}^T,   \quad\forall k\geq 1.
	\end{equation*}
	(ii) $\rho_k$ denotes the spectral norm  of matrix $\mathbb{E}\left[A_k^T(I_m-\frac{\textbf{1}\textbf{1}^T}{m})A_k\right]$ and
	\begin{equation}\label{spectral norm}
	\lim_{k\rightarrow\infty}k(1-\rho_k)=\infty.
	\end{equation}
	(iii) Matrix $A_k$ is independent of $\sigma$-algebra $\mathcal{F}_k$.
\end{ass}

Assumption \ref{ass:A rho} allows the broadcast gossip matrices and  (\ref{spectral norm}) holds if $\sup \rho_k<1$.

\begin{ass}[\textbf{step-size}]\label{ass:rho a}(i) $\alpha_{k}>0$ is nonincreasing and $\sum_{k=1}^\infty\alpha_k=\infty$. \\
	(ii) There exists $\beta>0.5$ such that
	\begin{align}
	\lim_{k\rightarrow\infty}k^\beta\alpha_{k}&=0,\label{stepsize}\\
	\liminf_{k\rightarrow\infty}\frac{1-\rho_k}{k^\beta\alpha_{k}}&>0.\label{stepsize and spectral norm}
	\end{align}
\end{ass}

Note that (\ref{stepsize}) implies $\sum_{k=1}^\infty\alpha_k^2<\infty$, which combines with Assumption \ref{ass:rho a}(i) is commonly used in SA. (\ref{stepsize and spectral norm}) means that the exchange of information between agents becomes rare as $k\rightarrow\infty$. When $A_k$ is an independent and identically distributed (i.i.d.) sequence, then $\rho_k\equiv\rho$ is constant, and both (\ref{spectral norm}) and (\ref{stepsize and spectral norm}) hold if and only if $\rho <1$ \cite{bianchi2012convergence}.

\begin{ass}[\textbf{sample and $\sigma$-algebra}]\label{ass:sample} For any $1\leq i, j\leq m$, (i) $\xi_{j,1}, \xi_{j,2}, \cdots$ is i.i.d. sample; (ii) $\xi_{i,k}$ and $\xi_{j,k}$ are conditionally independent given $\mathcal{F}^{'}_k\define\sigma(\mathcal{F}_k\cup\sigma(A_k))$ when $i\neq j$; (iii) $\xi_{j,k}$ is conditionally independent of $A_k$ given $\mathcal{F}_k$.
\end{ass}
Above condition (i) and (ii) guarantee that the sequence of observation noise of gradient
$\{s_{j,k}\}$  is  a martingale difference sequence, that is,
\begin{equation}\label{mds}
\mathbb{E}\left[s_{j,k}\big|\mathcal{F}_k\right]=0,
\end{equation}
and  the conditional covariance
$
\Cov(\nabla F_j(x_{j,k};\xi_{j,k}),\nabla F_i(x_{i,k};\xi_{i,k})\big|\mathcal{F}_k)=0.
$

Combining condition (i) with Assumption \ref{ass:lip func} implies that
\begin{equation}\label{observe noise bound}
\begin{aligned}
\mathbb{E}\left[\|s_{j,k}\|^p\big|\mathcal{F}_k\right]
&\le \left(\left(\mathbb{E}\left[\|\nabla f_j(x_{j,k})\|^p\big|\mathcal{F}_k\right]\right)^{1/p}+\left(\mathbb{E}\left[\|\nabla F_j(x_{j,k};\xi_{j,k})\|^p\big|\mathcal{F}_k\right]\right)^{1/p}\right)^p\\
&\le \left(\left(L_0^p\right)^{1/p}+\left(L_0^p\right)^{1/p}\right)^p= 2^pL_0^p,
\end{aligned}
\end{equation}
where the Minkowski inequality and the fact that $\nabla F_j(x;\xi_j),\nabla f_j(x)$ are bounded by $L_{0,j}(\xi_j)$ and $\mathbb{E}[L^p_{0,j}(\xi_j)]<\infty$ respectively have been involved.

Condition (iii) is similar with \cite[Assumption 1 (c)]{bianchi2012convergence},  which ensures that weight matrix $A_k$ and $\xi_{j,k}$ are  independent conditionally on the past.

For the regularizer $\psi(\cdot)$,  recall the concepts of \emph{mirror map} \cite{zhou2020convergence}
\begin{equation}\label{mirror map}
\q\left(z\right)\define\amax_{x\in\mathcal{X}}\{\langle z,x\rangle-\psi(x)\}
\end{equation}
and  \emph{Fenchel coupling}
\begin{equation*}
R(x,z): =\psi(x)+\psi^*(z)-\langle x,z\rangle,\forall x\in\mathcal{X},z\in\mathbb{R}^d,
\end{equation*}
where $\psi^*(z)\define\sup_{x\in\mathcal{X}}\left\{\langle z,x\rangle-\psi(x)\right\}$ is the conjugate function of $\psi(x)$.
\begin{ass}[\textbf{regularizer $\psi(\cdot)$}]
	\label{ass:Fc}
	For any $x\in\mathcal{X}$, $R(x,z_k)\rightarrow 0$ whenever $\mathrm{\q}(z_k)\rightarrow x$.
\end{ass}

Assumption \ref{ass:Fc} is called ``reciprocity condition'' \cite[Assumption 3]{zhou2020convergence}. Most common regularizers such as the Euclidean and entropic regularizer satisfy this assumption, for details refer to \cite[Examples 2.7 and 2.8]{zhou2020convergence}.

In the next, we study the convergence of sequences $\{x_{j,k}\}$ generated by Algorithm \ref{alg:DDA}.   By definition (\ref{mirror map}), the first step of Algorithm \ref{alg:DDA} can be rewritten as
$
x_{j,k}=\q(z_{j,k-1}).
$
As a key step, we define two auxiliary sequences
\begin{equation}
\begin{aligned}\label{consensus dual averaging alrithm}
\bar{z}_k:=\frac{1}{m}\sum_{j=1}^mz_{j,k},  \quad
\bar{x}_{k+1}:=\q\left(\bar{z}_k\right)
\end{aligned}
\end{equation}
as reference sequences to measure the agent disagreements. It is obvious that $\bar{z}_{k-1}$ and $\bar{x}_k$ are adapted to $\mathcal{F}_{k}$.
By \cite[Lemma 1]{nesterov2009primal},
\begin{equation}\label{consensus}
\|x_{j,k}-\bar{x}_k\|=\|\q\left(z_{j,k-1}\right)-\q\left(\bar{z}_{k-1}\right)\|\le\|z_{j,k-1}-\bar{z}_{k-1}\|/\sigma,
\end{equation}
where $\sigma$ is the strongly convex parameter of $\psi(x)$. 
Then for any $ 1\le j\le m$,  we may study the consensus of $\{x_{j,k}\}$ by showing  $z_{j,k}-\bar{z}_k\to 0$.

\begin{lem}\label{lem:consis}
	Suppose Assumptions \ref{ass:lip func}-\ref{ass:sample} hold. Then, for any $1\leq j\leq m$,
	\begin{itemize}
		\item [\rm{(i)}]\begin{equation}\label{consensus rate}
		\sup_{k}k^{2\beta}\mathbb{E}\left[\|\bar{z}_k-z_{j,k}\|^2\right]<\infty,
		\end{equation}
		where constant $\beta$ is defined in Assumption \ref{ass:rho a}(ii).
		\item [\rm{(ii)}] Furthermore, 	
		\begin{equation}\label{consunsus sum 0}
		\sum_{k=1}^\infty\|\bar{z}_k-z_{j,k}\|^2<\infty~ \text{a.s.}
		\end{equation}
		and for any positive sequence $\{\gamma_k\}$ such that $\sum_{k=1}^\infty\gamma_kk^{-\beta}<\infty,$
		\begin{equation}\label{consunsus sum}
		\sum_{k=1}^\infty\gamma_k\|\bar{z}_k-z_{j,k}\|<\infty~\text{a.s.}
		\end{equation}
		\item [\rm{(iii)}] If Assumption \ref{ass:A rho}(ii) and Assumption \ref{ass:rho a} are replaced by
		\begin{itemize}
			\item[(a)] $A_k,   k=1,2, \cdots $ is i.i.d. and the spectral norm $\rho$ of matrix $\mathbb{E}\left[A_k^T(I_m-\frac{\textbf{1}\textbf{1}^T}{m})A_k\right]$ satisfies $\rho<1$,
			\item[(b)] $\alpha_{k}>0$ is nonincreasing, $\sum_{k=1}^\infty\alpha_k=\infty$, $\sum_{k=1}^\infty\alpha_k^2<\infty$, and $\lim_{k\rightarrow\infty}\frac{\al_{k}}{\al_{k+1}}=1$,
		\end{itemize}
		respectively. Then
		\begin{equation}\label{consensus rate 1}
		\sup_{k}\al_{k}^{-2}\mathbb{E}\left[\|\bar{z}_k-z_{j,k}\|^2\right]<\infty,
		\end{equation}
		and thus (\ref{consunsus sum 0}) holds.
	\end{itemize}	
\end{lem}
We provide the proof of Lemma \ref{lem:consis} in Appendix \ref{proof:consis}.

Lemma \ref{lem:consis} shows that $ z_{j,k}-\bar{z}_k, \forall 1\leq j\leq m$ converges to zero, which in turn implies the consensus of sequences $\{x_{j,k}\}, j=1,\cdots, m$.  Moreover, it also shows that  $\bar{z}_k-z_{j,k}$ tends to zero in the $2$-nd mean at rate $\mathcal{O}(k^{-2\beta})$ under Assumptions \ref{ass:A rho}-\ref{ass:rho a} and at rate $\mathcal{O}(\al_k^2)$ under stronger conditions, which is the key results for analysing the convergence rate and asymptotic normality of sequences $\{x_{j,k}\}, j=1,\cdots, m$.

\begin{thm}\label{thm:convergence}
	Suppose Assumptions \ref{ass:lip func}-\ref{ass:Fc} hold with $p=2$ in Assumption \ref{ass:lip func}(ii). Then $x_{j,k}, j=1,\cdots, m$ and $\bar{x}_k$ converge  to some point in $\mathcal{X}^*$ almost surely.
\end{thm}
\begin{proof}
	For any  fixed $x^*\in\mathcal{X}^*$, denote $R_k:=R(x^*,\bar{z}_k)\ge0$. By \cite[Lemma 3.2 (3.2b)]{zhou2020convergence},
	\begin{equation*}
	\begin{aligned}
	R_{k}
	&\le R_{k-1}
	+\big\langle\q(\bar{z}_{k-1})-x^*,\bar{z}_k-\bar{z}_{k-1}\big\rangle
	+\frac{1}{2\sigma}\|\bar{z}_k-\bar{z}_{k-1}\|^2\\
	&= R_{k-1}
	+\big\langle \bar{x}_k-x^*,\bar{z}_k-\bar{z}_{k-1}\big\rangle
	+\frac{1}{2\sigma}\|\bar{z}_k-\bar{z}_{k-1}\|^2,
	\end{aligned}
	\end{equation*}
	where $\sigma$ is the strongly convex parameter of the regularizer $\psi(x)$.	
	Note that $\bar{z}_{k-1}$ is adapted to $\mathcal{F}_{k}$, we have by taking conditional expectation on both sides of the above inequality with respect to $\mathcal{F}_k$ that
	\begin{equation} \label{Lyapunov function}
	\mathbb{E}\left[R_{k}|\mathcal{F}_k\right]\leq R_{k-1}
	+\mathbb{E}\left[\big\langle \bar{x}_k-x^*,\bar{z}_k-\bar{z}_{k-1}\big\rangle|\mathcal{F}_k\right]
	+\frac{1}{2\sigma}\mathbb{E}\left[\|\bar{z}_k-\bar{z}_{k-1}\|^2|\mathcal{F}_k\right].
	\end{equation}
	
	Firstly, we  focus on the second term $\mathbb{E}\left[\big\langle \bar{x}_k-x^*,\bar{z}_k-\bar{z}_{k-1}\big\rangle|\mathcal{F}_k\right]$ on the right-hand side of (\ref{Lyapunov function}).
	By definitions,  $\bar{x}_k,x_{j,k}$ and $z_{j,k-1}$ are adapted to $\mathcal{F}_{k}$ and then
	\begin{equation*}
	\begin{aligned}
	&\mathbb{E}\left[\big\langle \bar{x}_k-x^*,\bar{z}_k-\bar{z}_{k-1}\big\rangle|\mathcal{F}_k\right]
	=\big\langle \bar{x}_k-x^*,\mathbb{E}\left[\bar{z}_k-\bar{z}_{k-1}|\mathcal{F}_k\right]\big\rangle\\
	=&\left\langle \bar{x}_k-x^*,\mathbb{E}\left[\left.\frac{1}{m}\sum_{j=1}^m\left(\sum_{i=1}^m[A_k]_{ij}-1\right)z_{j,k-1}-\frac{\alpha_{k}}{m}\sum_{j=1}^m\nabla F_j(x_{j,k};\xi_{j,k})\right|\mathcal{F}_k \right]\right\rangle\\
	=&\left\langle \bar{x}_k-x^*,\frac{1}{m}\sum_{j=1}^m\mathbb{E}\left[\left.\sum_{i=1}^m[A_k]_{ij}-1\right|\mathcal{F}_k\right]z_{j,k-1}-\frac{\alpha_{k}}{m}\sum_{j=1}^m\nabla f_j(x_{j,k})\right\rangle\\
	=&\left\langle \bar{x}_k-x^*,-\frac{\alpha_{k}}{m}\sum_{j=1}^m\nabla f_j(x_{j,k})\right\rangle,
	\end{aligned}
	\end{equation*}
	where the second equality follows from the definitions of $\bar z_k$ in (\ref{consensus dual averaging alrithm}) and $z_{j,k}$ in (\ref{DDA-dual}), and the last equality follows from the fact that
	\begin{equation}\label{col-stoc-in-mean}
	\mathbb{E}\left[\sum_{i=1}^m[A_k]_{ij}\bigg|\mathcal{F}_k\right]=\mathbb{E}\left[\sum_{i=1}^m[A_k]_{ij}\right]=1,\ 1\leq j\leq m,
	\end{equation}
	see Assumption \ref{ass:A rho}(i) and \ref{ass:A rho}(iii) for details.
	Moreover,
	\begin{equation*}
	\begin{aligned}
	\langle \bar{x}_k-x^*,-\nabla f_j(x_{j,k})\rangle
	=&\langle\nabla f_j(x_{j,k}),x^*-x_{j,k}\rangle+\langle\nabla f_j(x_{j,k}),x_{j,k}-\bar{x}_k\rangle\\
	\leq&f_j(x^*)-f_j(x_{j,k})+\|\nabla f_j(x_{j,k})\|\|x_{j,k}-\bar{x}_k\|\\
	\leq&f_j(x^*)-f_j(\bar{x}_k)+f_j(\bar{x}_k)-f_j(x_{j,k})+L_0\|x_{j,k}-\bar{x}_k\|\\
	\leq&f_j(x^*)-f_j(\bar{x}_k)+2L_0\|x_{j,k}-\bar{x}_k\|\\
	\leq&f_j(x^*)-f_j(\bar{x}_k)+{2L_0}\|z_{j,k-1}-\bar{z}_{k-1}\|/{\sigma},
	\end{aligned}
	\end{equation*}
	where $L_0$ is defined in (\ref{eq: L0}),
	the first inequality follows from the convexity of $f_j(\cdot)$ and the
	Cauchy-Schwarz inequality, the second and the third inequalities follow from the Lipsthitz condition (ii) of Assumption \ref{ass:lip func}
	and the  last inequality follows from (\ref{consensus}).	
	Consequently,
	\begin{equation}\label{second term of RH}
	\mathbb{E}\left[\big\langle \bar{x}_k-x^*,\bar{z}_k-\bar{z}_{k-1}\big\rangle\big|\mathcal{F}_k\right]\leq \frac{\alpha_{k}}{m}\big(f^*-f(\bar{x}_k)\big)+\frac{2\alpha_{k}L_0}{m\sigma}\sum_{j=1}^m\|z_{j,k-1}-\bar{z}_{k-1}\|.
	\end{equation}
	
	Next, we focus on the third term  $\frac{1}{2\sigma}\mathbb{E}\left[\|\bar{z}_k-\bar{z}_{k-1}\|^2\big|\mathcal{F}_k\right]$ on the right-hand side of (\ref{Lyapunov function}).
	\begin{equation}
	\begin{aligned}\label{third term of RH}
	&\frac{1}{2\sigma}\mathbb{E}\left[\|\bar{z}_k-\bar{z}_{k-1}\|^2|\mathcal{F}_k\right]\\
	=&\frac{1}{2\sigma}\mathbb{E}\left[\left\|\sum_{j=1}^m\frac{\sum_{i=1}^m[A_k]_{ij}}{m}(z_{j,k-1}-\bar{z}_{k-1})-\frac{\alpha_{k}}{m}\sum_{j=1}^m\nabla F_j(x_{j,k};\xi_{j,k})\right\|^2\bigg|\mathcal{F}_k\right]\\
	\leq& \frac{1}{\sigma }\mathbb{E}\left[\left\|\sum_{j=1}^m\frac{\sum_{i=1}^m[A_k]_{ij}}{m}\left(z_{j,k-1}-\bar{z}_{k-1}\right)\right\|^2\bigg|\mathcal{F}_k\right]+\frac{1}{\sigma}\mathbb{E}\left[\left\|\frac{\alpha_{k}}{m}\sum_{j=1}^m\nabla F_j(x_{j,k};\xi_{j,k})\right\|^2\bigg|\mathcal{F}_k\right]\\
	\leq& \frac{1}{\sigma }\sum_{j=1}^m\mathbb{E}\left[\frac{\sum_{i=1}^m[A_k]_{ij}}{m}\bigg|\mathcal{F}_k\right]\|\bar{z}_{k-1}-z_{j,k-1}\|^2+\frac{\alpha_{k}^2}{m\sigma}\sum_{j=1}^m\mathbb{E}\left[\|\nabla F_j(x_{j,k};\xi_{j,k})\|^2|\mathcal{F}_k\right]\\
	\leq& \frac{1}{m\sigma }\sum_{j=1}^m\|\bar{z}_{k-1}-z_{j,k-1}\|^2+\frac{L_0^2}{\sigma}\alpha_{k}^2,
	\end{aligned}
	\end{equation}
	where $L_0^2$ is defined in (\ref{eq: L0}),  the second inequality follows from the convexity of $\|\cdot\|^2$ and the fact that
	\begin{equation*}
	[A_k]_{ij}\geq0, \quad \sum_{j=1}^m\frac{\sum_{i=1}^m[A_k]_{ij}}{m}=1,
	\end{equation*}
	the last inequality follows from (\ref{col-stoc-in-mean}) 
	and the Lipsthitz condition  (ii) of Assumption \ref{ass:lip func}.

	Combining (\ref{Lyapunov function}), (\ref{second term of RH}) and (\ref{third term of RH}), it follows that
	{\small\begin{equation}\label{Lyapunov function 1}
		\begin{aligned}
		\mathbb{E}\left[R_{k}|\mathcal{F}_{k}\right]
		\leq& R_{k-1}-\frac{\alpha_k}{m}\left(f(\bar{x}_k)-f^*\right)+\frac{2L_0\alpha_{k}}{m\sigma}\sum_{j=1}^m\|z_{j,k-1}-\bar{z}_{k-1}\|+\frac{1}{m\sigma }\sum_{j=1}^m\|\bar{z}_{k-1}-z_{j,k-1}\|^2+\frac{L_0^2}{\sigma}\alpha_{k}^2.
		\end{aligned}
		\end{equation}}
	
	In what follows, we employ the supermartingale convergence theorem of Robbins and Siegmund (Lemma \ref{lem:Robbins-Siegmund} in Appendix \ref{sec:SA}) to study  the convergence of $R_k$.  For the consistency of the notations,  denote
	\begin{equation*}
	v_k\define R_k,  \quad a_k\define0, \quad \phi_k\define\frac{\alpha_k}{m}\left(f(\bar{x}_k)-f^*\right)
	\end{equation*}
	and
	\begin{equation*}
	b_k\define\frac{2\alpha_{k}}{m\sigma}\sum_{j=1}^mL_0\|z_{j,k-1}-\bar{z}_{k-1}\|\\
	+\frac{1}{m\sigma }\sum_{j=1}^m\|\bar{z}_{k-1}-z_{j,k-1}\|^2+\frac{L_0^2}{\sigma}\alpha_{k}^2.
	\end{equation*}
	Obviously, $v_k, a_k, b_k, \phi_k$ are nonnegative sequence  and adapted to $\mathcal{F}_k$.
	Note that
	\begin{equation*}
	\begin{aligned}
	\sum_{k=1}^\infty\al_k\|z_{j,k-1}-\bar{z}_{k-1}\|&=\al_1\|z_{j,0}-\bar{z}_{0}\|+\sum_{k=1}^\infty\al_{k+1}\|z_{j,k}-\bar{z}_{k}\|\\
	&\le \al_1\|z_{j,0}-\bar{z}_{0}\|+\sum_{k=1}^\infty\al_{k}\|z_{j,k}-\bar{z}_{k}\|<\infty ~\text{a.s.},
	\end{aligned}
	\end{equation*}
	where the inequality follows from  the  step-size $\alpha_k$ in nonincreasing by  Assumption \ref{ass:rho a} and the summability  follows from (\ref{consunsus sum}). Then by combining this with Assumption \ref{ass:rho a} and (\ref{consunsus sum 0}), we know that
	$
	\sum_{k=1}^\infty b_k<\infty,
	$
	and hence the conditions of Lemma \ref{lem:Robbins-Siegmund} hold.
	By applying the lemma,  we have that for any $x^*\in\mathcal{X}^*$, $R_k$ converges to a finite random variable $R_{\infty}$ almost surely and
	\begin{equation}\label{boundness}
	\sum_{k=1}^\infty\alpha_k\left(f(\bar{x}_k)-f^*\right)<\infty\quad \text{a.s.}
	\end{equation}
	By  \cite[Lemma 3.2 (a)]{zhou2020convergence},
	\begin{equation}\label{Fc bound}
	\|\bar{x}_k-x^*\|^2=\|\q(\bar{z}_{k-1})-x^*\|^2\le \dfrac{2}{\sigma}R_{k-1}
	\end{equation}
	and then $\{\bar{x}_k\}$ is bounded almost surely. In addition, according to (\ref{boundness}) and condition (i) of Assumption \ref{ass:rho a},
	\begin{equation*}
	\liminf_{k\rightarrow\infty}f(\bar{x}_k)-f^*=0\quad \text{a.s.}
	\end{equation*}
	Consider a subsequence $\{\bar{x}_{k_t}\}$ such that $\lim_{t\rightarrow\infty}f(\bar{x}_{k_t})=f^*$  and  denote $\check{x}$ as the limit point of $\{\bar{x}_{k_t}\}$. Since $f$ is continuous, we must have $f(\check{x})=f^*$, and hence $\check{x}\in\mathcal{X}^*$. 
	Fixing $x^*=\check{x}$ in the definition of $R_k$.  By Assumption \ref{ass:Fc}, we see that for any subsequence of $\{\bar{x}_{k_t}\}$ that converges to $\check{x}$,  the corresponding subsequence of $R_{k_t-1}$ must converges to $0$ almost surely, and thus $R_\infty$ equals to $0$ almost surely.
	Consequently, (\ref{Fc bound}) implies $\bar{x}_k\to\check{x}$  almost surely.  Note also that  for any $1\leq j\leq m$,
	\begin{align*}
	\|x_{j,k}-\check{x}\|\leq \|x_{j,k}-\bar{x}_k\|+\|\bar{x}_k-\check{x}\|
	\leq \frac{1}{\sigma} \|z_{j,k-1}-\bar{z}_{k-1}\|+\|\bar{x}_k-\check{x}\|,
	\end{align*}
	where the second inequality follows from  (\ref{consensus}). Then  $x_{j,k}\rightarrow \check{x}$   almost surely  as   $z_{j,k}\to \bar{z}_{k}$ and  $\bar{x}_k\to\check{x}$   almost surely.
	The proof is completed.
\end{proof}

A DDA algorithm is proposed by Duchi et al. \cite{duchi2011dual} where the convergence rate of gap between the functional value of local average and the optimal values have been established. In comparison, Theorem \ref{thm:convergence} establishes the almost sure convergence of the solutions
$x_{j,k}, j=1,\cdots, m$ and $\bar{x}_k$ generated by DDA algorithm \ref{alg:DDA}.

\subsection{Almost sure convergence rate}

Let $x^*$ be  the limit point of sequence $\{\bar x_{k}\}$ in Theorem  \ref{thm:convergence}.
In this subsection,  we  study the  convergence rate of  $\|\bar{x}_k-x^*\|$ to zero.
Hereafter,  we consider the case that the constraint set $\mathcal{X} $ in problem (\ref{problem model}) is defined by linear inequalities,
\begin{equation*}
\mathcal{X}=\{x\in\mathbb{R}^d: Bx-b\leq 0,~Cx-c\leq 0\}
\end{equation*}
and the  regularizer in (\ref{consensus dual averaging alrithm}) is $\psi(x)=\dfrac{1}{2}\|x\|^2$, where
$B\in\mathbb{R}^{d_1\times d}$, $b\in\mathbb{R}^{d_1}$,  $C\in\mathbb{R}^{d_2\times d}$ and $c\in\mathbb{R}^{d_2}$.
For simplicity, we assume that  $Bx^*-b= 0$,  $Cx^*-c< 0$, that is, $Bx-b\leq0$ is the active constraint on $x^*$ while the other is inactive, and
denote
\begin{equation}
\label{eq: Y}
\mathcal{Y}=\{x:Bx=0\}, \quad   U=(u_1,u_2,\cdots,u_d)\in\mathbb{R}^{d\times d},
\end{equation}
where $\mathcal{Y}$  is
a $r$-dimension subspace of $\mathbb{R}^d$, $u_1,u_2,\cdots,u_r$ and $u_{r+1},u_{r+1} ,\cdots,u_d$ are the standard orthogonal basis of $\mathcal{Y}$ and its orthogonal subspace respectively. Moreover,
the two auxiliary sequences defined in  (\ref{consensus dual averaging alrithm}) read as follows:
\begin{equation}\label{consensus dual averaging alrithm 1}
\bar{z}_k=\frac{1}{m}\sum_{j=1}^mz_{j,k},\quad
\bar{x}_{k+1}=\amin_{x\in\mathcal{X}}\{\left\langle-\bar{z}_k,x\right\rangle+\tfrac{1}{2}\|x\|^2\}.
\end{equation}

The following assumptions are needed.

\begin{ass}[\textbf{strengthened Assumption \ref{ass:lip func}}]\label{ass:str conv} (i) Assumption \ref{ass:lip func} holds. \\
	(ii) For any $1\leq j\leq m$, there exists a constant $L>0$ such that
	\begin{equation}\label{first order massage}
	\|\nabla f_j(x)-\nabla f_j(y)\|\leq L\|x-y\|,\forall x,y\in\mathcal{X}.
	\end{equation}
	There exist constants $c_0,\epsilon \in(0,\infty)$ such that for $x\in\mathcal{X}\cap\{x:\|x-x^*\|\leq\epsilon\}$,
	\begin{equation}\label{second order massage}
	\|\nabla f(x)-\nabla f(x^*)-\nabla^2 f(x^*)(x-x^*)\|\leq c_0\|x-x^*\|^2.
	\end{equation}
	(iii)There exists $\mu>0$ such that for any $x$ in the critical tangent cone $ \mathcal{T}_{\mathcal{X}}(x^*)$,
	\begin{equation}\label{res str conv}
	x^T\nabla^2 f(x^*)x\geq \mu\|x\|^2.
	\end{equation}
\end{ass}
Assumption \ref{ass:str conv}(iii) is the standard second-order sufficiency (or restricted strong convexity) condition \cite{wright1993identifiable}, which guarantees the uniqueness of minimizer of function $f(\cdot)$ over $\mathcal{X}$. Moreover, it implies that  \cite[Theorem 3.2(i)]{wright1993identifiable}: there exists $\epsilon'>0$ such that
\begin{equation}\label{restricted-strong-convexity}
\langle\nabla f(x),x-x^*\rangle\ge f(x)-f(x^*)\ge \epsilon' \min \left\{\|x-x^*\|^2,\|x-x^*\|\right\}  \quad   \forall x\in\mathcal{X}.
\end{equation}

\begin{ass}[\textbf{constraint qualification}]\label{ass:opt cond}\cite[Assumption B]{duchi2016asymptotic}  
	The vector $\nabla f(x^*)$ satisfies
	\begin{equation}\label{nondegeneracy condition}
	-\nabla f(x^*)\in\ri \mathcal{N}_{\mathcal{X}}(x^*),
	\end{equation}
	where $\ri \mathcal{N}_{\mathcal{X}}(x^*)$ is the relative interior of  normal cone  $\mathcal{N}_{\mathcal{X}}(x^*)$.
\end{ass}
The nondegeneracy condition (\ref{nondegeneracy condition})  is common  in manifold identification analysis \cite{lee2012manifold, duchi2016asymptotic}.
As we assumed that  $Bx^*=b$ and $Cx^*<c$,  the   norm cone in Assumption
\ref{ass:opt cond}   and critical tangent cone in Assumption \ref{ass:str conv} are
\begin{equation*}
\begin{aligned}
&\mathcal{N}_{\mathcal{X}}(x^*)=\{y:B^T\lambda=y, \lambda\in\mathbb{R}^{d_1}_+\},  \quad   \mathcal{T}_{\mathcal{X}}(x^*)=\{x:Bx=0\}.
\end{aligned}
\end{equation*}

We need stronger assumptions on weight matrix $A_k$ and step-size $\al_k$.

\begin{ass}[\textbf{stronger conditions on  weight matrix}]\label{ass:A&rho&a}
	(i)  $A_k, k=1,2, \cdots$ is doubly stochastic matrix with nonnegative components;
	(ii) $A_k,   k=1,2, \cdots $ is i.i.d. and the spectral norm $\rho$ of matrix $\mathbb{E}\left[A_k^T(I_m-\frac{\textbf{1}\textbf{1}^T}{m})A_k\right]$ satisfies $\rho<1$; (iii) Assumption \ref{ass:A rho} \rm{(iii)} holds.
\end{ass}
\begin{ass}[\textbf{stronger conditions on  step-size}]\label{ass:a} 
	The	step-size $\al_k=\tfrac{a}{k^\al}$ with $\al\in(\tfrac{2}{3},1),a>0$.
\end{ass}

The following lemma studies the active set identification of dual averaging algorithm \ref{alg:DDA}, which is an extension of \cite[Theorem 3]{duchi2016asymptotic} to distributed optimization setting.
\begin{lem}\label{lem:con iden}
	Suppose Assumptions \ref{ass:sample}, \ref{ass:str conv}-\ref{ass:a} hold. Then with probability one,
	there exists some (random) $K<\infty$ such that when $k\geq K$,
	\begin{equation*}
	B\bar{x}_k=b,\quad C\bar{x}_k<c.
	\end{equation*}
\end{lem}
The proof is presented in Appendix \ref{Proof:con iden}.

Define
\begin{equation}
\label{eq:pb}
P_B:=I_d-B^T(BB^T)^{\dagger} B
\end{equation}
as the projection operator onto subspace $ \mathcal{Y}$ (\ref{eq: Y}) and
\begin{equation}
\label{eq:h}
H\define \dfrac{1}{m}P_B\nabla^2f(x^*)P_B.
\end{equation}
Lemma \ref{lem:con iden} implies
\begin{equation*}
P_B(\bar{x}_k-x^*)=\bar{x}_k-x^*\quad \text{a.s.},
\end{equation*}
when $k$ is large enough. Therefore, we may study the convergence rate of $\|\bar{x}_k-x^*\|$ through   $\|P_B(\bar{x}_k-x^*)\|$. For easy of the notation, we denote
\begin{equation}
\label{eq:deltat}
\bigtriangleup_k:=P_B(\bar{x}_k-x^*)
\end{equation}
throughout the paper.

The following lemma provides the recursive formula of $\bigtriangleup_k$, whose proof is provided in Appendix \ref{Proof:linear recur}.
\begin{lem}\label{lem:linear recur}
	Suppose Assumptions \ref{ass:str conv}-\ref{ass:A&rho&a} 
	hold. Then
	\begin{equation}\label{recursion 0}
	\bigtriangleup_{k+1}=\bigtriangleup_k-\alpha_{k}H\bigtriangleup_k+\al_k\left(\zeta_k+\eta_k+s_k+\epsilon_k\right)
	\end{equation}
	or
	\begin{equation}\label{recursion 1}
	\bigtriangleup_{k+1}=\left[I_d-\alpha_{k}\left(H+D_k\right)\right]\bigtriangleup_k+\al_k\left(\eta_k+s_k+\epsilon_k\right),
	\end{equation}
	where
	\begin{equation}
	\left\{\begin{aligned}\label{errors}
	&\zeta_k=-\frac{1}{m}P_B\left[\nabla f(\bar{x}_k)-\nabla f(x^*)-\nabla^2 f(x^*)(\bar{x}_k-x^*)\right],\\
	&\eta_k=\dfrac{1}{m}\sum_{j=1}^mP_B\left[\nabla f_j(\bar{x}_k)-\nabla f_j(x_{j,k})\right],\\
	&\epsilon_k=\dfrac{1}{\al_k}P_BC^T(\mu_{k-1}-\mu_k)+\frac{1}{m}P_B\nabla^2f(x^*)(P_B-I_d)(\bar{x}_k-x^*),\\
	&s_k=-\frac{1}{m}\sum_{j=1}^{m} P_Bs_{j,k},\\
	&D_k=-\zeta_k\dfrac{\bigtriangleup_k^T}{\|\bigtriangleup_k\|^2}.
	\end{aligned}\right.
	\end{equation}	
\end{lem}

Lemma \ref{lem:linear recur} provides two kind of  recursive formulas of $\bigtriangleup_k$, where
(\ref{recursion 1}) will be used to analyse the almost sure convergence rate in Theorem \ref{thm:con rate} and  asymptotic normality of Algorithm \ref{alg:DDA}  in Theorem \ref{thm:asym norm}  and  (\ref{recursion 0}) will be
used to analysis the asymptotic efficiency of  Algorithm \ref{alg:DDA}   in Theorem \ref{asymptotic efficient}.

The following technical results will help us to study the rate of convergence of  $\|\bar{x}_k-x^*\|$ by focusing on the subspace  $\mathcal{Y}$  determined by the active constraints on the optimal solution $x^*$. 

\begin{lem}\label{lem:transition}
	Recall $\mathcal{Y}$, $U$ and $P_B$ have been defined in (\ref{eq: Y}) and (\ref{eq:pb}) respectively. Then
	\begin{itemize}
		\item [\rm{(i)}] $U^T:\mathcal{Y}\rightarrow\mathbb{R}^r\times\textbf{0}$ is a bijection, where $\textbf{0}^T=\underbrace{(0,0,\cdots,0)}_{d-r}$, and $'\times'$ is the Cartesian Product.
		\item [\rm{(ii)}]   For any  $y\in \mathcal{Y} $ and $ H\in\mathbb{R}^{d\times d}$,
		\begin{equation}\label{decomp}
		U^TP_{B}Hy=
		\left(
		\begin{array}{cc}
		G_1y_1\\
		\textbf{0}\\
		\end{array}\right),
		\end{equation}
		where $y_1\in\mathbb{R}^r$, $G_1$ is the $r$-order sequential principal minor of $U^THU$. Moreover, if there exists a constant $\mu>0$ such that
		\begin{equation*}
		y^THy\ge \mu \|y\|^2,\forall y\in \mathcal{Y},
		\end{equation*}
		then $G_1$ is a positive definite matrix.
	\end{itemize}
\end{lem}
The proof is presented in Appendix \ref{Proof:transition}.
\begin{thm}\label{thm:con rate}
	Suppose Assumptions \ref{ass:sample}, \ref{ass:str conv}-\ref{ass:a} hold with $p=2$ in Assumption \ref{ass:lip func}(ii). Then for any $\delta\in\left(0,1-1/(2\al)\right)$,
	\begin{equation}\label{rate}
	\|\bigtriangleup_{k}\|=o(\al_k^\delta) \quad  \text{a.s.}
	\end{equation}
\end{thm}
\begin{proof}
	We employ \cite[Lemma 3.1.1]{chen2006stochastic} (Lemma \ref{lem:rate} in Appendix \ref{sec:SA})  to prove (\ref{rate}). We reformulate the recursion $ \frac{\bigtriangleup_{k+1}}{\al_{k+1}^\delta}$  in the form of (\ref{linear reccursion}) in  Lemma \ref{lem:rate}  first.
	
	Dividing $\al_{k+1}^\delta$ on both sides of equation (\ref{recursion 1}), we have
	\begin{equation*}
	\begin{aligned}\label{recursion 3}
	\frac{\bigtriangleup_{k+1}}{\al_{k+1}^\delta}&=\left(\dfrac{\al_k}{\al_{k+1}}\right)^\delta\left[I_d-\alpha_{k}\left(H+D_k\right)\right]\dfrac{\bigtriangleup_k}{\al_k^\delta}+\al_k\left(\dfrac{\eta_k}{\al_{k+1}^\delta}+\dfrac{s_k}{\al_{k+1}^\delta}+\dfrac{\epsilon_k}{\al_{k+1}^\delta}\right)\\
	&=\left[I_d-\alpha_{k}\left(H+C_k\right)\right]\dfrac{\bigtriangleup_k}{\al_k^\delta}+\al_k\left(\dfrac{\eta_k}{\al_{k+1}^\delta}+\dfrac{s_k}{\al_{k+1}^\delta}+\dfrac{\epsilon_k}{\al_{k+1}^\delta}\right)\\
	&=\left[I_d-\alpha_{k}H_k\right]\dfrac{\bigtriangleup_k}{\al_k^\delta}+\al_k\left(\dfrac{\eta_k}{\al_{k+1}^\delta}+\dfrac{s_k}{\al_{k+1}^\delta}+\dfrac{\epsilon_k}{\al_{k+1}^\delta}\right),
	\end{aligned}
	\end{equation*}
	where
	\begin{equation*}
	C_k\define \dfrac{1}{\al_k}\left(1-\left(\dfrac{\al_k}{\al_{k+1}}\right)^\delta\right)I_d+\left(\left(\dfrac{\al_k}{\al_{k+1}}\right)^\delta-1\right)H+\left(\dfrac{\al_k}{\al_{k+1}}\right)^\delta D_k
	\end{equation*}
	and $H_k:=H+C_k$.	
	Note that by Assumption \ref{ass:a}: $\al_k=a/k^\al,\al\in(2/3,1)$, we obtain
	\begin{equation*}
	\begin{aligned}
	\left(\dfrac{\al_k}{\al_{k+1}}\right)^\delta\rightarrow 1,\quad
	\dfrac{1}{\al_k}\left(1-\left(\dfrac{\al_k}{\al_{k+1}}\right)^\delta\right)
	=\dfrac{k^\al}{a}\left(1-\left(1+\dfrac{1}{k}\right)^{\al\delta}\right)\rightarrow 0.
	\end{aligned}
	\end{equation*}
	Note also that
	\begin{equation*}
	\|D_k\|\le \dfrac{\|\zeta_k\|}{\|\bigtriangleup_k\|}\le \dfrac{c\|P_B\|\|\bar{x}_k-x^*\|^2}{\|\bigtriangleup_k\|}=\dfrac{c\|P_B\|\|\bar{x}_k-x^*\|^2}{\|\bar{x}_k-x^*\|}=c\|P_B\|\|\bar{x}_k-x^*\|\rightarrow 0,  \quad \text{a.s.},
	\end{equation*}
	where the second inequality follows from (\ref{second order massage}) and the fact $\bar{x}_k\rightarrow x^*$ almost surely. Then  $C_k\rightarrow 0$ almost surely which implies $H_k=H+C_k\rightarrow H$  almost surely.
	By definitions of $\bigtriangleup_k, H$ and $D_k$ in (\ref{eq:h}), (\ref{eq:deltat}) and (\ref{errors}) respectively,
	\begin{equation*}
	\bigtriangleup_k=P_B\bigtriangleup_k,~ H=P_BH,~ D_k=P_BD_k.
	\end{equation*}
	Then	
	\begin{equation*}
	\begin{aligned}\label{trans}
	H_k\dfrac{\bigtriangleup_k}{\al_k^\delta}&=(H+C_k)\dfrac{\bigtriangleup_k}{\al_k^\delta}\\
	&=\dfrac{1}{\al_k}\left(1-\left(\dfrac{\al_k}{\al_{k+1}}\right)^\delta\right)\dfrac{\bigtriangleup_k}{\al_k^\delta}+\left(\dfrac{\al_k}{\al_{k+1}}\right)^\delta H \dfrac{\bigtriangleup_k}{\al_k^\delta}+\left(\dfrac{\al_k}{\al_{k+1}}\right)^\delta D_k\dfrac{\bigtriangleup_k}{\al_k^\delta}\\
	&=\dfrac{1}{\al_k}\left(1-\left(\dfrac{\al_k}{\al_{k+1}}\right)^\delta\right)\dfrac{P_B\bigtriangleup_k}{\al_k^\delta}+\left(\dfrac{\al_k}{\al_{k+1}}\right)^\delta P_BH \dfrac{\bigtriangleup_k}{\al_k^\delta}+\left(\dfrac{\al_k}{\al_{k+1}}\right)^\delta P_BD_k\dfrac{\bigtriangleup_k}{\al_k^\delta}\\
	&=P_B\left(\dfrac{1}{\al_k}\left(1-\left(\dfrac{\al_k}{\al_{k+1}}\right)^\delta\right)I_d+\left(\dfrac{\al_k}{\al_{k+1}}\right)^\delta H+\left(\dfrac{\al_k}{\al_{k+1}}\right)^\delta D_k\right)\dfrac{\bigtriangleup_k}{\al_k^\delta}\\
	&=P_BH_k\dfrac{\bigtriangleup_k}{\al_k^\delta}.
	\end{aligned}
	\end{equation*}
	Subsequently,
	\begin{equation}\label{recursion 8} \frac{\bigtriangleup_{k+1}}{\al_{k+1}^\delta}=\left[I_d-\alpha_{k}P_BH_k\right]\dfrac{\bigtriangleup_k}{\al_k^\delta}+\al_k\left(\dfrac{\eta_k}{\al_{k+1}^\delta}+\dfrac{s_k}{\al_{k+1}^\delta}+\dfrac{\epsilon_k}{\al_{k+1}^\delta}\right).
	\end{equation}	
	Left  multiplying $U^T$ on both sides of equation (\ref{recursion 8}), we have
	\begin{equation*}
	\begin{aligned}
	U^T\frac{\bigtriangleup_{k+1}}{\al_{k+1}^\delta}=U^T\left[I_d-\alpha_{k}P_BH_k\right]\dfrac{\bigtriangleup_k}{\al_k^\delta}+\al_kU^T\left(\dfrac{\eta_k}{\al_{k+1}^\delta}+\dfrac{s_k}{\al_{k+1}^\delta}+\dfrac{\epsilon_k}{\al_{k+1}^\delta}\right).
	\end{aligned}
	\end{equation*}
	Since $\bigtriangleup_k,\eta_k,s_k,\epsilon_k\in\mathcal{Y}$,  Lemma \ref{lem:transition} implies
	\begin{equation}\label{re1}
	\left(
	\begin{array}{cc}
	\dfrac{\bigtriangleup_{k+1}^{'}}{\al_{k+1}^\delta}\\
	\textbf{0}\\
	\end{array}\right)=
	\left(
	\begin{array}{cc}
	\dfrac{\bigtriangleup_{k}^{'}}{\al_{k}^\delta}\\
	\textbf{0}\\
	\end{array}\right)
	-\al_k\left(
	\begin{array}{cc}
	G_k\dfrac{\bigtriangleup_{k}^{'}}{\al_{k}^\delta}\\
	\textbf{0}\\
	\end{array}\right)+\al_{k}\left[
	\left(
	\begin{array}{cc}
	\dfrac{\eta_k^{'}}{\al_{k+1}^\delta}\\
	\textbf{0}\\
	\end{array}\right)+
	\left(
	\begin{array}{cc}
	\dfrac{s_k^{'}}{\al_{k+1}^\delta}\\
	\textbf{0}\\
	\end{array}\right)+
	\left(
	\begin{array}{cc}
	\dfrac{\epsilon_k^{'}}{\al_{k+1}^\delta}\\
	\textbf{0}\\
	\end{array}\right)
	\right],
	\end{equation}
	where
	\begin{equation}
	\label{eq: Deltap}
	\Delta_k^{'}=(U^T)^{(r)}\Delta_k       \quad   \eta_k^{'}=(U^T)^{(r)}\eta_k,\quad s_k^{'}=(U^T)^{(r)}s_k,\quad\epsilon_k^{'}=(U^T)^{(r)}\epsilon_k,
	\end{equation}
	$(U^T)^{(r)}$ is a $r\times d$-matrix composed of first $r$ row vectors of $U^T$ and $G_k$ is the $r$-order sequential principal minor of $U^TH_kU$.
	Obviously,  we only need to focus on the linear recurrence
	\begin{equation}\label{recursion 4}
	\dfrac{\bigtriangleup_{k+1}^{'}}{\al_{k+1}^\delta}=(I_r-\al_{k}G_k)\dfrac{\bigtriangleup_{k}^{'}}{\al_{k}^\delta}+\al_{k}\left(\dfrac{\eta_k^{'}}{\al_{k+1}^\delta}+\dfrac{s_k^{'}}{\al_{k+1}^\delta}+\dfrac{\epsilon_k^{'}}{\al_{k+1}^\delta}\right).
	\end{equation}
	Denote
	\begin{equation*}
	y_k=\dfrac{\bigtriangleup_{k}^{'}}{\al_{k}^\delta},\quad F_k=-G_k, \quad e_k=\dfrac{s_k^{'}}{\al_{k+1}^\delta},\quad \nu_k=\dfrac{\eta_k^{'}}{\al_{k+1}^\delta}+\dfrac{\epsilon_k^{'}}{\al_{k+1}^\delta},
	\end{equation*}
	(\ref{recursion 4}) can be rewritten as
	\begin{equation*}
	y_{k+1}=y_k+\al_{k}F_ky_k+\al_{k}\left(e_k+\upsilon_k\right),
	\end{equation*}
	which is in the form (\ref{linear reccursion}) of Lemma \ref{lem:rate}.
	
	In what follows, we verify the conditions of  \cite[Lemma 3.1.1]{chen2006stochastic}.
	
	Firstly, we show that $F_k$ converges  to a stable matrix $F$. Note that
	$G_k$ is the $r$-order sequential principal minor of $U^TH_kU$ and $H_k\rightarrow H$ almost surely, $F_k$ converges to $-G$,
	where $G$ is
	the $r$-order sequential principal minor  of $U^THU$. By Lemma \ref{lem:transition}(ii) it follows from Assumption \ref{ass:str conv}(iii) that the $r$-order sequential principal minor  of $U^THU$ is  a  positive definite  matrix, which implies the
	stability of the limit of $\{F_k\}$.

	Next, we show $\nu_k\rightarrow 0$ almost surely, where it is sufficient to prove
	\begin{equation*}
	\dfrac{\epsilon_k^{'}}{\al_{k+1}^\delta}\rightarrow 0,  \qquad \dfrac{\eta_k^{'}}{\al_{k+1}^\delta}\rightarrow 0.
	\end{equation*}
	On the one hand, recall the definition (\ref{errors})
	\begin{equation*}
	\epsilon_k=\dfrac{1}{\al_k}P_BC^T(\mu_{k-1}-\mu_k)+\frac{1}{m}P_B\nabla^2f(x^*)(P_B-I_d)(\bar{x}_k-x^*).
	\end{equation*}
	By  Lemma \ref{lem:con iden},
	$\epsilon_k=0$  almost surely when $k$ is large enough  as $\mu_{k}=\mu_{k+1}=0$  and $(P_B-I_d)(\bar{x}_k-x^*)=0$ when $k\geq K$, where $K<\infty$ is specified in Lemma \ref{lem:con iden}.
	Then
	$\dfrac{\epsilon_k^{'}}{\al_{k+1}^\delta}=\dfrac{(U^T)^{(r)}\epsilon_k}{\al_{k+1}^\delta}=0$  almost surely. On the other hand, note that
	\begin{equation*}
	\begin{aligned}\label{sum noise 0}
	&\mathbb{E}[\|\eta_k^{'}\|^2]=\mathbb{E}[\|(U^T)^{(r)}\eta_k\|^2]=\mathbb{E}\Big[\Big\|(U^T)^{(r)}{\tfrac{1}{m}\sum_{j=1}^mP_B(\nabla f_j(x_{j,k})-\nabla f_j(\bar{x}_k))}\Big\|^2\Big]\\
	\le&\frac{\left\|(U^T)^{(r)}\right\|^2\|P_B\|^2}{m}\sum_{j=1}^m\mathbb{E}[\|\nabla f_j(x_{j,k})-\nabla f_j(\bar{x}_k)\|^2]
	\le\frac{\left\|(U^T)^{(r)}\right\|^2\|P_B\|^2L^2}{m}\sum_{j=1}^m\mathbb{E}[\|x_{j,k}-\bar{x}_k\|^2],
	\end{aligned}
	\end{equation*}
	where the  last inequality follows from  the Lipschitz continuity of $\nabla f_j(\cdot)$. By using the fact
	\begin{equation}\label{est:a_k-delta}
	\left(\dfrac{\al_k}{\al_{k+1}}\right)^\delta=\left(1+\dfrac{1}{k}\right)^{\al\delta}\le2^{\al\delta},
	\end{equation}
	and denoting $c_m={4^{\delta}\left\|(U^T)^{(r)}\right\|^2\|P_B\|^2 L^2}/{m}$, we have
	\begin{equation}\label{sum noise 1}
	\begin{aligned}
	&\sum_{k=1}^\infty\mathbb{E}\left[\left\|\dfrac{\eta_k^{'}}{\al_{k+1}^\delta}\right\|^2\right]\le
	c_m\sum_{k=1}^\infty\sum_{j=1}^m\mathbb{E}\left[\left\|\dfrac{x_{j,k}-\bar{x}_k}{\al_{k}^\delta}\right\|^2\right]\\
	\le& c_m\sum_{k=1}^\infty\sum_{j=1}^m\mathbb{E}\left[\left\|\dfrac{z_{j,k-1}-\bar{z}_{k-1}}{\al_{k}^\delta}\right\|^2\right] =c_m\sum_{k=1}^\infty\sum_{j=1}^m\mathbb{E}\left[\left\|\dfrac{z_{j,k}-\bar{z}_{k}}{\al_{k+1}^\delta}\right\|^2\right]+c_m\mathbb{E}\left[\left\|\dfrac{z_{j,0}-z_0}{\al_{1}^\delta}\right\|^2\right]\\
	\le& c'_m\sum_{k=2}^\infty\dfrac{\al_k^2}{\al_{k}^{2\delta}}+c_m\mathbb{E}\left[\left\|\dfrac{z_{j,0}-z_0}{\al_{1}^\delta}\right\|^2\right]
	=c'_m\sum_{k=2}^\infty\dfrac{a^{2-2\delta}}{k^{2\al(1-\delta)}}+c_m\mathbb{E}\left[\left\|\dfrac{z_{j,0}-z_0}{\al_{1}^\delta}\right\|^2\right]<\infty,
	\end{aligned}
	\end{equation}
	where the second inequality follows from (\ref{consensus}), the third one from (\ref{consensus rate 1}), the last one from $2\al(1-\delta)>1$ by the definition,	and $c'_m$ is a constant.	
	Then by monotone convergence theorem,
	\begin{equation*}
	\sum_{k=0}^\infty\left\|\dfrac{\eta_k^{'}}{\al_{k+1}^\delta}\right\|^2<\infty\quad\text{a.s.},
	\end{equation*}
	which implies $\dfrac{\eta_k^{'}}{\al_{k+1}^\delta}\rightarrow 0$ almost surely. Therefore,  $\nu_k\rightarrow 0$ almost surely.
	
	We are left to verify
	\begin{equation}\label{sum noise}
	\sum_{k=1}^\infty\al_ke_k
	<\infty\quad \text{a.s.}
	\end{equation}
	Denote
	\begin{equation*}
	e_k^{'}=\left(\dfrac{\al_k}{\al_{k+1}}\right)^\delta(U^T)^{(r)}s_k.
	\end{equation*}
	Obviously,  $\{e_k^{'},\mathcal{F}_{k+1}\}$  is a  martingale difference sequence since
	$\{s_k,\mathcal{F}_{k+1}\}$ is a martingale difference sequence.
	Then
	\begin{equation*}
	\begin{aligned} &\sup_{k}\mathbb{E}[\|e_k^{'}\|^2|\mathcal{F}_{k}]=\sup_{k}\mathbb{E}[\|\left(\dfrac{\al_k}{\al_{k+1}}\right)^\delta(U^T)^{(r)}s_k\|^2|\mathcal{F}_{k}]\\
	&\le 4^\delta\left\|(U^T)^{(r)}\right\|^2\sup_{k}\mathbb{E}[\|s_k\|^2|\mathcal{F}_{k}]
	=4^\delta\left\|(U^T)^{(r)}\right\|^2\sup_{k}\mathbb{E}\left[\left\|\frac{1}{m}\sum_{j=1}^{m} P_Bs_{j,k}\right\|^2\bigg|\mathcal{F}_{k}\right]\\
	&\le 4^\delta\left\|(U^T)^{(r)}\right\|^2\|P_B\|^2\sup_{k}\frac{1}{m}\sum_{j=1}^{m}\mathbb{E}\left[\left\|s_{j,k}\right\|^2\bigg|\mathcal{F}_{k}\right]
	\le4^\delta\left\|(U^T)^{(r)}\right\|^2 \|P_B\|^24L_0^2<\infty,
	\end{aligned}
	\end{equation*}
	where the first inequality follows from (\ref{est:a_k-delta}), 
	the second one from the convexity of $\|\cdot\|^2$, and the last one from
	Assumptions \ref{ass:sample} and \ref{ass:str conv},  which imply
	\begin{equation*}
	\mathbb{E}\left[\|s_{j,k}\|^2\big|\mathcal{F}_k\right]=\mathbb{E}\left[\|\nabla f_j(x_{j,k})-\nabla F_j(x_{j,k};\xi_{j,k})\|^2\big|\mathcal{F}_k\right]\leq 4L_0^2,
	\end{equation*}
	and $L_0^2$ is defined as in (\ref{eq: L0}).
	Since
	\begin{equation*}
	\sum_{k=1}^\infty\al_k^{2(1-\delta)}=\sum_{k=1}^\infty\frac{a^{2(1-\delta)}}{k^{2(1-\delta)\al}}<\infty,
	\end{equation*}
	then by the convergence theorem
	for martingale difference sequences \cite[Appendix B.6, Theorem B 6.1]{chen2006stochastic},
	\begin{equation*}
	\sum_{k=1}^\infty\al_ke_k=\sum_{k=1}^\infty\al_k^{1-\delta}e_k^{'}<\infty.
	\end{equation*}
	Then employing \cite[Lemma 3.1.1]{chen2006stochastic} yields $y_k=\dfrac{\bigtriangleup_{k}^{'}}{\al_{k}^\delta}\to 0$ almost surely. By the definition of   $\Delta_{k}'$ in (\ref{eq: Deltap}), we conclude that
	$\|\bigtriangleup_{k+1}\|=o(\al_k^\delta)$ almost surely. The proof is completed.	
\end{proof}

The almost sure convergence rate in terms of the step-size of stochastic approximation algorithms for root-finding problems have been well studied, see \cite{chen2006stochastic, Ljung1992Stochastic}.
More recently, \cite{xu2012consensus,Tang2018Convergence}  study the convergence rate of consensus problem when stochastic approximation method is used.
To the best of our knowledge, Theorem \ref{thm:con rate} seems to be the first result  on  almost convergence rate of  stochastic approximation method for  distributed  constrained  stochastic optimization problems. As we will see, this result is useful for establishing asymptomatic normality of the DDA algorithm.

\section{Asymptotic normality and asymptotic efficiency}
Asymptotic normality and asymptotic efficiency of stochastic algorithms can be traced back to the works on 1950s \cite{chung1954stochastic,sacks1958asymptotic}. More recently,   \cite{lei2018asymptotic, bianchi2013performance} study  the  asymptotic normality and asymptotic efficiency of  stochastic algorithms for  distributed unconstrained optimization problem.  In this section, we focus on these asymptotic properties of Algorithm \ref{alg:DDA} for distributed constrained optimization problems.

We first present the asymptotic normality of Algorithm \ref{alg:DDA}.

\begin{thm}\label{thm:asym norm}
	Suppose Assumptions \ref{ass:sample}, \ref{ass:str conv}-\ref{ass:a} hold with $p>2$ in Assumption \ref{ass:lip func}(ii). Let $x^*$ be  the limit point of sequence $\{\bar x_{k}\}$.  The covariance matrix 	mapping
	$\sum_{j=1}^m\Cov(\nabla F_j(\cdot; \xi_j))$ is continuous  at point  $x^*$. Then for any $1\leq j\leq m$,
	\begin{equation}\label{limit d}
	\dfrac{x_{j,k}-x^*}{\sqrt{\al_{k}}}\overset{d}{\longrightarrow}N(0,\Sigma),
	\end{equation}
	where
	\begin{equation}\label{def:Sigma}
	\Sigma=U\left(
	\begin{array}{cc}
	\Sigma_1&\textbf{0}\\
	\textbf{0}&\textbf{0}\\
	\end{array}\right)U^T,
	\end{equation}
	\begin{equation}\label{def:Sigma_1}
	\Sigma_1=\int_{0}^{\infty}e^{(-G)t}(U^T)^{(r)}P_B\bar{\Sigma}P_B(U^T)^{(r)T}e^{(-G^T)t}\d t,\quad\bar{\Sigma}=\dfrac{1}{m^2}\sum_{j=1}^m\Cov(\nabla F_j(x^*;\xi_j)),
	\end{equation}
	$(U^T)^{(r)}\in\mathbb{R}^{r\times d}$ is composed by first $r$ row vectors of $U^T$, $G$ is the $r$-order sequential principal minor of $U^THU$ and $H$ is defined
	as in (\ref{eq:h}).

\end{thm}
\begin{proof}
	We employ  \cite[Theorem 3.3.1]{chen2006stochastic} (Lemma \ref{lem:asym norm} in Appendix \ref{sec:SA})  to prove (\ref{limit d}). By definition (\ref{eq:h}),
	\begin{equation*}
	H=\dfrac{1}{m}P_B\nabla^2f(x^*)P_B=\dfrac{1}{m}P_B^2\nabla^2f(x^*)P_B=P_BH.
	\end{equation*}
	Then
	(\ref{recursion 0}) can be reformulated  as
	\begin{equation}\label{recursion 7}
	\bigtriangleup_{k+1}=\left[I_d-\alpha_{k}P_BH\right]\bigtriangleup_k+\al_k\left(\zeta_k+\eta_k+s_k+\epsilon_k\right).
	\end{equation}
	Left  multiplying $U^T$ on both side of  (\ref{recursion 7}),
	Lemma \ref{lem:transition} implies
	\begin{equation*}
	\left(
	\begin{array}{cc}
	\bigtriangleup_{k+1}^{'}\\
	\textbf{0}\\
	\end{array}\right)=
	\left(
	\begin{array}{cc}
	\bigtriangleup_{k}^{'}\\
	\textbf{0}\\
	\end{array}\right)
	-\al_k\left(
	\begin{array}{cc}
	G\bigtriangleup_{k}^{'}\\
	\textbf{0}\\
	\end{array}\right)+\al_{k}\left[
	\left(
	\begin{array}{cc}
	\zeta_k^{'}\\
	\textbf{0}\\
	\end{array}\right)+
	\left(
	\begin{array}{cc}
	\eta_k^{'}\\
	\textbf{0}\\
	\end{array}\right)+
	\left(
	\begin{array}{cc}
	s_k^{'}\\
	\textbf{0}\\
	\end{array}\right)+
	\left(
	\begin{array}{cc}
	\epsilon_k^{'}\\
	\textbf{0}\\
	\end{array}\right)
	\right],
	\end{equation*}
	where
	\begin{equation*}
	\bigtriangleup_{k}^{'}=(U^T)^{(r)}\bigtriangleup_{k},\quad\zeta_k^{'}=(U^T)^{(r)}\zeta_k,\quad\eta_k^{'}=(U^T)^{(r)}\eta_k,\quad s_k^{'}=(U^T)^{(r)}s_k,\quad\epsilon_k^{'}=(U^T)^{(r)}\epsilon_k,
	\end{equation*}
	$(U^T)^{(r)}$ is a $r\times d$-matrix composed of first $r$ row vectors of $U^T$, and $G$ is the $r$-order sequential principal minor of $U^THU$. Define
	\begin{equation}\label{recursion 5}
	\bigtriangleup_{k+1}^{''}: =(I_r-\al_{k}G)\bigtriangleup_{k}^{''}+\al_{k}\left(\zeta_k^{'}+s_k^{'}+\epsilon_k^{'}\right),
	\end{equation}
	where the initial $\bigtriangleup_{0}^{''}\in\mathbb{R}^r$ is arbitrary.
	Consequently,
	\begin{equation}\label{recursion 6}
	\begin{aligned} \dfrac{\bigtriangleup_{k+1}^{'}-\bigtriangleup_{k+1}^{''}}{\sqrt{\al_{k+1}}}&=\sqrt{\dfrac{\al_{k}}{\al_{k+1}}}(I_r-\al_{k}G)\dfrac{\bigtriangleup_{k}^{'}-\bigtriangleup_{k}^{''}}{\sqrt{\al_{k}}}+\dfrac{\al_k}{\sqrt{\al_{k+1}}}\eta_k^{'}\\
	&=(I_r-\al_{k}G_k^{'})\dfrac{\bigtriangleup_{k}^{'}-\bigtriangleup_{k}^{''}}{\sqrt{\al_{k}}}+\dfrac{\al_k}{\sqrt{\al_{k+1}}}\eta_k^{'},
	\end{aligned}
	\end{equation}
	where
	\begin{equation*}
	G_k^{'}:=\left(\dfrac{1}{\al_{k}}-\dfrac{1}{\sqrt{\al_k\al_{k+1}}}\right)I_r+\sqrt{\dfrac{\al_{k}}{\al_{k+1}}}G.
	\end{equation*}
	For $k\geq t$, denote
	\begin{equation*}
	\Psi_t^k\define \left(I_r-\al_{k}G_k^{'}\right)\cdots\left(I_r-\al_{t}G_t^{'}\right),~\Psi_{t+1}^{t}=I_r.
	\end{equation*}
	Recursively,  we  can reformulate (\ref{recursion 6})  as
	\begin{equation}\label{re2}
	\dfrac{\bigtriangleup_{k+1}^{'}-\bigtriangleup_{k+1}^{''}}{\sqrt{\al_{k+1}}}=\Psi_1^k\dfrac{\bigtriangleup_{1}^{'}-\bigtriangleup_{1}^{''}}{\sqrt{\al_{1}}}+\sum_{t=1}^k\Psi_{t+1}^k\dfrac{\al_t}{\sqrt{\al_{t+1}}}\eta_t^{'}.
	\end{equation}
	
	By the Assumption \ref{ass:a} and the definition of $G_k^{'}$, it is easy to get that $\lim_{k\rightarrow\infty}G_k^{'}=G.$
	Since $-G$ is stable, by \cite[Inequality (3.1.8) in Lemma 3.1.1]{chen2006stochastic}, there exist constants $b_1,b_2>0$ such that
	\begin{equation}\label{mat bound}
	\left\|\Psi_t^k\right\|\le b_1\exp(-b_2\sum_{l=t}^k\al_l), \quad  \forall k\ge t.
	\end{equation}
	Obviously, (\ref{mat bound}) implies the first term on the right-hand side of (\ref{re2}) tends to zero  almost surely. Next, we show that the second term on the right-hand side of (\ref{re2}) tends to $0$ in probability.

	Note that
	\begin{equation*}
	\dfrac{\al_t}{\sqrt{\al_{t+1}}}=
	\left(1+\dfrac{1}{t}\right)^{\dfrac{\al}{2}}\sqrt{\al_{t}}\le\sqrt{a}\left(\dfrac{3}{2}\right)^{\dfrac{\al}{2}}\sqrt{\al_{t}},
	\end{equation*}
	and
	\begin{equation*}
	\begin{aligned}
	\mathbb{E}\left[\left\|\eta_t^{'}\right\|\right]&=\mathbb{E}\left[\left\|(U^T)^{(r)}\dfrac{1}{m}\sum_{j=1}^mP_B(\nabla f_j(x_{j,t})-\nabla f_j(\bar{x}_t))\right\|\right]\\
	&\le \dfrac{\|(U^T)^{(r)}\|\|P_B\|{L}}{m}\sum_{j=1}^m\mathbb{E}\left[\left\|x_{j,t}-\bar{x}_t\right\|\right]\\
	&\le \dfrac{\|(U^T)^{(r)}\|\|P_B\|{L}}{m}\sum_{j=1}^m\mathbb{E}\left[\left\|z_{j,t-1}-\bar{z}_{t-1}\right\|\right]\le b_3\alpha_{t-1},
	\end{aligned}
	\end{equation*}
	where the second inequality follows from (\ref{consensus}) and the last one from (\ref{consensus rate 1}). We obtain the estimate
	\begin{equation}\label{mean con 1}
	\begin{aligned}
	\mathbb{E}\left[\left\|\sum_{t=1}^k\Psi_{t+1}^k\dfrac{\al_t}{\sqrt{\al_{t+1}}}\eta_t^{'}\right\|\right]&\le \sum_{t=1}^k\left\|\Psi_{t+1}^k\right\|\dfrac{\al_t}{\sqrt{\al_{t+1}}}\mathbb{E}\left[\left\|\eta_t^{'}\right\|\right]\\
	&\le b_3\sqrt{a}\left(\dfrac{3}{2}\right)^{\dfrac{\al}{2}}\sum_{t=1}^k\left\|\Psi_{t+1}^k\right\|\sqrt{\al_{t}}\al_{t-1}\\
	&\le b_3\sqrt{a}\left(\dfrac{3}{2}\right)^{\dfrac{\al}{2}}\sum_{t=1}^k\al_t\|\Psi_{t+1}^k\|^{\frac{1}{2}}\|\Psi_{t+1}^k\|^{\frac{1}{2}}o(1),
	\end{aligned}
	\end{equation}
	where $o(1)=\frac{\al_{t-1}}{\sqrt{\al_{t}}}\rightarrow0$ as $t\rightarrow\infty$.
	By (\ref{mat bound}) and \cite[Inequality (3.3.6) in Lemma 3.3.2]{chen2006stochastic}, the term on right-hand side of the last inequality of (\ref{mean con 1}) tends to $0$,  which implies  the second term on the right hand of (\ref{re2}) tends to 0 in probability.
	Therefore (\ref{re2}) tends to 0 in probability, which implies $\dfrac{\bigtriangleup_{k+1}^{'}}{\sqrt{\al_{k+1}}}$ and $\dfrac{\bigtriangleup_{k+1}^{''}}{\sqrt{\al_{k+1}}}$ have same limit distribution.

	Next, we focus on investigating the limit distribution of $\dfrac{\bigtriangleup_{k+1}^{''}}{\sqrt{\al_{k+1}}}$.	
	Denote
	\begin{equation*}
	y_k=\bigtriangleup_{k}^{''},\quad F_k=-G, \quad e_k=s_k^{'},\quad \nu_k=\zeta_k^{'}+\epsilon_k^{'}, \; \alpha_k=\alpha_k.
	\end{equation*}
	(\ref{recursion 5}) can be rewritten as
	\begin{equation*}
	y_{k+1}=y_k+\al_{k}F_ky_k+\al_{k}\left(e_k+\upsilon_k\right),
	\end{equation*}
	which is in the form (\ref{linear reccursion}) of Lemma \ref{lem:asym norm}. Then we may employ
	\cite[Theorem 3.3.1]{chen2006stochastic} to study the limit distribution of  $\dfrac{\bigtriangleup_{k+1}^{''}}{\sqrt{\al_{k+1}}}$. In what follows, we verify the conditions of Lemma \ref{lem:asym norm} in Appendix \ref{sec:SA}. By Assumption \ref{ass:a},
	\begin{equation*}
	\al_{k+1}^{-1}-\al_{k}^{-1}\rightarrow 0,
	\end{equation*}
	which implies condition (i) of Lemma \ref{lem:asym norm}.
	Note also that $F_k=-G$  is stable, condition (ii) of  Lemma \ref{lem:asym norm} holds. Then we focus on condition (iii) of  Lemma \ref{lem:asym norm}. On the one hand, we may show that $\nu_k=\epsilon_k^{'}+\zeta_k^{'}=o(\sqrt{\al_k})$ almost surely. In fact, for $\epsilon_k^{'}$, recall the definition of $\epsilon_k$ in (\ref{errors}).
	By  Lemma \ref{lem:con iden}, $\epsilon_k=0$ and then $\epsilon_k^{'}=(U^T)^{(r)}\epsilon_k=0$ almost surely when $k$ is large enough.
	For $\zeta_k^{'}$, when $k$ is large enough
	\begin{equation*}
	\begin{aligned}
	\|\zeta_k^{'}\|&=\left\|-\frac{1}{m}(U^T)^{(r)}P_B\left[\nabla f(\bar{x}_k)-\nabla f(x^*)-\nabla^2 f(x^*)(\bar{x}_k-x^*)\right]\right\|\\
	&\le \frac{1}{m}\left\|(U^T)^{(r)}P_B\right\|\left\|\nabla f(\bar{x}_k)-\nabla f(x^*)-\nabla^2 f(x^*)(\bar{x}_k-x^*)\right\|\\
	&\le \frac{c_0}{m}\left\|(U^T)^{(r)}P_B\right\|\left\|\bar{x}_k-x^*\right\|^2\\
	&=\frac{c_0}{m}\left\|(U^T)^{(r)}P_B\right\|\left\|\bigtriangleup_{k}\right\|^2=\frac{c_0}{m}\left\|(U^T)^{(r)}P_B\right\|o\left(\al_k^{2\delta}\right)\quad \text{a.s.},
	\end{aligned}
	\end{equation*}
	where the second inequality follows from (\ref{second order massage}) in Assumption \ref{ass:str conv} and  $\bar{x}_k\rightarrow x^*~$ almost surely, the second equality follows from Lemma \ref{lem:con iden} and the last equality follows from Theorem \ref{thm:con rate}. Therefore,
	\begin{equation*}
	\nu_k=\epsilon_k^{'}+\zeta_k^{'}=o\left(\al_k^{2\delta}\right)\le o(\sqrt{\al_k})\quad \text{a.s.}
	\end{equation*}
	as $\delta\in [1/4,1-{1}/{(2\al)})$.
	
	On the other hand, we verify (\ref{c1})-(\ref{c3}) of Lemma \ref{lem:asym norm} for the term $e_k=s_k^{'}$. By definition $s_k^{'}=(U^T)^{(r)}s_k$, it is easy to verify that
	\begin{equation}\label{sup}
	\mathbb{E}\left[s_k^{'}|\mathcal{F}_{k}\right]=0,~\sup_k\mathbb{E}\left[\|s_k^{'}\|^2|\mathcal{F}_{k}\right]\le\left\|(U^T)^{(r)}\right\|^2\sup_{k}\mathbb{E}[\|s_k\|^2|\mathcal{F}_{k}]\le\|P_B\|^2\left\|(U^T)^{(r)}\right\|^2 4L_0^2,
	\end{equation}
	and hence (\ref{c1}) of Lemma \ref{lem:asym norm} holds. By the definition of $s_k$,
	\begin{equation}
	\begin{aligned}\label{noise cov}
	&\mathbb{E}\left[s_ks_k^T\big|\mathcal{F}_{k}\right]
	=\mathbb{E}\left[\left(\dfrac{1}{m}P_B\sum_{j=1}^ms_{j,k}\right)\left(\dfrac{1}{m}P_B\sum_{j=1}^ms_{j,k}\right)^T\bigg|\mathcal{F}_{k}\right]\\
	=&\dfrac{1}{m^2}P_B\left(\sum_{1\leq i, \;j\leq m}\mathbb{E}\left[s_{j,k}(s_{j,k})^T\big|\mathcal{F}_{k}\right]\right)P_B\\
	=&\dfrac{1}{m^2}P_B\left(\sum_{1\leq i, \;j\leq m}\mathbb{E}\left[\left[\nabla F_i(x_{i, k};\xi_{i,k})-\nabla f_i(x_{i,k})\right]\left[\nabla F_j(x_{j,k};\xi_{j,k})-\nabla f_j(x_{j,k})\right]^T\big|\mathcal{F}_{k}\right]\right)P_B\\
	=&\dfrac{1}{m^2}P_B\left(\sum_{j=1}^m\mathbb{E}\left[\left[\nabla F_j(x_{j,k};\xi_{j, k})-\nabla f_j(x_{j,k})\right]\left[\nabla F_j(x_{j,k};\xi_{j,k})-\nabla f_j(x_{j,k})\right]^T\big|\mathcal{F}_{k}\right]\right)P_B\\
	=&\dfrac{1}{m^2}P_B\left(\sum_{j=1}^m\Cov(\nabla F_j(x;\xi_j))|_{x=x_{j,k}}\right)P_B,
	\end{aligned}
	\end{equation}
	where the fourth equality follows from that $\xi_{j,k}$ is independent of $\xi_{i,k}$ for any $i\neq j$, 
	$\Cov(\nabla F_j(x;\xi_j))|_{x=x_{j,k}}$
	means the value of covariance matrix $\Cov(\nabla F_j(x;\xi_j))$
	with respect to $\xi_j$ taking at the point $x=x_{j,k}$. Since
	for any $1\leq j\leq m$,
	$x_{j,k}\rightarrow x^*$ almost surely and the $\sum_{j=1}^m\Cov(\nabla F_j(\cdot;\xi_j))$ is continuous at  point $x^*$,
	\begin{equation}
	\begin{aligned}\label{con 1}
	\lim_{k\rightarrow\infty}\mathbb{E}\left[s_ks_k^T\big|\mathcal{F}_{k}\right]&=\lim_{k\rightarrow\infty}\dfrac{1}{m^2}P_B\left(\sum_{j=1}^m\Cov(\nabla F_j(x;\xi_j))|_{x=x_{j,k}}\right)P_B\\
	&=\dfrac{1}{m^2}P_B\left(\sum_{j=1}^m\Cov(\nabla F_j(x^*;\xi_j))\right)P_B\quad \text{a.s.}
	\end{aligned}
	\end{equation}
	Note that $\sup_{k}\mathbb{E}[\|s_k\|^2\big|\mathcal{F}_{k}]\le \|P_B\|^24L_0^2$.
	Then according to dominated convergence theorem,
	\begin{equation}\label{con 2} \lim_{k\rightarrow\infty}\mathbb{E}\left[s_ks_k^T\right]=\mathbb{E}\left[\lim_{k\rightarrow\infty}\mathbb{E}\left[s_ks_k^T\big|\mathcal{F}_{k}\right]\right]=\dfrac{1}{m^2}P_B\left(\sum_{j=1}^m\Cov(\nabla F_j(x^*;\xi_j))\right)P_B.
	\end{equation}
	Moreover, by the definition of $s_k^{'}$ and (\ref{con 1})-(\ref{con 2}), we have
	\begin{equation*}
	\begin{aligned}
	\lim_{k\rightarrow\infty}\mathbb{E}\left[s_k^{'}(s_k^{'})^T\big|\mathcal{F}_{k-1}\right]
	&=\lim_{k\rightarrow\infty}(U^T)^{(r)}\mathbb{E}\left[s_ks_k^T\big|\mathcal{F}_{k-1}\right](U^T)^{(r)T}\\
	&=\dfrac{1}{m^2}(U^T)^{(r)}P_B\left(\sum_{j=1}^m\Cov(\nabla F_j(x^*;\xi_j))\right)P_B(U^T)^{(r)T}\quad \text{a.s.},
	\end{aligned}
	\end{equation*}
	\begin{equation*}
	\lim_{k\rightarrow\infty}\mathbb{E}\left[s_k^{'}(s_k^{'})^T\right]
	=\dfrac{1}{m^2}(U^T)^{(r)}P_B\left(\sum_{j=1}^m\Cov(\nabla F_j(x^*;\xi_j))\right)P_B(U^T)^{(r)T},
	\end{equation*}
	which shows (\ref{c2}) in Lemma \ref{lem:asym norm}. 

	By Chebyshev's inequality and (\ref{sup})
	\begin{equation*}
	\P(\|s_k^{'}\|> N)\le \dfrac{ \mathbb{E}[\|s_k^{'}\|^2]}{N^2}\le\dfrac{\|P_B\|^2\left\|(U^T)^{(r)}\right\|^2 4L_0^2}{N^2}.
	\end{equation*}
	Furthermore, for $p>2$ given in Assumption \ref{ass:lip func}(ii) and $q>0$ such that $2/p+1/q=1$,
	\begin{equation*}
	\begin{aligned}
	\mathbb{E}\left[\|s_k^{'}\|^21_{\{\|s_k^{'}\|>N\}}\right]&\le\left(\mathbb{E}\left[\|s_k^{'}\|^{2(p/2)}\right]\right)^{2/p}\left(\mathbb{E}\left[1^q_{\{\|s_k^{'}\|>N\}}\right]\right)^{1/q}\\
	&=\left(\mathbb{E}\left[\left\|(U^T)^{(r)}P_B\frac{1}{m}\sum_{j=1}^{m} s_{j,k}\right\|^{p}\right]\right)^{2/p}\left(\P(\|s_k^{'}\|> N)\right)^{1/q}\\
	&\le \|(U^T)^{(r)}\|^2\|P_B\|^2\left(\frac{1}{m}\sum_{j=1}^{m}\mathbb{E}\left[\left\| s_{j,k}\right\|^{p}\right]\right)^{2/p}\left(\P(\|s_k^{'}\|> N)\right)^{1/q}\\
	&\le \|(U^T)^{(r)}\|^2\left\|P_B\right\|^2 4(L_0^p)^{2/p}\left(\P(\|s_k^{'}\|> N)\right)^{1/q}\\
	&\le\|(U^T)^{(r)}\|^{2+2/q}\left\|P_B\right\|^{2+2/q}\dfrac{16(L_0^p)^{2/p}(L_0^2)^{1/q}}{N^{2/q}},
	\end{aligned}
	\end{equation*}
	where   the first inequality follows from the $\rm{H\ddot{o}lder}$ inequality, the second inequality follows from the convexity of $\|\cdot\|^p$ and the third inequality  follows from (\ref{observe noise bound}). Then we have
	\begin{equation}\label{con 3}
	\lim_{N\rightarrow\infty}\sup_{k}\mathbb{E}\left[\|s_k^{'}\|^21_{\{\|s_k^{'}\|>N\}}\right]\le\lim_{N\rightarrow\infty}\|(U^T)^{(r)}\|^{2+2/q}\left\|P_B\right\|^{2+2/q}\dfrac{16(L_0^p)^{2/p}(L_0^2)^{1/q}}{N^{2/q}}=0,
	\end{equation}
	which verifies (\ref{c3}) in  Lemma \ref{lem:asym norm}. Therefore,  by  Lemma \ref{lem:asym norm}, 
	\begin{equation}
	\dfrac{\bigtriangleup_{k}^{''}}{\sqrt{\al_{k}}}\xrightarrow[k\rightarrow\infty]{d}N(0,\Sigma_1),
	\end{equation}
	where $\Sigma_1$ is defined in (\ref{def:Sigma_1}),
	and $(U^T)^{(r)}\in\mathbb{R}^{r\times d}$ is composed by first $r$ row vectors of $U^T$.
	
	Note that $\bigtriangleup_{k}=U\left((\bigtriangleup_{k}^{'})^T~\textbf{0}^T\right)^T$ and $\dfrac{\bigtriangleup_{k}^{'}}{\al_k}$ has the same limit  distribution with $\dfrac{\bigtriangleup_{k}^{''}}{\al_k}$.  Therefore,
	\begin{equation}
	\dfrac{\bigtriangleup_{k}}{\sqrt{\al_{k}}}\xrightarrow[k\rightarrow\infty]{d}N(0,\Sigma),
	\end{equation}
	where $\Sigma$ is defined in (\ref{def:Sigma}). Recall that Lemma \ref{lem:con iden} implies that $\bigtriangleup_{k}=\bar{x}_k-x^*$ when $k$ is large enough and hence
	\begin{equation*}
	\mathbb{E}\left[\dfrac{\|\bar{x}_k-x_{j,k}\|}{\sqrt{\al_{k}}}\right]=\mathcal{O}(\sqrt{\al_{k}})\rightarrow 0,  \quad \forall 1\leq j\leq m,
	\end{equation*}
	by Lemma \ref{lem:consis} and (\ref{consensus}). Therefore, an application of Slutsky's theorem yields (\ref{limit d}). The proof is completed.
\end{proof}

Theorem \ref{thm:asym norm} presents the the asymptotic normality of Algorithm \ref{alg:DDA} with the rate $\sqrt{\alpha_k}$.   Note that  $\al_k={a}{k^{-\al}},\al\in(2/3,1)$, the convergence given by (\ref{limit d}) implies that $\delta$ in the convergence rate $x_{j,k}-x^*=o(\al_k^\delta)$ cannot be improved to ${1}/{2}$.   Next, we employ the averaging technique introduced  in \cite{polyak1992acceleration} to
derive the asymptotic efficiency of  Algorithm \ref{alg:DDA}.

For simplicity, we present a technical result first.

\begin{lem}\label{lem:asym effi}
	Suppose Assumptions \ref{ass:sample}, \ref{ass:str conv}-\ref{ass:a} hold with $p>2$ in Assumption \ref{ass:lip func}(ii). Then
	\begin{equation}\label{bigtri}
	\frac{1}{\sqrt{k}}\sum_{t=1}^{k}\|\bigtriangleup_t\|^2\rightarrow 0   \quad   \text{a.s.},
	\end{equation}
	where  the projected error $\Delta_t$ is defined in (\ref{eq:deltat}).
\end{lem}
The proof is presented  in Appendix \ref{proof:lem:asym effi}.

\begin{thm}\label{asymptotic efficient}
	Suppose Assumptions \ref{ass:sample}, \ref{ass:str conv}-\ref{ass:a} hold with $p>2$ in Assumption \ref{ass:lip func}(ii). Let $x^*$ be  the limit point of sequence $\{\bar x_{k}\}$. The covariance matrix 	mapping
	$\sum_{j=1}^m\Cov(\nabla F_j(\cdot; \xi_j))$ is continuous  at point  $x^*$. Then for any $1\leq j\leq m$,
	\begin{equation}\label{limi d 1}
	\frac{1}{\sqrt{k}}\sum_{t=1}^k(x_{j,t}-x^*)\stackrel{d}{\longrightarrow}N(0,\Sigma^*),
	\end{equation}
	where
	\begin{equation*}
	\Sigma^*= H^\dag P_B\bar{\Sigma} P_BH^\dag ,\quad\bar{\Sigma}=\dfrac{1}{m^2}\sum_{j=1}^m\Cov(\nabla F_j(x^*;\xi_j)),
	\end{equation*}
	$H^\dag$ is the  Moore-Penrose inverse of $H$,
	$P_B$ and $H$ are defined in (\ref{eq:pb}) and (\ref{eq:h}) respectively.
\end{thm}
\begin{proof}
	Lemma \ref{lem:con iden} has shown that  $\bigtriangleup_{k}=\bar{x}_k-x^*$ almost surely when $k$ is large enough.
	Then
	\begin{equation*}
	\frac{1}{\sqrt{k}}\sum_{t=1}^k\big(\bar{x}_{t}-x^*\big)\quad {and}  \quad \frac{1}{\sqrt{k}}\sum_{t=1}^k\bigtriangleup_t
	\end{equation*}
	have the same limit distribution.
	Note also that, for any $1\leq j\leq m$,
	\begin{equation}\label{ine err}
	\begin{aligned} &\mathbb{E}\left\|\frac{1}{\sqrt{k}}\sum_{t=1}^k\big(x_{j,t}-x^*\big)-\frac{1}{\sqrt{k}}\sum_{t=1}^k\big(\bar{x}_{t}-x^*\big)\right\|=\mathbb{E}\left\|\frac{1}{\sqrt{k}}\sum_{t=1}^k\big(x_{j,t}-\bar{x}_{t}\big)\right\|\\
	\leq&\frac{1}{\sqrt{k}}\sum_{t=1}^k\mathbb{E}\|x_{j,t}-\bar{x}_{t}\|
	\overset{(\ref{consensus})}{\leq}\frac{1}{\sqrt{k}}\sum_{t=1}^k\mathbb{E}\|z_{j,t-1}-\bar{z}_{t-1}\|\\
	\leq &\frac{\mathbb{E}\|z_{j,0}-\bar{z}_{0}\|}{\sqrt{k}}+\frac{1}{\sqrt{k-1}}\sum_{t=1}^{k-1}\mathbb{E}\|z_{j,t}-\bar{z}_{t}\|\\
	\leq &\frac{\mathbb{E}\|z_{j,0}-\bar{z}_{0}\|}{\sqrt{k}}+\frac{c}{\sqrt{k-1}}\sum_{t=1}^{k-1}\al_t,
	\end{aligned}
	\end{equation}
	where the last inequality follows from (\ref{consensus rate 1}) and $c>0$ is a constant.
	In addition, by the Kronecker lemma and the fact $\sum_{t=1}^\infty\frac{1}{\sqrt{t}}\al_t=\sum_{t=1}^k\dfrac{a}{t^{1/2+\al}}<\infty$,  $\frac{1}{\sqrt{k}}\sum_{t=1}^{k}\al_t\rightarrow 0$. Thus, it follows from (\ref{ine err}) that, for any $1\leq j\leq m$,
	$
	\frac{1}{\sqrt{k}}\sum_{t=1}^k\big(x_{j,t}-x^*\big)\quad  \text{and} \quad \frac{1}{\sqrt{k}}\sum_{t=1}^k\big(\bar{x}_{t}-x^*\big)
	$
	have the same limit distribution. Therefore, it is sufficient to show that
	\begin{equation}
	\label{eq:anormal-delta}
	\frac{1}{\sqrt{k}}\sum_{t=1}^k\bigtriangleup_t\stackrel{d}{\longrightarrow}N(0,\Sigma^*).
	\end{equation}

	In what follows, we employ    \cite[Proposition 2]{duchi2016asymptotic}     to prove (\ref{eq:anormal-delta}).
	Recall   $ \bigtriangleup_{k+1}$  in  (\ref{recursion 0}) of Lemma  \ref{lem:linear recur}:
	\begin{equation}
	\begin{aligned}\label{recursion 2}
	\bigtriangleup_{k+1}&=\bigtriangleup_k-\alpha_{k}H\bigtriangleup_k+\al_k\left(\zeta_k+\eta_k+s_k+\epsilon_k\right)\\
	&=\bigtriangleup_k-\alpha_{k} P_BHP_B\bigtriangleup_k+\al_kP_B\left(\zeta_k+\eta_k+s_k\right)+\al_k\epsilon_k,
	\end{aligned}
	\end{equation}
	where the second equality follows from the   (\ref{eq:h})  and the fact that $\zeta_k,\eta_k, s_k$  defined in (\ref{errors})
	are all in subspace $\mathcal{Y}$. With a slight abuse of notation, define
	\begin{equation}
	\label{eq: zetap}
	\zeta_k^{'}\define\zeta_k+\eta_k, \quad \epsilon_k^{'}\define\al_k\epsilon_k.
	\end{equation}
	Identifying $P_B$, $s_k$, $\zeta_k^{'}$, and $\epsilon_k^{'}$ to $P_\mathcal{T}$, $\xi_k$, $\zeta_k$, and $\varepsilon_k$, respectively, then  (\ref{recursion 2})  falls into  the  form \cite[(34)]{duchi2016asymptotic}.  We are left to verify
	Assumptions F and G of  \cite[Proposition 2]{duchi2016asymptotic}.
	
	Firstly,  by the definition of $H$ in (\ref{eq:h})
	\begin{equation*}
	y^THy=(P_B y)^T\left(\dfrac{1}{m}\nabla^2 f(x^*) \right)P_By=y^T\left(\dfrac{1}{m}\nabla^2 f(x^*) \right)y\ge \dfrac{\mu}{m}\|y\|^2,   \;\; \forall y\in{\mathcal{Y}},
	\end{equation*}
	where the  inequality  follows from  Assumption \ref{ass:str conv}. 	
	Secondly, $\{s_k, \mathcal{F}_{k+1}\}$ is a martingale difference sequence and
	\begin{equation}\label{cond_2}
	\mathbb{E}\left[\|s_k\|^2\big|\mathcal{F}_{k}\right]=\mathbb{E}\left[\left\|\dfrac{1}{m}\sum_{j=1}^ms_{j,k}\right\|^2\big|\mathcal{F}_{k}\right]\le\dfrac{1}{m}\sum_{j=1}^m\mathbb{E}\left[\|s_{j,k}\|^2\big|\mathcal{F}_{k}\right]\le 4L_0^2.
	\end{equation}
	
	For validation of \cite[Assumption F]{duchi2016asymptotic}, we may employ \cite[Lemma 3.3.1]{chen2006stochastic} to prove that
	\begin{equation*}
	\dfrac{1}{\sqrt{k}}\sum_{t=1}^ks_t\overset{d}{\longrightarrow} \mathcal{N}(0,\bar{\Sigma}).
	\end{equation*}
	In fact, since $\{s_k, \mathcal{F}_{k+1}\}$ is a martingale difference sequence satisfying (\ref{cond_2}) and (\ref{noise cov})-(\ref{con 2}), and also the fact
	\begin{equation*}\label{con 4}
	\lim_{N\rightarrow\infty}\sup_{k}\mathbb{E}\left[\|s_k\|^21_{\{\|s_k\|>N\}}\right]\le\lim_{N\rightarrow\infty}\left\|P_B\right\|^{2+2/q}\dfrac{16(L_0^p)^{2/p}(L_0^2)^{1/q}}{N^{2/q}}=0,
	\end{equation*}
	which is similar to the analysis of (\ref{con 3}), then identifying $s_i/\sqrt{k}$ to $\xi_{k,i}$ in \cite[Lemma 3.3.1]{chen2006stochastic}, we can derive the desired argument.
	
	Next, we verify \cite[Assumption G]{duchi2016asymptotic}.
	Recall   (\ref{errors}) and (\ref{eq: zetap}), we have
	\begin{equation*}
	\zeta_{t}^{'}=\zeta_k+\eta_k=-\frac{1}{m}P_B\big[\nabla f(\bar{x}_k)-\nabla f(x^*)-\nabla^2 f(x^*)(\bar{x}_k-x^*)\big]+\frac{1}{m}\sum_{j=1}^mP_B(\nabla f_j(x_{j,k})-\nabla f_j(\bar{x}_k)).
	\end{equation*}
	Then
	\begin{equation}
	\begin{aligned}\label{error ine}
	&\frac{1}{\sqrt{k}}\sum_{t=1}^{k}\|P_B\zeta_{t}^{'}\|\leq  \frac{1}{\sqrt{k}}\sum_{t=1}^{k}\|\zeta_t\|
	+\frac{1}{m\sqrt{k}}\|P_B\|\sum_{t=1}^{k}\sum_{j=1}^{m}\|\nabla f_j(x_{j,t})-\nabla f_j(\bar{x}_{t})\|\\
	\leq& \frac{1}{\sqrt{k}}\sum_{t=1}^{k}
	\|\zeta_t\|1_{\{\|\bar{x}_{t}-x^*\|>\epsilon\}}
	+\frac{1}{m\sqrt{k}}\|P_B\|\sum_{t=1}^{k}\|\bar{x}_{t}-x^*\|^2
	+\frac{L\|P_B\|}{m\sqrt{k}}\sum_{t=1}^{k}\sum_{j=1}^{m}\|x_{j,t}-\bar{x}_{t}\|\\
	\leq&\frac{1}{\sqrt{k}}\sum_{t=1}^{k}
	\|\zeta_t\|1_{\{\|\bar{x}_{t}-x^*\|>\epsilon\}}
	+\frac{1}{m\sqrt{k}}\|P_B\|\sum_{t=1}^{k}\|\bar{x}_{t}-x^*\|^2
	+\frac{L\|P_B\|}{m\sqrt{k}}\sum_{t=1}^{k}\sum_{j=1}^{m}\|z_{j,t-1}-\bar{z}_{t-1}\|,
	\end{aligned}
	\end{equation}
	where the second inequality follows from (\ref{second order massage}) in Assumption \ref{ass:str conv} and the Lipschitz continuity of $\nabla f_j(\cdot)$ in Assumption \ref{ass:str conv}, the third inequality follows from (\ref{consensus}).
	
	We need to show that all terms on the right-hand side of inequality (\ref{error ine}) converge to $0$ almost surely. Evidently, the first term on the right-hand side of inequality (\ref{error ine}) converge to $0$ almost surely as $\bar{x}_{t} \rightarrow x^*$.  Note that
	$\bigtriangleup_t=\bar{x}_{t}-x^*$ when $t$ is large enough, and hence
	the second term converges to $0$ almost surely by Lemma \ref{lem:asym effi}, while the third term converges to $0$ in probability by (\ref{ine err}). Therefore,
	\begin{equation*}
	\frac{1}{\sqrt{k}}\sum_{t=1}^{k}\|P_B\zeta_{t}^{'}\|\rightarrow 0\quad
	\text{a.s.}
	\end{equation*}
	Note also that  by Lemma \ref{lem:con iden},  $\epsilon_k^{'}=\al_k\epsilon_k=0$ when $k$ is large enough and Lemma \ref{lem:asym effi} implies
	\begin{equation*}
	\frac{1}{\sqrt{k}}\sum_{t=1}^{k}\|\bar{x}_{t}-x^*\|^2\rightarrow 0,
	\end{equation*}
	hence \cite[Assumption G]{duchi2016asymptotic} holds. Then an application of \cite[Proposition 2]{duchi2016asymptotic} yields (\ref{eq:anormal-delta}). The proof is completed.
\end{proof}

\section{Numerical simulation}

In this section, we give a numerical example to justify the theoretical analysis. We carry out simulations on the distributed parameter estimation problem \cite{lei2018asymptotic,towfic2015stability}.
Over a connected network consisting of $m$ agents, we want to estimate a real vector $x^*$ in a distributed manner. Each agent $j = 1,\cdots,m$ at time $k$ has access to its real scalar measurement $d_{j,k}$ given by the following linear time-varying model
\begin{equation*}
d_{j,k}=u_{j,k}^Tx^*+v_{j,k},
\end{equation*}
where $u_{j,k}\in\mathbb{R}^d$ is the regression vector accessible to agent $j$, and $v_{j,k}$ is the observation noise of agent $j$. Assume that $\{u_{j,k}\}$ and $\{v_{j,k}\}$ are mutually independent i.i.d. Gaussian sequences with distributions $\mathcal{N}(0,R_{u,j})$ and $\mathcal{N}(0,\sigma^2_{v,j})$ respectively. Then the problem can be reformulated as follows:
\begin{equation}\label{numerical example}
\min\limits_{x\in \mathbb{R}^d} ~f(x)=\sum_{j=1}^mf_j(x)\quad\st ~x\in\mathcal{X},
\end{equation}
where each agent's cost function
\begin{equation*}
f_j(x)\define\mathbb{E}\big[(u_{j,k}^Tx-d_{j,k})^2\big]=(x-x^*)^TR_{u,j}(x-x^*)+\sigma^2_{v,j}.
\end{equation*}

\begin{figure}[http]
	\centering
	\subfigure[Trajectories of $x_{j,k}$  and $\sum_{j=1}^{50}x_{j,k}/50$]{
		\begin{minipage}[t]{0.4\textwidth}
			\centering
			\includegraphics[width=6cm]{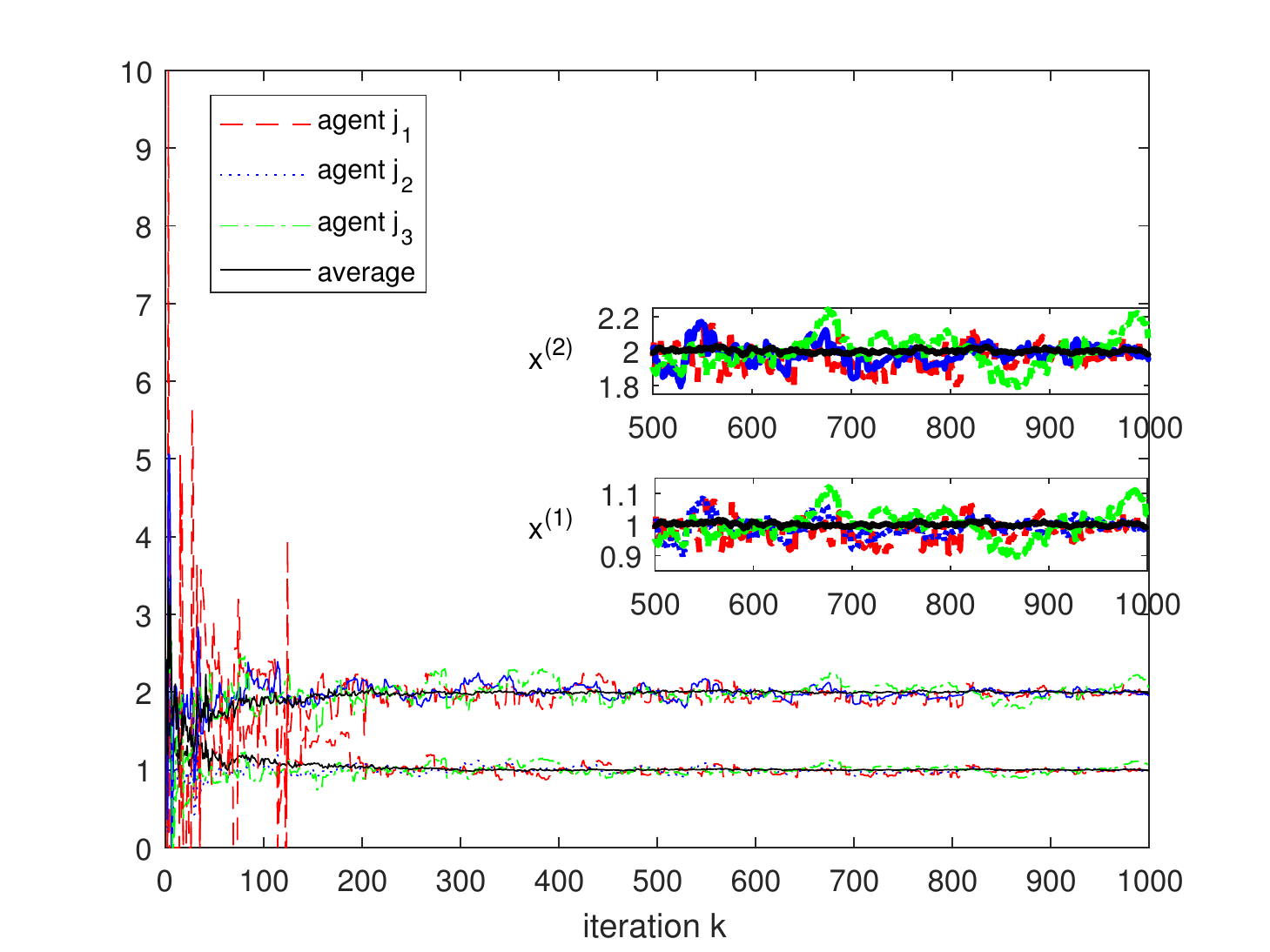}
	\end{minipage}}
	\subfigure[Active set identification of $\bar x_k$]{
		\begin{minipage}[t]{0.4\textwidth}
			\centering
			\includegraphics[width=5.9cm]{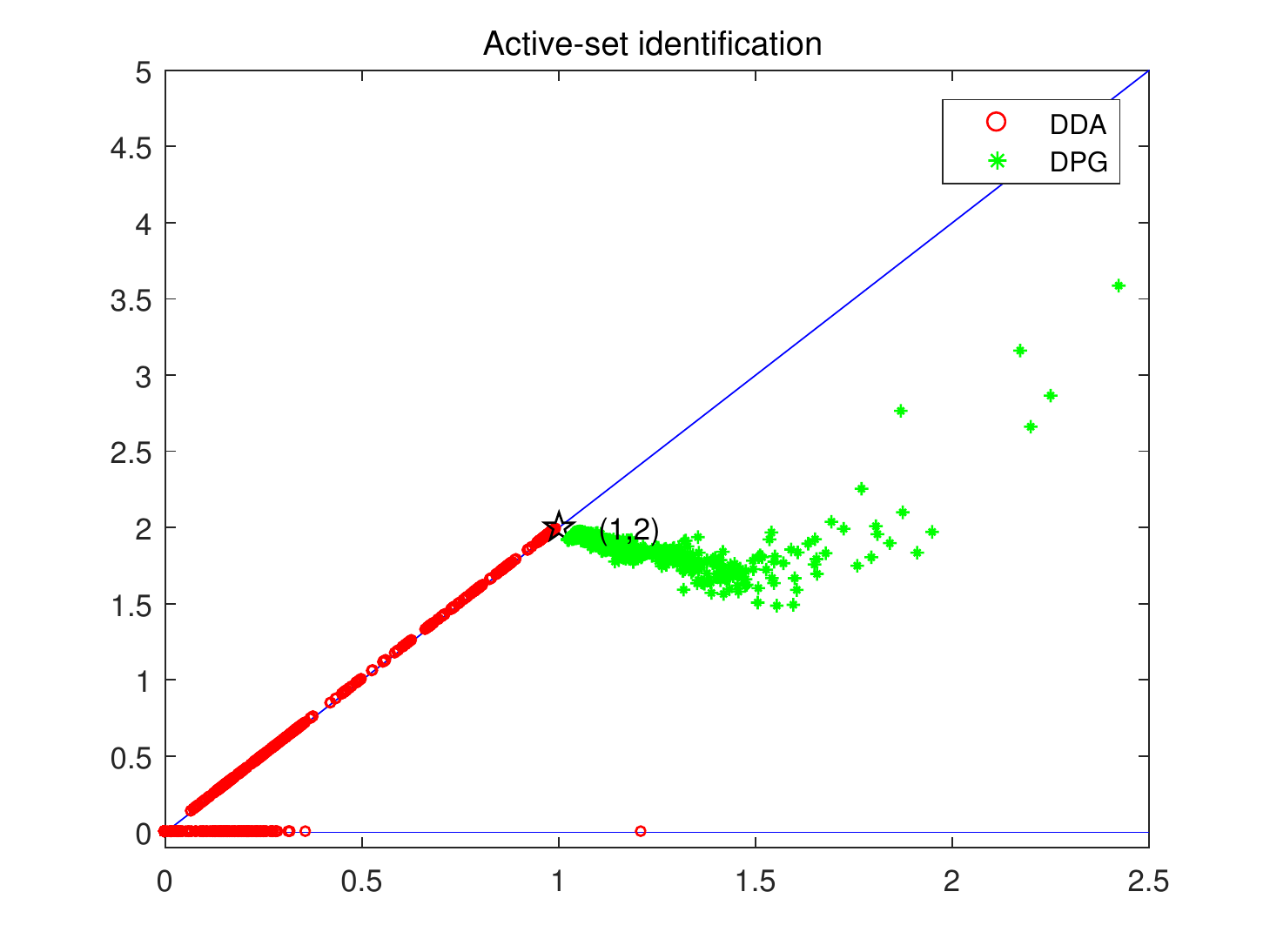}
	\end{minipage}}
	\caption{Convergence properties of selected agents' estimates $x_{j,k}$, $\sum_{j=1}^{50}x_{j,k}/50$, and $\bar x_k$. }
	\label{fig:convergence}
\end{figure}	
\begin{figure}[http]
	\centering
	\subfigure[$\dfrac{x_{1,k}-x^*}{\sqrt{\al_{k}}}$]{
		\begin{minipage}[t]{0.4\textwidth}
			\centering
			\includegraphics[width=6cm]{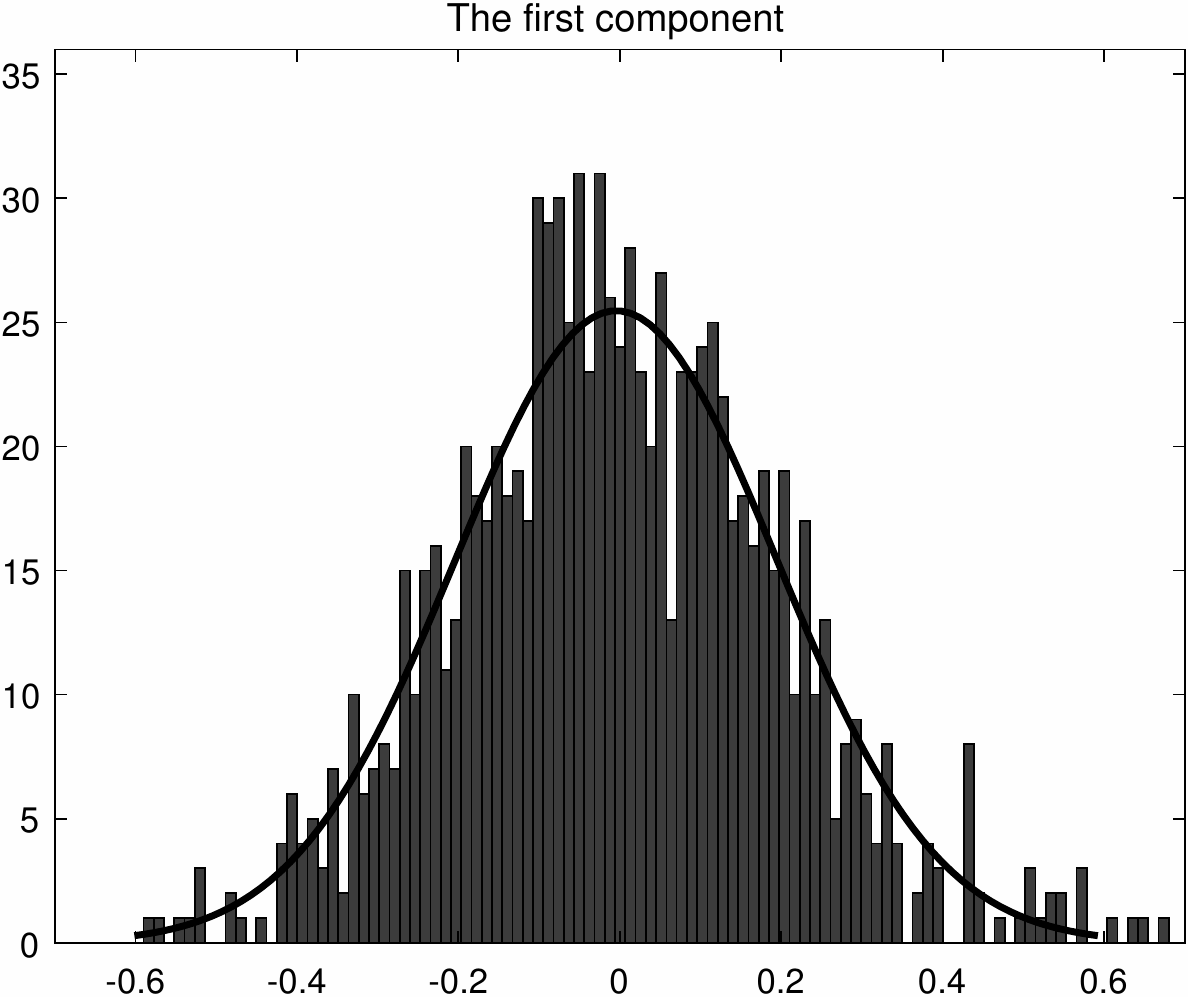}
			\includegraphics[width=6cm]{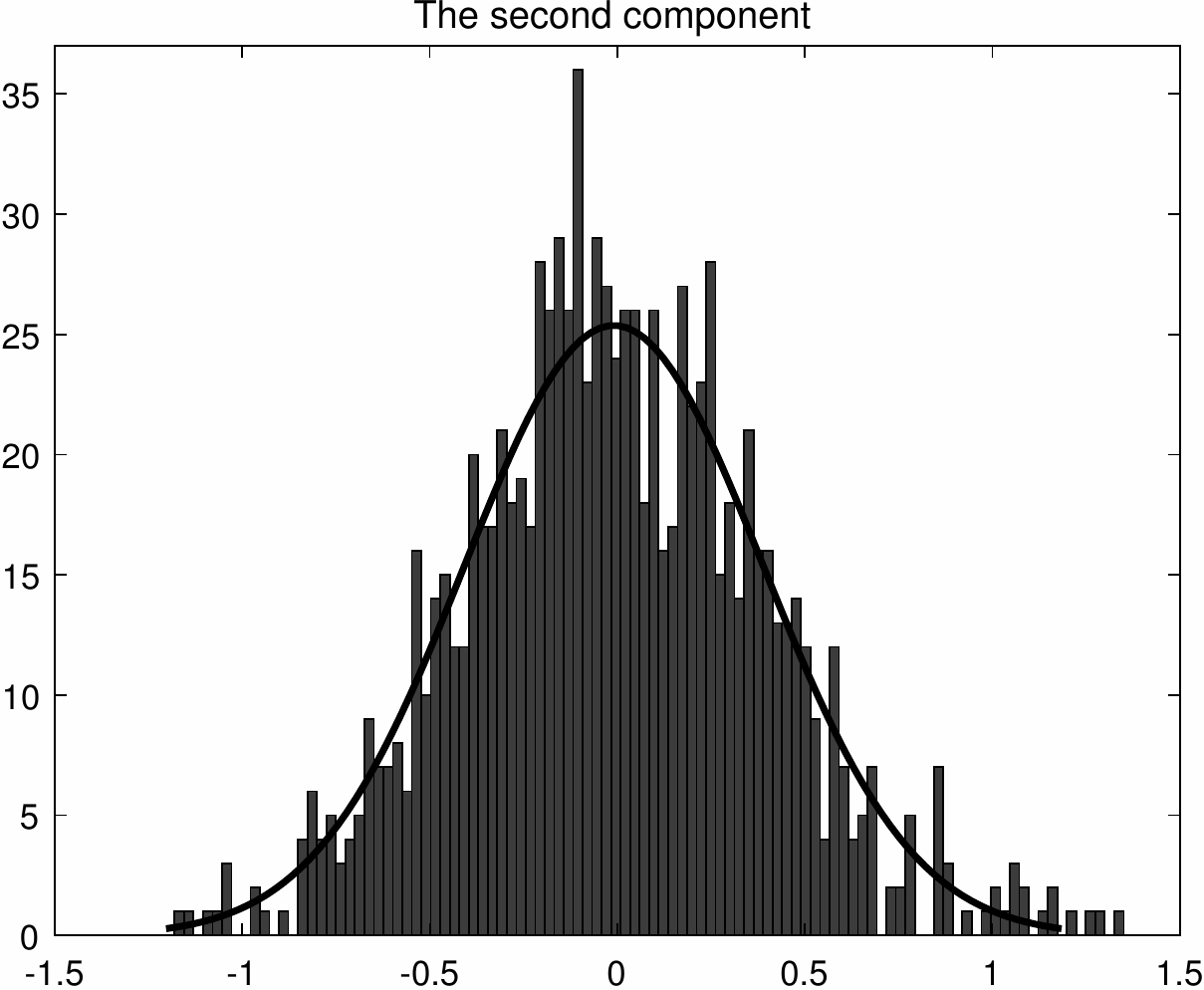}	
			\includegraphics[width=6cm]{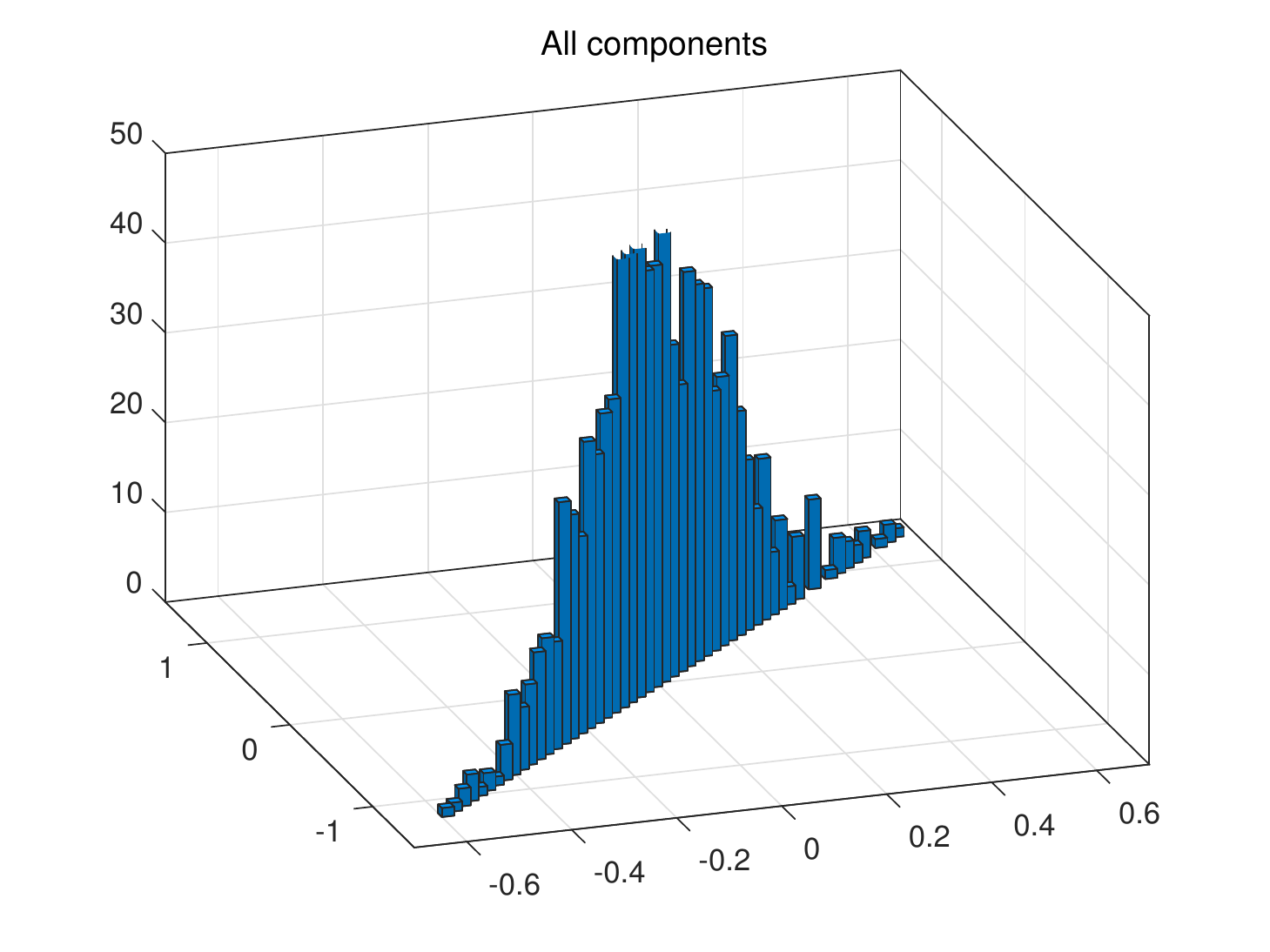}
		\end{minipage}
	}
	\subfigure[$\frac{1}{\sqrt{k}}\sum_{t=1}^k(x_{1,t}-x^*)$]{
		\begin{minipage}[t]{0.4\textwidth}
			\centering
			\includegraphics[width=6cm]{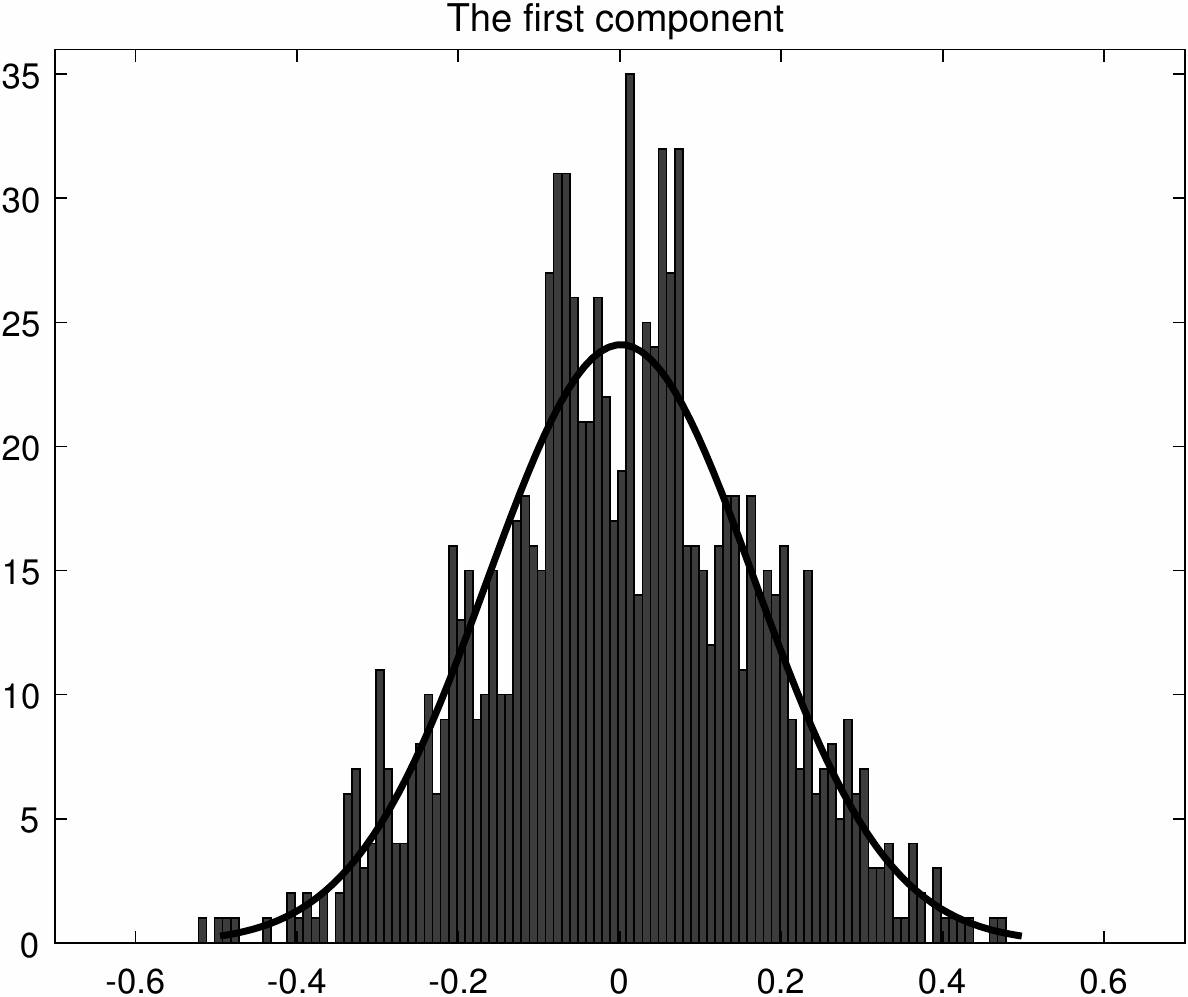}
			\includegraphics[width=6cm]{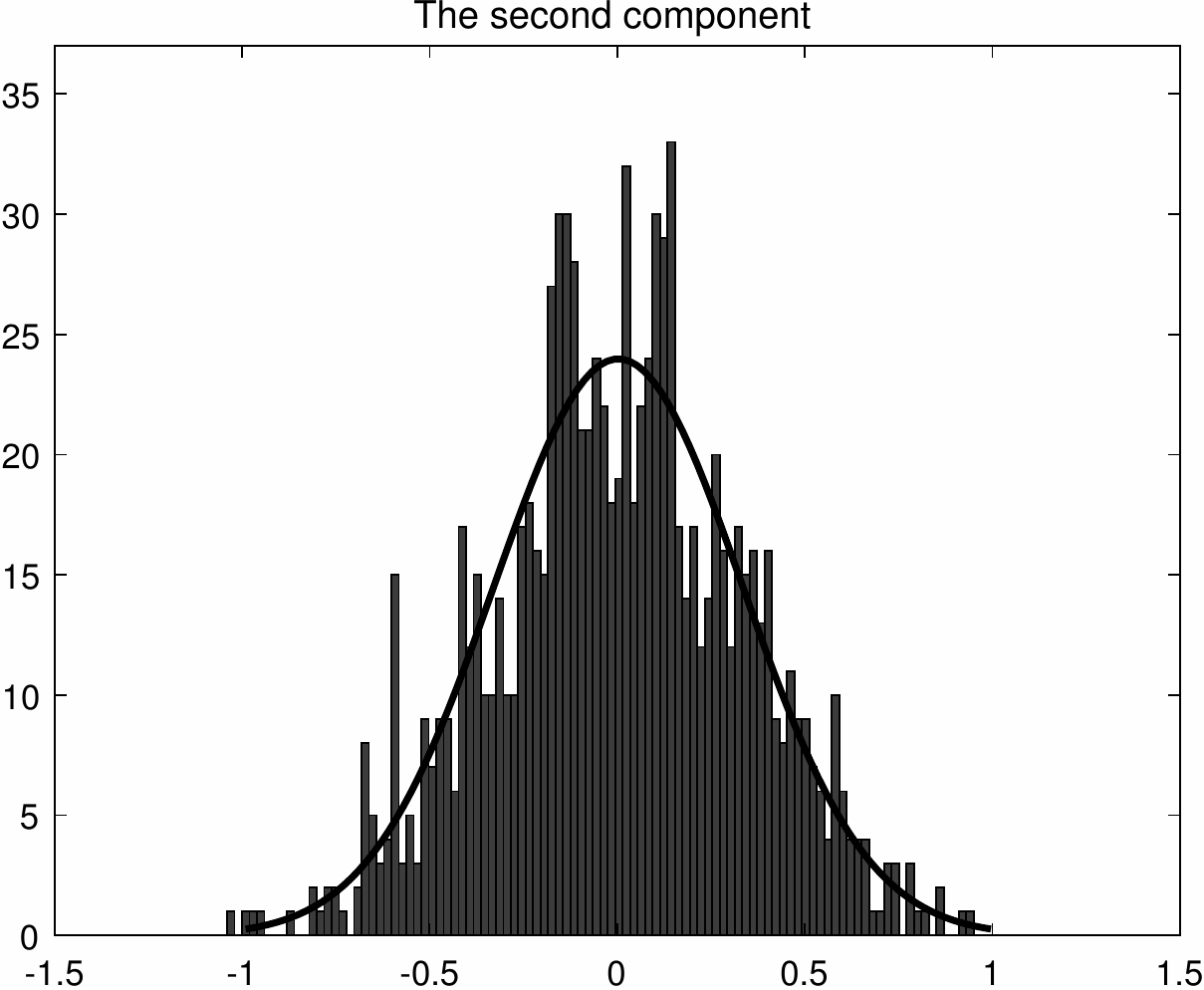}
			\includegraphics[width=6cm]{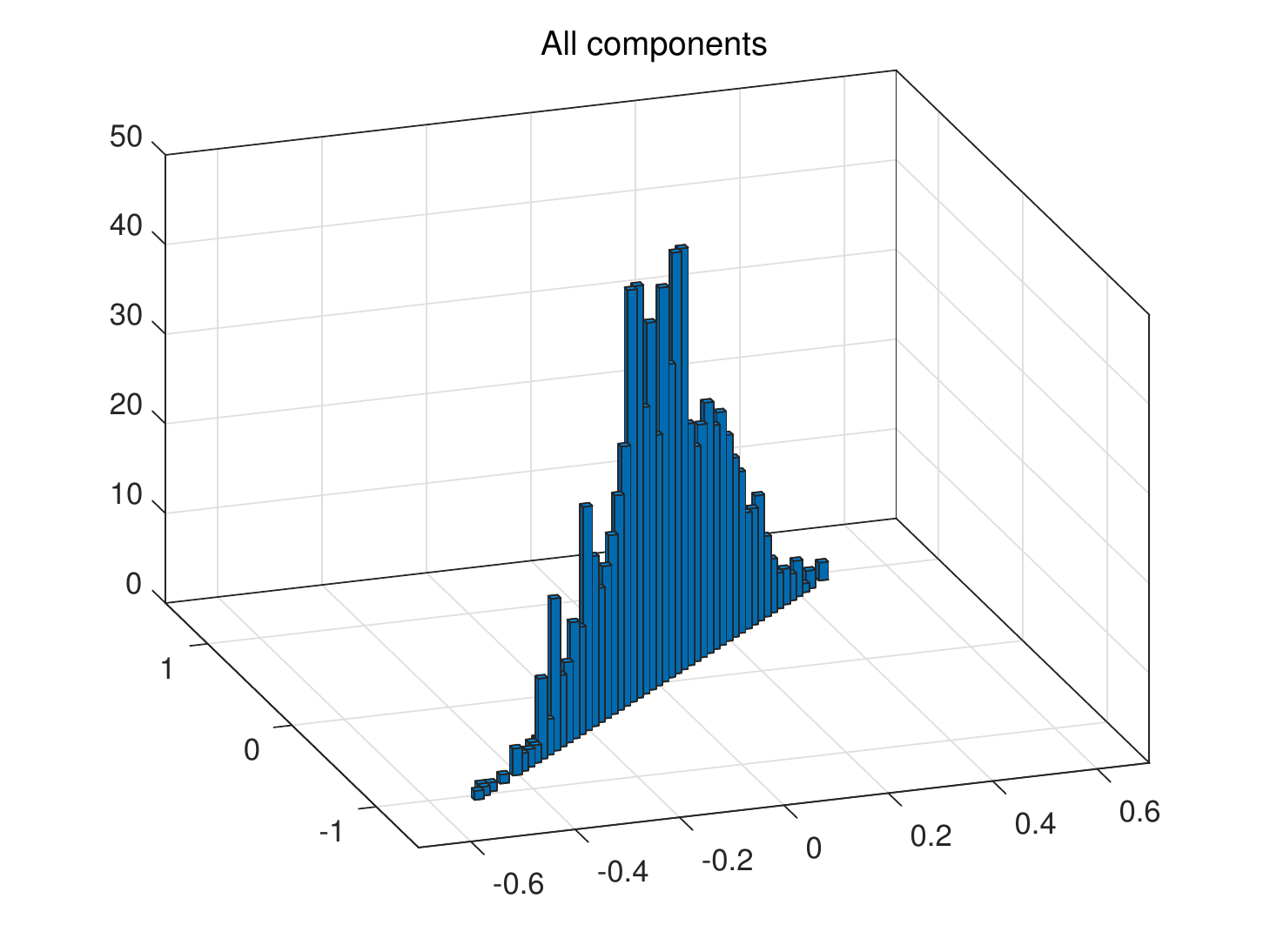}
		\end{minipage}
	}
	\caption{The histograms and limit distributions for $\frac{x_{1,k}-x^*}{\sqrt{\al_{k}}}$ and $\frac{1}{\sqrt{k}}\sum_{t=1}^k(x_{1,t}-x^*)$. }
	\label{fig:asym norm,effi}
\end{figure}

In the numerical test,  we set the optimal solution $x^*=(1, 2)^T$,
\begin{equation*}
\mathcal{X}\define\{(x^{(1)},x^{(2)})^T\in\mathbb{R}^2:-2x^{(1)}+x^{(2)}\le 0,~x^{(1)}\le 5,~x^{(2)}\ge0\}
\end{equation*}
and the subspace corresponds to (\ref{eq: Y}) is
\begin{equation}\label{simu:active-set}
\mathcal{Y}=\{x:-2x^{(1)}+x^{(2)}=0\}.
\end{equation}
$R_{u,j} ,j=1,\cdots, m$ is randomly generated semi-positive definite matrix in $\mathbb{R}^{2\times 2}$. Moreover, the regularizer is  $\psi(x)=\frac{1}{2}\|x\|^2$.  For each implement,	the step-size $\alpha_k=5/k^{0.67}$, the initial point is random generated in set $[0,5]\times[0, 5]$.

In the first simulation, we set the number of agents $m=50$, and the weigh matrix is generated by the broadcast gossip scheme, which is not doubly stochastic but $\mathbf{1}^T\mathbb{E}(A_k)=\mathbf{1}^T$ \cite{bianchi2013performance}.

To demonstrate the path-wise convergence properties of the algorithm, the trajectories with $k\leq 1000$ of selected agents' estimates $x_{j,k}$, which are picked randomly from three of 50 agents, and averaged estimate $\sum_{j=1}^{50} x_{j,k}/50$ are shown in Fig. \ref{fig:convergence}(a). The
simulation results are consistent with Theorem \ref{thm:convergence}.

To show the result of active set identification, the points of $\bar x_k$ generated by Algorithm \ref{alg:DDA} (denoted by DDA) and  distributed projection stochastic gradient (DPG) algorithm are plotted in phase plane respectively. It can be seen from Fig. \ref{fig:convergence}(b) that the DPG algorithm fails to identify the active constraint (\ref{simu:active-set}), while the DDA algorithm identifies it.

In the second simulation, the weigh matrix is generated by the pairwise gossip scheme, which is doubly stochastic \cite{bianchi2013performance}. Algorithm \ref{alg:DDA} is run for 1000 times independently.

Fig. \ref{fig:asym norm,effi} demonstrates the asymptotic normality and asymptotic efficiency of Algorithm \ref{alg:DDA}.
On the one hand, Fig. \ref{fig:asym norm,effi}(a) shows the histograms for each component and all component of $\frac{x_{1,k}-x^*}{\sqrt{\al_{k}}}$ at time $k=2000$ respectively. We use the normal distribution to fit the 1000 samples for $\frac{x_{1,k}^{(1)}-x^{*(1)}}{\sqrt{\al_{k}}}$, $\frac{x_{1,k}^{(2)}-x^{*(2)}}{\sqrt{\al_{k}}}$ and $\frac{x_{1,k}-x^*}{\sqrt{\al_{k}}}$ with $k=2000$. It is shown that the data set are fitted with the normal distribution, which verifies the asymptotic normality result of Theorem \ref{thm:asym norm}.  Moreover, the left bottom figure in Fig. \ref{fig:asym norm,effi} shows that almost all the points lies on the subspace $\mathcal{Y}$ defined by (\ref{simu:active-set}), which is consistent with active-set identification result of Lemma \ref{lem:con iden}.
On the other hand,  Fig. \ref{fig:asym norm,effi}(b) presents the histograms of averaged estimate $\frac{1}{\sqrt{k}}\sum_{t=1}^k(x_{1,t}-x^*)$.
In order to eliminate the impact of  non-identification points of active-set at the beginning iterations,
we take the average of the  last 500 of the 2000 iterations, that is, $\frac{1}{\sqrt{500}}\sum^{2000}_{t=1501}(x_{1,t}-x^*) $.
It is shown from Fig. \ref{fig:asym norm,effi}(b) that the averaged estimates have small variances compared to the left counterparts, which is coherent with the asymptotic efficiency result given in Theorem \ref{asymptotic efficient}.


\bibliographystyle{siamplain}
\bibliography{mybib}

\begin{thebibliography}{10}
\expandafter\ifx\csname url\endcsname\relax
  \def\url#1{\texttt{#1}}\fi
\expandafter\ifx\csname urlprefix\endcsname\relax\def\urlprefix{URL }\fi
\expandafter\ifx\csname href\endcsname\relax
  \def\href#1#2{#2} \def\path#1{#1}\fi

\bibitem{tsitsiklis1984problems}
J.~N. Tsitsiklis, Problems in decentralized decision making and computation.,
  Tech. rep., Massachusetts Inst of Tech Cambridge Lab for Information and
  Decision Systems (1984).

\bibitem{tsitsiklis1986distributed}
J.~{Tsitsiklis}, D.~{Bertsekas}, M.~{Athans}, Distributed asynchronous
  deterministic and stochastic gradient optimization algorithms, IEEE
  Transactions on Automatic Control 31~(9) (1986) 803--812.

\bibitem{bertsekas1989parallel}
D.~P. Bertsekas, J.~N. Tsitsiklis, Parallel and distributed computation:
  numerical methods, Vol.~23, Prentice hall Englewood Cliffs, NJ, 1989.

\bibitem{ren2008distributed}
W.~Ren, R.~W. Beard, Distributed consensus in multi-vehicle cooperative
  control, Vol.~27, Springer-Verlag, London, 2008.

\bibitem{naghshineh1996distributed}
M.~{Naghshineh}, M.~{Schwartz}, Distributed call admission control in
  mobile/wireless networks, IEEE Journal on Selected Areas in Communications
  14~(4) (1996) 711--717.

\bibitem{fitzek2006cooperation}
F.~H. Fitzek, M.~D. Katz, Cooperation in wireless networks: principles and
  applications, Springer, Netherlands, 2006.

\bibitem{lian2017can}
X.~Lian, C.~Zhang, H.~Zhang, C.-J. Hsieh, W.~Zhang, J.~Liu, Can decentralized
  algorithms outperform centralized algorithms? a case study for decentralized
  parallel stochastic gradient descent, Advances in Neural Information
  Processing Systems 30 8 (2018) 5331--5341.

\bibitem{ram2010distributed}
S.~S. Ram, A.~Nedi{\'c}, V.~V. Veeravalli, Distributed stochastic subgradient
  projection algorithms for convex optimization, Journal of Optimization Theory
  and Applications 147~(3) (2010) 516--545.

\bibitem{duchi2011dual}
J.~C. {Duchi}, A.~{Agarwal}, M.~J. {Wainwright}, Dual averaging for distributed
  optimization: Convergence analysis and network scaling, IEEE Transactions on
  Automatic Control 57~(3) (2012) 592--606.

\bibitem{yuan2014randomized}
D.~Yuan, D.~W. Ho, Randomized gradient-free method for multiagent optimization
  over time-varying networks, IEEE Transactions on Neural Networks and Learning
  systems 26~(6) (2014) 1342--1347.

\bibitem{chen2017strong}
X.-M. Chen, C.~Gao, Strong consistency of random gradient-free algorithms for
  distributed optimization, Optimal Control Applications and Methods 38~(2)
  (2017) 247--265.

\bibitem{nedic2016push-sum}
A.~{Nedi\'{c}}, A.~{Olshevsky}, Stochastic gradient-push for strongly convex
  functions on time-varying directed graphs, IEEE Transactions on Automatic
  Control 61~(12) (2016) 3936--3947.

\bibitem{bianchi2012convergence}
P.~{Bianchi}, J.~{Jakubowicz}, Convergence of a multi-agent projected
  stochastic gradient algorithm for non-convex optimization, IEEE Transactions
  on Automatic Control 58~(2) (2013) 391--405.

\bibitem{shah2018distributed}
S.~M. Shah, V.~S. Borkar, Distributed stochastic approximation with local
  projections, SIAM Journal on Optimization 28~(4) (2018) 3375--3401.

\bibitem{lee2013distributed}
S.~{Lee}, A.~{Nedic}, Distributed random projection algorithm for convex
  optimization, IEEE Journal of Selected Topics in Signal Processing 7~(2)
  (2013) 221--229.

\bibitem{ram2009asynchronous}
S.~{Sundhar Ram}, A.~{Nedić}, V.~V. {Veeravalli}, Asynchronous gossip
  algorithms for stochastic optimization, in: Proceedings of the 48h IEEE
  Conference on Decision and Control (CDC) held jointly with 2009 28th Chinese
  Control Conference, 2009, pp. 3581--3586.

\bibitem{chung1954stochastic}
K.~L. Chung, On a stochastic approximation method, The Annals of Mathematical
  Statistics (1954) 463--483.

\bibitem{sacks1958asymptotic}
J.~Sacks, Asymptotic distribution of stochastic approximation procedures, The
  Annals of Mathematical Statistics 29~(2) (1958) 373--405.

\bibitem{fabian1968asymptotic}
V.~Fabian, On asymptotic normality in stochastic approximation, The Annals of
  Mathematical Statistics 39~(4) (1968) 1327--1332.

\bibitem{ruppert1985newton}
D.~Ruppert, A newton-raphson version of the multivariate {Robbins-Monro}
  procedure, The Annals of Statistics (1985) 236--245.

\bibitem{polyak1992acceleration}
B.~T. Polyak, A.~B. Juditsky, Acceleration of stochastic approximation by
  averaging, SIAM Journal on Control and Optimization 30~(4) (1992) 838--855.

\bibitem{duchi2016asymptotic}
J.~Duchi, F.~Ruan, Asymptotic optimality in stochastic optimization (2016).
\newblock \href {http://arxiv.org/abs/1612.05612} {\path{arXiv:1612.05612}}.

\bibitem{bianchi2013performance}
P.~Bianchi, G.~Fort, W.~Hachem, Performance of a distributed stochastic
  approximation algorithm, IEEE Transactions on Information Theory 59~(11)
  (2013) 7405--7418.

\bibitem{lei2018asymptotic}
J.~Lei, H.-F. Chen, H.-T. Fang, Asymptotic properties of primal-dual algorithm
  for distributed stochastic optimization over random networks with imperfect
  communications, SIAM Journal on Control and Optimization 56~(3) (2018)
  2159--2188.

\bibitem{nesterov2009primal}
Y.~Nesterov, Primal-dual subgradient methods for convex problems, Mathematical
  Programming 120~(1) (2009) 221--259.

\bibitem{xiao2010DA}
L.~Xiao, Dual averaging methods for regularized stochastic learning and online
  optimization, J. Mach. Learn. Res. 11 (2010) 2543–2596.

\bibitem{lee2012manifold}
S.~Lee, S.~J. Wright, Manifold identification in dual averaging for regularized
  stochastic online learning, Journal of Machine Learning Research 13~(Jun)
  (2012) 1705--1744.

\bibitem{agarwal2011distributed}
A.~Agarwal, J.~C. Duchi, Distributed delayed stochastic optimization, in:
  Advances in Neural Information Processing Systems, 2011, pp. 873--881.

\bibitem{yuan2012distributed}
D.~Yuan, S.~Xu, H.~Zhao, L.~Rong, Distributed dual averaging method for
  multi-agent optimization with quantized communication, Systems \& Control
  Letters 61~(11) (2012) 1053--1061.

\bibitem{hosseini2013online}
S.~Hosseini, A.~Chapman, M.~Mesbahi, Online distributed optimization via dual
  averaging, in: 52nd IEEE Conference on Decision and Control, IEEE, 2013, pp.
  1484--1489.

\bibitem{liang2019distributed}
S.~Liang, L.~Wang, G.~Yin, Distributed quasi-monotone subgradient algorithm for
  nonsmooth convex optimization over directed graphs, Automatica 101 (2019)
  175--181.

\bibitem{nesterov2015quasi}
Y.~Nesterov, V.~Shikhman, Quasi-monotone subgradient methods for nonsmooth
  convex minimization, Journal of Optimization Theory and Applications 165~(3)
  (2015) 917--940.

\bibitem{zhou2020convergence}
Z.~Zhou, P.~Mertikopoulos, N.~Bambos, S.~P. Boyd, P.~W. Glynn, On the
  convergence of mirror descent beyond stochastic convex programming, SIAM
  Journal on Optimization 30~(1) (2020) 687--716.

\bibitem{wright1993identifiable}
S.~J. Wright, Identifiable surfaces in constrained optimization, SIAM Journal
  on Control and Optimization 31~(4) (1993) 1063--1079.

\bibitem{chen2006stochastic}
H.-F. Chen, Stochastic approximation and its applications, Vol.~64, Kluwer
  Academic Publishers, New York, 2006.

\bibitem{Ljung1992Stochastic}
L.~Ljung, G.~Pflug, H.~Walk, Stochastic approximation and optimization of
  random systems, Vol.~17, Birkh{\"a}user, 2012.

\bibitem{xu2012consensus}
J.~Xu, H.~Zhang, L.~Shi, Consensus and convergence rate analysis for
  multi-agent systems with time delay, in: 2012 12th International Conference
  on Control Automation Robotics \& Vision (ICARCV), IEEE, 2012, pp. 590--595.

\bibitem{Tang2018Convergence}
H.~Tang, T.~Li, Convergence rates of discrete-time stochastic approximation
  consensus algorithms: Graph-related limit bounds, Systems \& Control Letters
  112 (2018) 9--17.

\bibitem{towfic2015stability}
Z.~J. Towfic, A.~H. Sayed, Stability and performance limits of adaptive
  primal-dual networks, IEEE Transactions on Signal Processing 63~(11) (2015)
  2888--2903.

\bibitem{Robbins1971A}
H.~Robbins, D.~Siegmund, A convergence theorem for non negative almost
  supermartingales and some applications, Optimizing Methods in Statistics
  (1971) 233--257.

\end{thebibliography}

\section*{Appendix}
\begin{appendices}
\section{Proof of Lemma \ref{lem:consis}}\label{proof:consis}
\begin{proof}
	Let
	\begin{equation}\label{weight mat}
	\hat{A}_k:=A_k\otimes I_d,  \quad J:=\frac{1}{m}\textbf{1}\textbf{1}^T\otimes I_d,  \quad J_{\bot}:=I_{md}-J,
	\end{equation}
	where $\otimes$ denotes the Kronecker product. Denote
	\begin{equation}\label{consensus vector}
	Z_k:=
	\begin{pmatrix}
	z_{1,k}\\
	z_{2,k}\\
	\vdots\\
	z_{m,k}\\
	\end{pmatrix}
	\in \mathbb{R}^{md}, \quad
	\bar{Z}_k:=JZ_k=
	\begin{pmatrix}
	\bar{z}_k\\
	\bar{z}_k\\
	\vdots\\
	\bar{z}_k\\
	\end{pmatrix}
	\in \mathbb{R}^{md},\quad
	{Z}_{k,\bot}\define J_{\bot}Z_k=Z_k-\bar{Z}_k.
	\end{equation}
	We prove the lemma by investigating the recursion of disagreement vector ${Z}_{k,\bot}$. Recall the following recursion in Algorithm \ref{alg:DDA}
	$$z_{j,k}=\sum_{i\in N_{j,k}}[A_k]_{ji}z_{i,k-1}-\alpha_{k}\nabla F_j(x_{j,k};\xi_{j,k}),$$
	which reduces to
	\begin{equation}\label{Z-recursion}
	Z_{k}=\hat{A}_{k}Z_{k-1}-\alpha_{k}G_{k} \text{ with } G_{k}=\left(\nabla F_1(x_{1,k};\xi_{1,k})^T,\cdots,\nabla F_m(x_{m,k};\xi_{m,k})^T\right)^T
	\end{equation}
	by using the notation (\ref{weight mat}). Hence, we obtain the recursion for $Z_{k,\bot}$
	\begin{equation}\label{expandation of consensus error norm1}
	{Z}_{k,\bot}=J_{\bot}\hat{A}_{k}Z_{k-1}-\alpha_{k}J_{\bot}G_{k}
	=J_{\bot}\hat{A}_{k}\bar{Z}_{k-1,\bot}-\alpha_{k}J_{\bot}G_{k},
	\end{equation}
	where the second equality follows from the fact
	$
	J_{\bot}\hat{A}_kJ_{\bot}=J_{\bot}\hat{A}_k.
	$
	Introducing an auxiliary matrix
	$$W_{k}:=\hat{A}_{k}J_{\bot}^2\hat{A}_{k}=\Big(A_k^T(I_m-\tfrac{11^T}{m})A_k\Big)\otimes I_{d},$$
	it follows from (\ref{expandation of consensus error norm1}) that
	\begin{equation}
	\begin{aligned}\label{expandation of consensus error norm}
	\left\|{Z}_{k,\bot}\right\|^2
	=&{Z}_{k-1,\bot}^TW_{k}{Z}_{k-1,\bot}+\alpha_{k}^2G_{k}^TJ_{\bot}^2G_{k}
	-2\alpha_{k}G_{k}^TJ_{\bot}^2\hat{A}_{k}{Z}_{k-1,\bot}\\
	=&{Z}_{k-1,\bot}^TW_{k}{Z}_{k-1,\bot}+\alpha_{k}^2G_{k}^TJ_{\bot}G_{k}
	-2\alpha_{k}G_{k}^TJ_{\bot}\hat{A}_{k}{Z}_{k-1,\bot},
	\end{aligned}
	\end{equation}
	where the second equality follows from the fact $J_{\bot}^2=J_{\bot}$.

	Note that $W_{k}$ is independent of $\mathcal{F}_{k}$ and $G_{k}$ by Assumption \ref{ass:A rho} and \ref{ass:sample}.  Taking conditional expectation on both side of  (\ref{expandation of consensus error norm}) with respect to $\mathcal{F}_{k}$ and $G_{k}$, we have
	\begin{equation*}
	\mathbb{E}\left[\|\bar{Z}_{k,\bot}\|^2\big|\mathcal{F}_{k},G_{k}\right]\leq\rho_{k}\|\bar{Z}_{k-1,\bot}\|^2
	+2\alpha_{k}\sqrt{m}\|J_{\bot}\|\|\bar{Z}_{k-1,\bot}\|\|G_{k}\|
	+\alpha_{k}^2\|J_{\bot}\|\|G_{k}\|^2.
	\end{equation*}
	where the bound $\|\hat{A}_k\|=\|A_k\otimes I_d\| =\|A_k\| \leq \sqrt{m}$ is obtained by the row stochasticity of $A_k$. Taking expectation on both sides, we arrive at
	\begin{equation}
	\begin{aligned}\label{1 expandation of consensus error norm}
	&\mathbb{E}\left[\|\bar{Z}_{k,\bot}\|^2\right]\\
	\leq&\rho_{k}\mathbb{E}\left[\|\bar{Z}_{k-1,\bot}\|^2\right]
	+2\alpha_{k}\sqrt{m}\|J_{\bot}\|\mathbb{E}\left[\|\bar{Z}_{k-1,\bot}\|\|G_{k}\|\right]+\alpha_{k}^2\|J_{\bot}\|\mathbb{E}\left[\|G_{k}\|^2\right]\\
	\leq&\rho_{k}\mathbb{E}\left[\|\bar{Z}_{k-1,\bot}\|^2\right]
	+2\alpha_{k}\sqrt{m}\|J_{\bot}\|\sqrt{\mathbb{E}\left[\|\bar{Z}_{k-1,\bot}\|^2\right]\mathbb{E}\left[\|G_{k}\|^2\right]}+\alpha_{k}^2\|J_{\bot}\|mL_0^2\\
	\leq&\rho_{k}\mathbb{E}\left[\|\bar{Z}_{k-1,\bot}\|^2\right]
	+2\alpha_{k}m\|J_{\bot}\|\sqrt{L_0^2}\sqrt{\mathbb{E}\left[\|\bar{Z}_{k-1,\bot}\|^2\right]}+\alpha_{k}^2\|J_{\bot}\|mL_0^2,
	\end{aligned}
	\end{equation}
	where the second inequality follows from the Cauchy-Schwarz inequality and the last inequality follows from the fact
	\begin{equation*}
	\begin{aligned}
	\sqrt{\mathbb{E}\left[\|G_{k}\|^2\right]}&= \sqrt{\mathbb{E}\left[\left\|\left(\nabla F_1(x_{1,k-1};\xi_{1,k-1})^T,\nabla F_2(x_{2,k-1};\xi_{2,k-1})^T\cdots,\nabla F_m(x_{m,k-1};\xi_{m,k-1})^T\right)^T\right\|^2\right]}\\
	&=\sqrt{\mathbb{E}\left[\sum_{j=1}^m\left\|\nabla F_j(x_{j,k-1};\xi_{j,k})\right\|^2\right]}\le \sqrt{mL_0^2},
	\end{aligned}
	\end{equation*}
	and $L_0^2$  is defined in  (\ref{eq: L0}). 
	
	Define $u_k=\mathbb{E}\left[\|\bar{Z}_{k,\bot}\|^2\right],M=\max\{2m\|J_{\bot}\|\sqrt{L_0^2},\|J_{\bot}\|mL_0^2\}$. 
	Then (\ref{1 expandation of consensus error norm}) can be rewritten as
	\begin{equation}\label{recursion-consensus-error}
	u_{k}\leq\rho_{k}u_{k-1}
	+M\al_k\sqrt{u_{k-1}}+M\al_k^2.
	\end{equation}
	
	We now apply \cite[Lemma 3]{bianchi2013performance} to prove the lemma. For this, we need to validate conditions (22)-(25) of \cite[Lemma 3]{bianchi2013performance}. First of all, the step-size $\al_k$ fulfills the requirement and (\ref{recursion-consensus-error}) can be viewed as a special case of (22)-(23) of \cite[Lemma 3]{bianchi2013performance} with $v_k\equiv0$. Then, we verify the bound $\mathop{\lim\sup}_k \phi_ku_k$ for two scenarios of the lemma.
	
	(i) Taking $\phi_k=k^{2\beta}$, by Assumption \ref{ass:rho a} (ii), we have
	\begin{equation}\label{verify-phi-conditions}
	\left\{\begin{aligned}
	&\mathop{\lim\sup}_k\Big(\al_k\sqrt{\phi_k}+\tfrac{\phi_{k-1}}{\phi_k}\Big)=\mathop{\lim\sup}_k\Big(\al_kk^\beta+(1-\tfrac{1}{k})^\beta\Big)=1<\infty\\
	&\mathop{\lim\inf}_k(\al_k\sqrt{\phi_k})^{-1}\Big(\tfrac{\phi_{k-1}}{\phi_k}-\rho_k\Big)=\mathop{\lim\inf}_k\frac{(1-\tfrac{1}{k})^\beta-\rho_k}{\al_kk^\beta}>0\\
	&\sum_{k=1}^\infty\phi_k^{-1}=\sum_{k=1}^\infty\frac{1}{k^{2\beta}}<\infty
	\end{aligned}\right.
	\end{equation}
	hence all conditions of \cite[Lemma 3]{bianchi2013performance} are satisfied, we obtain (\ref{consensus rate}).
	
	(ii) Noticing $\|\bar{Z}_{k,\bot}\|^2=\sum_{j=1}^m\|\bar{z}_k-z_{j,k}\|^2$ and (\ref{consensus rate}), there exists a constant $d$ such that for any $j\in V$
	\begin{equation*}
	\sum_{k=1}^\infty\mathbb{E}\left[\|\bar{z}_k-z_{j,k}\|^2\right]\leq d^2\sum_{k=1}^\infty k^{-2\beta}<\infty.
	\end{equation*}
	By the monotone convergence theorem, we have
	\begin{equation*}
	\sum_{k=0}^{\infty}\|\bar{z}_k-z_{j,k}\|^2<\infty\quad \text{a.s.}
	\end{equation*}
	Similarly,
	\begin{equation*}
	\sum_{k=1}^\infty\gamma_k\mathbb{E}\left[\|z_k-z_{j,k}\|\right]\leq\sum_{k=1}^\infty\gamma_k\sqrt{\mathbb{E}\left[\|z_k-z_{j,k}\|^2\right]}\leq d\sum_{k=1}^\infty\gamma_kk^{-\beta}<\infty,
	\end{equation*}
	and hence we obtain that for any $j\in V$
	\begin{equation*}
	\sum_{k=1}^\infty\gamma_k\|z_k-z_{j,k}\|<\infty\quad \text{a.s.}
	\end{equation*}
	
	(iii) If $A_k,   k=1,2, \cdots $ satisfies (a) and the step-size satisfies (b), then by taking $\phi_k=\al_k^{-2}$,  we can also show in a similar way to (\ref{verify-phi-conditions}) that all conditions of \cite[Lemma 3]{bianchi2013performance} are satisfied, and hence (\ref{consensus rate 1}) holds.
\end{proof}

\section{Proof of Lemma \ref{lem:con iden}}\label{Proof:con iden}
\begin{proof}
	By the iteration (\ref{consensus dual averaging alrithm 1})
	\begin{equation}\label{iteration}
	\bar{x}_{k+1}=\amin_{x\in\{Bx\leq b,~Cx\leq c\}}\left\{\langle\nabla f(x^*),x\rangle+\langle v_k,x\rangle+\frac{m}{2\widetilde{\al}_k}\|x\|^2\right\},
	\end{equation}
	where
	$$
	v_k=\frac{-m\bar{z}_k}{\widetilde{\al}_k}-\nabla f(x^*), \quad \widetilde{\al}_k\define\sum_{t=1}^k\al_{t}.
	$$
	According to the Karush-Kuhn-Tucker (KKT) conditions of problem (\ref{iteration}), there exist $\lambda_k,\mu_k\geq 0$ such that
	$$\nabla f(x^*)+v_k+\frac{m\bar{x}_{k+1}}{\widetilde{\al}_k}+B^T\lambda_k+C^T\mu_k=0.$$
	Note that $\bar{x}_k\rightarrow x^*$ almost surely and $\widetilde{\al}_k\define\sum_{t=1}^k\al_{t}\rightarrow\infty$, which implies
	$$\frac{m\bar{x}_{k+1}}{\widetilde{\al}_k}\rightarrow 0\quad \text{a.s.}$$
	If $v_k\to 0$ almost surely,  it is easy to show in a similar way to \cite[part 12.1]{duchi2016asymptotic} that
	$$B\bar{x}_k=b,~C\bar{x}_k<c\quad \text{a.s.}$$
	when $k$ is large enough.
	
	Next,  we show $v_k\to 0$ almost surely. For convenience of notation, we denote
	\begin{equation*}
	\nabla\mathbf{f}^*\define
	\begin{pmatrix}
	\nabla f(x^*)\\
	\nabla f(x^*)\\
	\cdots\\
	\nabla f(x^*)
	\end{pmatrix},
	\;  \nabla\mathbf{f}_{ag}^*=    \begin{pmatrix} \nabla f_1(x^*)\\\nabla f_2(x^*)\\\cdots \\\nabla f_m(x^*) \end{pmatrix},  \;
	\nabla\mathbf{f}_{t} \define  \begin{pmatrix}
	\nabla f_1(x_{1,t})\\
	\nabla f_2(x_{2,t})\\ \cdots \\ \nabla f_m(x_{m,t})   \end{pmatrix}, \; S_{t}\define  \begin{pmatrix} s_{1,t}\\ s_{2,t} \\ \cdots  \\ s_{m,t}   \end{pmatrix}.
	\end{equation*}
	Recall the  definition of $\bar{Z}_k$ in (\ref{consensus vector}),
	\begin{equation*}
	\|v_k\|^2=\left\|\frac{-m\bar{z}_k}{\widetilde{\al}_k}-\nabla f(x^*)\right\|^2=\dfrac{1}{m}\left\|\frac{-m\bar{Z}_k}{\widetilde{\al}_k}-\nabla\mathbf{f}^*\right\|^2.
	\end{equation*}
	Then it is sufficient to show $\left\|\frac{-m\bar{Z}_k}{\widetilde{\al}_k}-\nabla\mathbf{f}^*\right\|^2$ converges to 0 almost surely.
	Recall the definitions of (\ref{consensus vector}) and (\ref{Z-recursion}),
	Note also that
	\begin{equation*}
	\begin{aligned}
	\bar{Z}_k=JZ_k
	&=J\left(\hat{A}_kZ_{k-1}-\alpha_{k}G_{k}
	\right)\\
	&=J\hat{A}_kZ_{k-1}-\alpha_{k}JG_{k}
	=JZ_{k-1}-\alpha_{k}JG_{k}\\
	&=\bar{Z}_{k-1}-\alpha_{k}JG_{k}
	\cdots
	=\bar{Z}_{0}-\sum_{t=1}^k\alpha_tJG_t,
	\end{aligned}
	\end{equation*}
	where the third equality follows from the fact that $\hat{A}_k$ is doubly stochastic. Without loss of generality, we set $\bar{Z}_{0}=\textbf{0}$. Then  by the fact $G_t= \nabla\mathbf{f}_{t}+S_{t}$
	\begin{equation*}
	\bar{Z}_k=-\sum_{t=1}^k\alpha_tJG_t
	=-\sum_{t=1}^k\alpha_tJ\nabla\mathbf{f}_{t}-\sum_{t=1}^k\alpha_tJS_{t}.
	\end{equation*}
	Thus
	\begin{equation*}
	\left\|\frac{-m\bar{Z}_k}{\widetilde{\al}_k}-\nabla\mathbf{f}^*\right\|^2\leq  2\left\|\sum_{t=1}^k\dfrac{\alpha_t}{\widetilde{\al}_k}mJ\nabla\mathbf{f}_{t}-\nabla\mathbf{f}^*\right\|^2+2\left\|\sum_{t=1}^k\dfrac{\alpha_t}{\widetilde{\al}_k}mJS_{t}\right\|^2,
	\end{equation*}
	where the inequality due to the fact $\|a+b\|^2\le 2\|a\|^2+2\|b\|^2$. We left to show that the two terms on the right-hand side of above inequality converge to 0.
	
	Note that
	\begin{equation*}
	\begin{aligned}
	\left\|\sum_{t=1}^k\frac{\alpha_t}{\widetilde{\al}_k}mJ\nabla\mathbf{f}_{t}-\nabla\mathbf{f}^*\right\|^2
	&=\left\|\sum_{t=1}^k\frac{\alpha_t}{\widetilde{\al}_k}mJ\left(\nabla\mathbf{f}_{t}-\nabla\mathbf{f}_{ag}^*\right)\right\|^2\\
	&\le m^2\|J\|^2\left\|\sum_{t=1}^k\frac{\alpha_t}{\widetilde{\al}_k}\left(\nabla\mathbf{f}_{t}-\nabla\mathbf{f}_{ag}^*\right)\right\|^2\\
	&\leq m^2\|J\|^2\sum_{t=1}^k\frac{\alpha_t}{\widetilde{\al}_k}\|\nabla\mathbf{f}_{t}-\nabla\mathbf{f}_{ag}^*\|^2\\
	&=m^2\|J\|^2\sum_{t=1}^k\frac{\alpha_t}{\widetilde{\al}_k}\sum_{j=1}^m\|\nabla f_{j}(x_{j,t})-\nabla f_j(x^*)\|^2\\
	&\leq m^2\|J\|^2L\sum_{t=1}^k\frac{\alpha_t}{\widetilde{\al}_k}\sum_{j=1}^m\|x_{j,t}-x^*\|^2,
	\end{aligned}
	\end{equation*}
	where the first equality follows  from the fact $\nabla\mathbf{f}^*=mJ\nabla\mathbf{f}_{ag}^*$, the second inequality follows from the convexity of $\|\cdot\|^2$ and the fact $\sum_{t=1}^k\frac{\alpha_t}{\widetilde{\al}_k}=1$, the third inequality follows from  the Lipschitz continuity of $\nabla f_j(\cdot)$.
	Moreover,
	\begin{equation}\label{suf cond 0}
	\begin{aligned}
	\|\sum_{t=1}^k\frac{\alpha_t}{\widetilde{\al}_k}mJ\nabla\mathbf{f}_{t}-\nabla\mathbf{f}^*\|^2
	&\leq m^2\|J\|^2\sum_{t=1}^k\frac{\alpha_t}{\widetilde{\al}_k}\sum_{j=1}^m\|x_{j,t}-x^*\|^2\\
	&\leq 2m^2\|J\|^2\sum_{t=1}^k\frac{\alpha_t}{\widetilde{\al}_k}\sum_{j=1}^m\|x_{j,t}-\bar{x}_{t}\|^2+2m^3\|J\|^2\sum_{t=1}^k\frac{\alpha_t}{\widetilde{\al}_k}\|\bar{x}_{t}-x^*\|^2\\
	&\leq \dfrac{2m^2\|J\|^2}{\sigma}\sum_{t=1}^k\frac{\alpha_t}{\widetilde{\al}_k}\sum_{j=1}^m\|z_{j,t-1}-\bar{z}_{t-1}\|^2+2m^3\|J\|^2\sum_{t=1}^k\frac{\alpha_t}{\widetilde{\al}_k}\|\bar{x}_{t}-x^*\|^2\\
	&\le \dfrac{2m^2\|J\|^2}{\sigma}\frac{1}{\widetilde{\al}_k}\sum_{t=1}^\infty\alpha_t\sum_{j=1}^m\|z_{j,t-1}-\bar{z}_{t-1}\|^2+2m^3\|J\|^2\frac{1}{\widetilde{\al}_k}\sum_{t=1}^\infty\alpha_t\|\bar{x}_{t}-x^*\|^2,
	\end{aligned}
	\end{equation}
	where the third inequality follows from (\ref{consensus}).
	By (\ref{consunsus sum 0}) in Lemma \ref{lem:consis}
	\begin{equation*}
	\sum_{t=1}^\infty\alpha_t\sum_{j=1}^m\|z_{j,t-1}-\bar{z}_{t-1}\|^2\le \al_1 \sum_{t=1}^\infty\sum_{j=1}^m\|z_{j,t-1}-\bar{z}_{t-1}\|^2<\infty,  \quad \text{a.s.}
	\end{equation*}
	On the other hand,
	by \cite[ Lemma 9.5]{duchi2016asymptotic} and   (\ref{boundness}),
	$$\sum_{k=1}^\infty\alpha_k\|\bar{x}_{k}-x^*\|^2\le\frac{1}{c}\sum_{k=1}^\infty\alpha_k\left(f(\bar{x}_k)-f(x^*)\right)<\infty, \quad \text{a.s.}$$
	where $c$ is a random positive constant that depends on the bound $M\define\sup_{t}\|\bar{x}_t-x^*\|\vee1<\infty$. Therefore, we can argument that $\|\sum_{t=1}^k\dfrac{\alpha_t}{\widetilde{\al}_k}mJ\nabla\mathbf{f}_{t}-\nabla\mathbf{f}^*\|^2$ converges to $0$  almost surely  as $\widetilde \alpha_k\to \infty.$
	
	Next, we show $\sum_{t=1}^k\dfrac{\alpha_t}{\widetilde{\al}_k}mJS_{t}$ converges to $0$ almost surely. For this purpose, by the Kronecker lemma,  it is sufficient to show that
	$$ \sum_{t=1}^\infty\dfrac{\alpha_t}{\widetilde{\al}_t}mJS_{t}<\infty,\quad \text{a.s.}
	$$
	Note that  $\{\sum_{t=1}^k\alpha_tmJS_{t},\mathcal{F}_{k+1}\}$ is a martingale sequence as $\{S_{k},\mathcal{F}_{k+1}\}$ is a martingale difference sequence. Moreover,
	\begin{equation*}
	\sum_{t=1}^\infty\dfrac{1}{\widetilde{\al}_t^2} \mathbb{E}\left[\|\alpha_tmJS_{t}\|^2\big|\mathcal{F}_{t}\right]   \le  \sum_{t=1}^\infty\dfrac{1}{\widetilde{\al}_t^2}    \al_t^2m^2\|J\|^2\mathbb{E}\left[\sum_{j=1}^m\|s_{j,t}\|^2\bigg|\mathcal{F}_{t}\right]   \le     \sum_{t=1}^\infty\dfrac{1}{\widetilde{\al}_t^2}  4L_0^2m^3\|J\|^2\al_{t}^2 <\infty,
	\end{equation*}
	where the second inequality follows from (\ref{observe noise bound}). Then the convergence theorem
	for martingale difference sequences \cite[Appendix B.6, Theorem B 6.1]{chen2006stochastic} implies $\sum_{t=1}^k\dfrac{\alpha_t}{\widetilde{\al}_t}JS_{t}$ converges almost surely.
	This proof is completed.
\end{proof}

\section{Proof of Lemma \ref{lem:linear recur}}\label{Proof:linear recur}
\begin{proof}
	By the definition  (\ref{consensus dual averaging alrithm 1}), $\bar{x}_{k+1}$  satisfies the following KKT condition
	$$\bar{x}_{k+1}-\bar{z}_k+B^T\lambda_k+C^T\mu_k=0,$$
	where $\lambda_k\ge0$  and $\mu_k\ge0$ are the corresponding Lagrange multipliers. Then
	\begin{equation}\label{error}
	\bar{x}_{k+1}-x^*=\bar{x}_k-x^*+(\bar{z}_k-\bar{z}_{k-1})+B^T(\lambda_{k-1}-\lambda_k)+C^T(\mu_{k-1}-\mu_k).
	\end{equation}
	By the definition $\bar z_k$ in  (\ref{consensus dual averaging alrithm 1}),
	\begin{equation}\label{def:bar-zk}
	\begin{aligned}
	\bar{z}_k&=\dfrac{1}{m}\sum_{j=1}^m z_{j,k}
	=\dfrac{1}{m}\sum_{j=1}^m \left(\sum_{i=1}^m[A_k]_{ji}z_{i,k-1}-\alpha_{k}\nabla F_j(x_{j,k};\xi_{j,k})\right)\\
	&=\dfrac{1}{m}\sum_{j=1}^m\sum_{i=1}^m[A_k]_{ji}z_{i,k-1}-\frac{\alpha_{k}}{m}\sum_{j=1}^m\nabla F_j(x_{j,k};\xi_{j,k})\\
	&=\bar{z}_{k-1}-\frac{\alpha_{k}}{m}\sum_{j=1}^m\nabla F_j(x_{j,k};\xi_{j,k}),
	\end{aligned}
	\end{equation}
	where the fourth equality follows from that $A_k$ is doubly stochastic matrix. Then
	\begin{equation*}
	\begin{aligned}
	\bar{z}_k-\bar{z}_{k-1}&=-\frac{\alpha_{k}}{m}\sum_{j=1}^{m}\nabla F_j(x_{j,k};\xi_{j,k})\\
	&=-\frac{\alpha_{k}}{m}\sum_{j=1}^{m}\nabla f_j(x_{j,k})-\frac{\alpha_{k}}{m}\sum_{j=1}^{m}s_{j,k}\\
	&=\frac{\alpha_{k}}{m}\sum_{j=1}^{m}\left[\nabla f_j(\bar{x}_k)-\nabla f_j(x_{j,k})\right]-\frac{\alpha_{k}}{m}\big[\nabla f(\bar{x}_k)-\nabla f(x^*)\\
	&~~~~-\nabla^2 f(x^*)(\bar{x}_k-x^*)\big]-\frac{\alpha_{k}}{m}\left[\nabla f(x^*)+\nabla^2 f(x^*)(\bar{x}_k-x^*)\right]-\frac{\alpha_{k}}{m}\sum_{j=1}^{m}s_{j,k}.
	\end{aligned}
	\end{equation*}
	By left multiplying $P_B$ on both side of formula above,
	\begin{equation*}
	\begin{aligned}
	P_B(\bar{z}_k-\bar{z}_{k-1})
	&=\frac{\al_k}{m}\sum_{j=1}^{m}P_B\left[\nabla f_j(\bar{x}_k)-\nabla f_j(x_{j,k})\right]-\frac{\alpha_{k}}{m}P_B\big[\nabla f(\bar{x}_k)-\nabla f(x^*)\\
	&~~~~-\nabla^2 f(x^*)(\bar{x}_k-x^*)\big]-\frac{\alpha_{k}}{m}P_B\nabla^2 f(x^*)(\bar{x}_k-x^*)-\frac{\alpha_{k}}{m}\sum_{j=1}^{m}P_Bs_{j,k},
	\end{aligned}
	\end{equation*}
	where the equality follows from  the fact $P_B\nabla f(x^*)=0$. By the definition of $\Delta_k$ in (\ref{eq:deltat}) and the fact  $P_BB^T=0$, we have by left multiplying $P_B$ on both side of  (\ref{error}) that
	\begin{equation*}
	\begin{aligned}
	\bigtriangleup_{k+1}&=\bigtriangleup_{k}+P_B(\bar{z}_k-\bar{z}_{k-1})+P_BC^T(\mu_{k-1}-\mu_k)\\
	&=\bigtriangleup_{k}-\frac{\alpha_{k}}{m}P_B\nabla^2f(x^*)P_B(\bar{x}_k-x^*)-\frac{\alpha_{k}}{m}P_B\big[\nabla f(\bar{x}_k)-\nabla f(x^*)\\
	&~~~~-\nabla^2 f(x^*)(\bar{x}_k-x^*)\big]+\frac{\al_k}{m}\sum_{j=1}^{m}P_B\left[\nabla f_j(\bar{x}_k)-\nabla f_j(x_{j,k})\right]\\	 &~~~~+\frac{\alpha_{k}}{m}P_B\nabla^2f(x^*)(P_B-I_d)(\bar{x}_k-x^*)+P_BC^T(\mu_{k-1}-\mu_k)-\frac{\alpha_{k}}{m}\sum_{j=1}^{m}P_Bs_{j,k}\\
	&=\bigtriangleup_k-\alpha_{k}H\bigtriangleup_k+\al_k\left(\zeta_k+\eta_k+\epsilon_k+s_k\right),
	\end{aligned}
	\end{equation*}
	where $H$ is defined  in (\ref{eq:h}) and  $\zeta_k,\eta_k,\epsilon_k,s_k$ are defined in (\ref{errors}).
	Obviously,   formula above can  be rewritten as
	\begin{equation*}
	\bigtriangleup_{k+1}=\left[I_d-\alpha_{k}\left(H+D_k\right)\right]\bigtriangleup_k+\al_k\left(\eta_k+s_k+\epsilon_k\right),
	\end{equation*}
	where $D_k=-\zeta_k\dfrac{\bigtriangleup_k^T}{\|\bigtriangleup_k\|^2}$.
	The proof is completed.
\end{proof}

\section{Proof of Lemma \ref{lem:transition}}\label{Proof:transition}
\begin{proof}
	Part (i) is the well known result in linear algebra and we only prove part (ii).

	By definition
	\begin{equation*}
	U^TP_{B}U=
	\left(
	\begin{array}{cc}
	I_r&\textbf{0}_1\\
	\textbf{0}_2&\textbf{0}_3\\
	\end{array}\right),
	\end{equation*}
	where $\textbf{0}_1\in\mathbb{R}^{r\times(d-r)},\textbf{0}_2\in\mathbb{R}^{(d-r)\times r},\textbf{0}_3\in\mathbb{R}^{(d-r)\times(d-r)}$. Then for any $y\in\mathcal{Y}$, we have
	\begin{equation*}
	\begin{aligned}
	U^TP_{B}Hy&=U^TP_{B}HP_{B}y\\
	&=U^TP_{B}(UU^T)H(UU^T)P_{B}(UU^T)y\\
	&=(U^TP_{B}U)(U^THU)(U^TP_{B}U)U^Ty\\
	&=\left(
	\begin{array}{cc}
	I_r&\textbf{0}_1\\
	\textbf{0}_2&\textbf{0}_3\\
	\end{array}\right)U^THU
	\left(
	\begin{array}{cc}
	I_r&\textbf{0}_1\\
	\textbf{0}_2&\textbf{0}_3\\
	\end{array}\right)U^Ty.\\
	\end{aligned}
	\end{equation*}
	Let $U^THU=\begin{pmatrix}
	G_1&G_2\\
	G_3&G_4
	\end{pmatrix}$. Then,
	\begin{equation*}
	\begin{aligned}
	U^TP_{B}Hy
	&=\left(
	\begin{array}{cc}
	I_r&\textbf{0}_1\\
	\textbf{0}_2&\textbf{0}_3\\
	\end{array}\right)
	\left(
	\begin{array}{cc}
	G_1&G_2\\
	G_3&G_4\\
	\end{array}\right)
	\left(
	\begin{array}{cc}
	I_r&\textbf{0}_1\\
	\textbf{0}_2&\textbf{0}_3\\
	\end{array}\right)U^Ty\\
	&=\left(
	\begin{array}{cc}
	G_1&\textbf{0}_1\\
	\textbf{0}_2&\textbf{0}_3\\
	\end{array}\right)U^Ty=\left(
	\begin{array}{cc}
	G_1y_1\\
	\textbf{0}\\
	\end{array}\right),
	\end{aligned}
	\end{equation*}
	where $y_1\in\mathbb{R}^r$ determined by $U^Ty=(y_1^T,\textbf{0}^T)^T$, which means equality (\ref{decomp}) holds.
	
	For any nonzero vector $y_1\in\mathbb{R}^r$, let $y\define U(y_1^T,\textbf{0}^T)^T$. By the definition of matrix $U$, we have that $y$ is a nonzero vector and $y\in\mathcal{Y}$. Then
	\begin{equation}
	\begin{aligned}
	y_1^TG_1y_1&=\left(
	\begin{array}{c}
	y_1\\
	\textbf{0}\\
	\end{array}\right)^T
	\left(
	\begin{array}{cc}
	G_1&G_2\\
	G_3&G_4\\
	\end{array}\right)
	\left(
	\begin{array}{c}
	y_1\\
	\textbf{0}\\
	\end{array}\right)\\
	&=\left(
	\begin{array}{c}
	y_1\\
	\textbf{0}\\
	\end{array}\right)^T
	U^THU
	\left(
	\begin{array}{c}
	y_1\\
	\textbf{0}\\
	\end{array}\right)\\
	&=(U^Ty)^T
	U^THU(U^Ty)\\
	&=y^T(UU^T)H(UU^T)y\\
	&=y^THy\ge \mu\|y\|^2>0.
	\end{aligned}
	\end{equation}
	Therefore $G_1$ is positive definite. The proof is completed.
\end{proof}

\section{Proof of Lemma \ref{lem:asym effi}}\label{proof:lem:asym effi}
\begin{proof}
	Note that when $\bar{x}_{k},\bar{x}_{k-1}\in\{Bx=b,Cx<c\}$, $\bar{x}_{k}-\bar{x}_{k-1}$ can be expressed as
	\begin{equation*}
	\begin{aligned}
	\bar{x}_{k}-\bar{x}_{k-1}&=P_B\big[\bar{x}_{k}-\bar{x}_{k-1}\big]\\
	&=P_B\big[(\bar{z}_{k-2}-\bar{z}_{k-1})+B^T(\lambda_{k-2}-\lambda_{k-1})\big]\\
	&=P_B(\bar{z}_{k-2}-\bar{z}_{k-1}),
	\end{aligned}
	\end{equation*}
	and hence we obtain the recursion
	\begin{equation}\label{espacial expression of algorithm}
	\bar{x}_{k}=\bar{x}_{k-1}+P_B(\bar{z}_{k-2}-\bar{z}_{k-1}).
	\end{equation}
	For $\epsilon>0$ specified in Assumption \ref{ass:str conv}(ii), define the event $$\Upsilon_{l,k}=\{\|\bigtriangleup_j\|\leq \epsilon,B\bar{x}_j=b,C\bar{x}_j<c,~l\leq j\leq k\},$$ then $\Upsilon_{l,k}\in\mathcal{F}_{k}$. 
	Define $V_{l,k}=\|\bigtriangleup_k\|1_{\Upsilon_{l,k}}$ and note that $1_{\Upsilon_{l,k}}\leq1_{\Upsilon_{l,k-1}}$, it follows from (\ref{espacial expression of algorithm}) that
	\begin{equation}\label{recu-rela:t}
	\begin{aligned} V_{l,k}^2:=&\|\bigtriangleup_k\|^21_{\Upsilon_{l,k}}\leq\|\bigtriangleup_k\|^21_{\Upsilon_{l,k-1}}=\|\bigtriangleup_{k-1}+P_B(\bar{z}_{k-2}-\bar{z}_{k-1})\|^21_{\Upsilon_{l,k-1}}\\
	\le&V_{l,k-1}^2+2\langle \bigtriangleup_{k-1} 1_{\Upsilon_{l,k-1}},P_B(\bar{z}_{k-2}-\bar{z}_{k-1})\rangle+\|\bar{z}_{k-2}-\bar{z}_{k-1}\|^2,
	\end{aligned}
	\end{equation}
	where the non-expansiveness property of $P_B$ is involved in the last inequality of (\ref{recu-rela:t}).
	
	Taking the conditional expectation,
	\begin{equation}\label{analysis of error containing stop time}
	\begin{aligned}
	\mathbb{E}[V_{l,k}^2|\mathcal{F}_{k-1}]
	\leq &V_{l,k-1}^2+2\mathbb{E}[\langle P_B\bigtriangleup_{k-1} 1_{\Upsilon_{l,k-1}},\bar{z}_{k-2}-\bar{z}_{k-1}\rangle|\mathcal{F}_{k-1}]+\mathbb{E}[\|\bar{z}_{k-2}-\bar{z}_{k-1}\|^2|\mathcal{F}_{k-1}]\\
	=&V_{l,k-1}^2+2\mathbb{E}[\langle \bigtriangleup_{k-1} 1_{\Upsilon_{l,k-1}},\bar{z}_{k-2}-\bar{z}_{k-1}\rangle|\mathcal{F}_{k-1}]+\mathbb{E}[\|\bar{z}_{k-2}-\bar{z}_{k-1}\|^2|\mathcal{F}_{k-1}],
	\end{aligned}
	\end{equation}
	where the equality follows from the fact $P_B\bigtriangleup_{k-1}=\bigtriangleup_{k-1}$ due to $\bigtriangleup_{k-1}\in\{x:Bx=0\}$. Next we analyse the last two terms on the right-hand side of the equality of (\ref{analysis of error containing stop time}).
	
	For the third term, by (\ref{def:bar-zk}) and Assumption \ref{ass:lip func}(ii), we have
	\begin{equation}\label{1 analysis of error containing stop time}
	\begin{aligned}
	\mathbb{E}\left[\|\bar{z}_{k-2}-\bar{z}_{k-1}\|^2|\mathcal{F}_{k-1}\right]&=\mathbb{E}\Big[\|\frac{\alpha_{k-1}}{m}\sum_{j=1}^m\nabla F_j(x_{j,k-1};\xi_{j,k-1})\|^2\big|\mathcal{F}_{k-1}\Big]\\
	&\leq \frac{\alpha_{k-1}^2}{m}\sum_{j=1}^m\mathbb{E}\left[\|\nabla F_j(x_{j,k-1};\xi_{j,k-1})\|^2|\mathcal{F}_{k-1}\right]\\
	&\leq L_0^2\alpha_{k-1}^2,
	\end{aligned}
	\end{equation}
	where $L_0^2$ is defined in (\ref{eq: L0}).
	
	For the second term, substituting the following expression
	\begin{equation*}
	\begin{aligned}
	&\bar{z}_{k-2}-\bar{z}_{k-1}=-\frac{\alpha_{k-1}}{m}\sum_{j=1}^m\nabla F_j(x_{j,k-1};\xi_{j,k-1})\\
	=&-\frac{\alpha_{k-1}}{m}\nabla f(\bar{x}_{k-1})-\frac{\alpha_{k-1}}{m}\sum_{j=1}^m \big[\nabla f_j(x_{j,k-1})-\nabla f_j(\bar{x}_{k-1})\big]
	-\frac{\alpha_{k-1}}{m}\sum_{j=1}^m s_{j,k-1},
	\end{aligned}
	\end{equation*}
	and noticing that $\bigtriangleup_{k-1}=\bar{x}_{k-1}-x^*$ almost surely when $k$ is large enough by Lemma \ref{lem:con iden}, we arrive at
	\begin{equation}\label{2 analysis of error containing stop time}
	\begin{aligned}
	&\mathbb{E}[\langle \bigtriangleup_{k-1} 1_{\Upsilon_{l,k-1}},\bar{z}_{k-2}-\bar{z}_{k-1}\rangle|\mathcal{F}_{k-1}]\\
	=&\mathbb{E}\Big[\Big\langle\bigtriangleup_{k-1} 1_{\Upsilon_{l,k-1}},-\frac{\alpha_{k-1}}{m}\nabla f(\bar{x}_{k-1})+\frac{\alpha_{k-1}}{m}\sum_{j=1}^m \big[\nabla f_j(x_{j,k-1})-\nabla f_j(\bar{x}_{k-1})\big]\Big\rangle\Big|\mathcal{F}_{k-1}\Big]\\
	\leq& \frac{\alpha_{k-1}}{m}(f(x^*)-f(\bar{x}_{k-1}))1_{\Upsilon_{l,k-1}}+\frac{\epsilon\alpha_{k-1}}{m}\sum_{j=1}^m\|\nabla f_j(x_{j,k-1})-\nabla f_j(\bar{x}_{k-1})\|\\
	\leq& \frac{\alpha_{k-1}}{m}(f(x^*)-f(\bar{x}_{k-1}))1_{\Upsilon_{l,k-1}}+\frac{\epsilon\alpha_{k-1}L}{m}\sum_{j=1}^m\|\bar{z}_{k-1}-z_{j,k-1}\|.
	\end{aligned}
	\end{equation}
	Substituting (\ref{1 analysis of error containing stop time}) and (\ref{2 analysis of error containing stop time}) into (\ref{analysis of error containing stop time}), it follows that
	\begin{equation*}
	\begin{aligned} \mathbb{E}[V_{l,k}^2|\mathcal{F}_{k-1}]&\leq V_{l,k-1}^2+\frac{2\alpha_{k-1}}{m}(f(x^*)-f(\bar{x}_{k-1}))1_{\Upsilon_{l,k-1}}\\
	&~~~~+\frac{2\epsilon L\alpha_{k-1}}{m}\sum_{j=1}^m\|\bar{z}_{k-1}-z_{j,k-1}\|+L_0^2\alpha_{k-1}^2.
	\end{aligned}
	\end{equation*}
	By the restricted strongly convex property (\ref{restricted-strong-convexity}),
	we find that
	\begin{equation*}
	\mathbb{E}[V_{l,k}^2|\mathcal{F}_{k-1}]\leq\Big(1-\frac{2\epsilon'\alpha_{k-1}}{m}\Big)V_{l,k-1}^2
	+\frac{2\epsilon L\alpha_{k-1}}{m}\sum_{j=1}^m\|\bar{z}_{k-1}-z_{j,k-1}\|+L_0^2\alpha_{k-1}^2
	\end{equation*}
	for some constant $\epsilon'>0$. Taking expectation on both sides of the above inequality yields
	\begin{equation*}
	\begin{aligned}
	\mathbb{E}[V_{l,k}^2]\leq&\Big(1-\frac{2\epsilon'\alpha_{k-1}}{m}\Big)\mathbb{E}[V_{l,k-1}^2] +\frac{2\epsilon L\alpha_{k-1}}{m}\sum_{j=1}^m\mathbb{E}\left[\|\bar{z}_{k-1}-z_{j,k-1}\|\right]
	+L_0^2\alpha_{k-1}^2\\
	\leq& \exp\Big(-\frac{2\epsilon'\alpha_{k-1}}{m}\Big)\mathbb{E}[V_{l,k-1}^2]+\left(2\epsilon LD+L_0^2\right)\alpha_{k-1}^2,
	\end{aligned}
	\end{equation*}
	where the last inequality follows from the fact $\exp(-x)\geq (1-x), x\in(0,1)$ and $D$ is a constant specified in Lemma \ref{lem:consis}(iii) such that
	\begin{equation}\label{first order cons}
	\sup_k\alpha_{k}^{-1}\mathbb{E}\left[\|\bar{z}_{k}-z_{j,k}\|\right]\le\sup_k\sqrt{\alpha_{k}^{-2}\mathbb{E}\left[\|\bar{z}_{k}-z_{j,k}\|^2\right]}\le D<\infty.
	\end{equation}
	Taking iterations down to $l=[k/2]$ in such way, here $[x]$ denotes the integer part of $x$, we obtain
	\begin{equation}\label{analysis of error containing stop time 1}
	\begin{aligned}
	\mathbb{E}[V_{[k/2],k}^2]
	&\leq \exp\Big(-2\epsilon'\sum_{t=[k/2]}^{k-1}\frac{\alpha_{t}}{m}\Big) \mathbb{E}\left[\|\bigtriangleup_{[k/2]}\|^2\right]
	+\left(2\epsilon LD+L_0^2\right)\sum_{t=[k/2]}^{k-1}\alpha_{t}^2\exp\Big(-\sum_{l=t+1}^{k-1}\frac{\alpha_{l}}{m}\Big).
	\end{aligned}
	\end{equation}
	In what follows, we prove that $\sup_{k}\mathbb{E}\left[\|\bigtriangleup_{k}\|^2\right]<\infty$. In fact, taking expectation on both sides of (\ref{Lyapunov function 1}) and noting that the regularizer $\psi(x)=\frac{1}{2}\|x\|^2$ is $1$-strong convex, we find that
	\begin{equation*}
	\begin{aligned}
	\mathbb{E}\left[R_{k}\right]
	\leq& \mathbb{E}\left[R_{k-1}\right]-\frac{\alpha_k}{m}\mathbb{E}\left[\left(f(\bar{x}_k)-f(x^*)\right)\right]+\frac{2L_0\alpha_{k}}{m}\sum_{j=1}^m\mathbb{E}\left[\|z_{j,k-1}-\bar{z}_{k-1}\|\right]\\
	&+\frac{1}{m}\sum_{j=1}^m\mathbb{E}\left[\|\bar{z}_{k-1}-z_{j,k-1}\|^2\right]+L_0^2\alpha_{k}^2\\
	\le& \mathbb{E}\left[R_{k-1}\right]-\frac{\alpha_k}{m}\mathbb{E}\left[\left(f(\bar{x}_k)-f(x^*)\right)\right]+\frac{2L_0\alpha_{k}}{m}\sum_{j=1}^m\sqrt{\mathbb{E}\left[\|z_{j,k-1}-\bar{z}_{k-1}\|^2\right]}\\
	&+\frac{1}{m}\sum_{j=1}^m\mathbb{E}\left[\|\bar{z}_{k-1}-z_{j,k-1}\|^2\right]+L_0^2\alpha_{k}^2\\
	\le& \mathbb{E}\left[R_{k-1}\right]-\frac{\alpha_k}{m}\mathbb{E}\left[\left(f(\bar{x}_k)-f(x^*)\right)\right]+2L_0\sqrt{D}\al_{k}\al_{k-1}+D^2\al_{k-1}^2+L_0^2\alpha_{k}^2\\
	\le& \mathbb{E}\left[R_{k-1}\right]-\frac{\alpha_k}{m}\mathbb{E}\left[\left(f(\bar{x}_k)-f(x^*)\right)\right]+\left(2L_0\sqrt{D}+D^2+L_0^2\right)\alpha_{k-1}^2,
	\end{aligned}
	\end{equation*}
	where the second inequality follows from the Cauchy-Schwarz inequality, the third inequality from (\ref{first order cons}), and the last inequality from $\al_k$ being nonincreasing.
	Summing the above inequality from $1$ to $k$ yields
	\begin{equation*}
	\begin{aligned}
	\mathbb{E}\left[R_{k}\right]
	&\leq \mathbb{E}\left[R_{0}\right]-\frac{\alpha_k}{m}\sum_{t=1}^k\mathbb{E}\left[\left(f(\bar{x}_t)-f(x^*)\right)\right]+\left(2L_0\sqrt{D}+D^2+L_0^2\right)\sum_{t=1}^{k-1}\alpha_{t}^2,
	\end{aligned}
	\end{equation*}
	which implies $\sup_{k}\mathbb{E}\left[R_{k}\right]<\infty$. Therefore, we have
	$$\sup_{k}\mathbb{E}\left[\|\bigtriangleup_{k}\|^2\right]\leq\|P_B\|^2\sup_{k}\mathbb{E}\left[\|\bar{x}_k-x^*\|^2\right]\le 2\|P_B\|^2\sup_{k}\mathbb{E}\left[R_k\right]<\infty,$$
	where the second inequality follows from (\ref{Fc bound}). Denote $D_1=\sup_{k}\mathbb{E}\left[\|\bigtriangleup_{k}\|^2\right]$. Then by (\ref{analysis of error containing stop time 1}) and the fact that there has a constant $D_2$ such that $\sum_{t=[k/2]}^{k-1}\alpha_{t}>D_2 k^{1-\al}$, we have
	\begin{equation}\label{analysis of error containing stop time 2}
	\begin{aligned}
	\mathbb{E}\left[\|\bigtriangleup_k\|^21_{\Upsilon_{[k/2],k}}\right]
	\le& D_1\exp(-D_2 {k}^{1-\al}) \\
	&+\left(2\epsilon DL+L_0^2\right)\sum_{t=[k/2]}^{k-1}\alpha_{t}^2\exp\left(-D_2\left({k}^{1-\al}-t^{1-\al}\right)\right).
	\end{aligned}
	\end{equation}
	
	By Theorem \ref{thm:con rate} and Lemma \ref{lem:con iden}, for any given $a>0$,
	\begin{equation}\label{Pr:conv&con iden}
	\P\left\{\sup_{2k_0\le t<\infty}\|\bigtriangleup_t\|<\epsilon,~K<k_0\right\}>1-a,
	\end{equation}
	if $k_0$ is sufficiently large, where $K$ is a finite random integer specified in Lemma \ref{lem:con iden}. Summing (\ref{analysis of error containing stop time 2}) from $2k_0$ to $k$ yields
	\begin{equation*}
	\begin{aligned}
	\sum_{t=2k_0}^k\dfrac{1}{\sqrt{t}}\mathbb{E}[\|\bigtriangleup_t\|^21_{\Upsilon_{[t/2],t}}]
	&\le D_1\sum_{t=2k_0}^k\dfrac{1}{\sqrt{t}}\exp\left(-D_2 t^{1-\al}\right) +\left(2\epsilon DL+L_0^2\right)\sum_{t=2k_0}^k\dfrac{1}{\sqrt{t}}\dfrac{\log t}{t^{\al+1/2}}
	\end{aligned}
	\end{equation*}
	which follows from \cite[Lemma 15.5, Part 15]{duchi2016asymptotic}. Let $k\rightarrow\infty$, we have
	\begin{equation*}
	\sum_{t=2k_0}^{\infty}\mathbb{E}\left[\dfrac{1}{\sqrt{t}}\|\bigtriangleup_t\|^21_{\Upsilon_{[t/2],t}}\right]<\infty,
	\end{equation*}
	and by the monotone convergence theorem,
	\begin{equation}\label{analysis of error containing stop time 3}
	\sum_{t=2k_0}^{\infty}\dfrac{1}{\sqrt{t}}\|\bigtriangleup_t\|^21_{\Upsilon_{[t/2],t}}<\infty~\text{a.s.}
	\end{equation}
	which means that
	\begin{equation}\label{analysis of error containing stop time 4}
	\begin{aligned}
	\left\{\sup_{t\geq2k_0}\|\bigtriangleup_t\|<\epsilon,K<k_0\right\}
	\subset&\left\{\sup_{t\geq2k_0}\|\bigtriangleup_t\|<\epsilon,B\bar{x}_t=b,C\bar{x}_t<c,\forall t\geq 2k_0\right\}\\
	\subset&\left\{\sum_{t=2k_0}^{\infty}\dfrac{1}{\sqrt{t}}\|\bigtriangleup_t\|^2<\infty\right\},
	\end{aligned}
	\end{equation}
	
	Combining (\ref{Pr:conv&con iden}) with (\ref{analysis of error containing stop time 4}) shows that
	\begin{equation*}
	\P\left\{\sum_{t=2k_0}^{\infty}\dfrac{1}{\sqrt{t}}\|\bigtriangleup_t\|^2<\infty\right\}>1-a,
	\end{equation*}
	or equivalently
	\begin{equation*}
	\P\left\{\sum_{t=1}^{\infty}\dfrac{1}{\sqrt{t}}\|\bigtriangleup_t\|^2<\infty\right\}>1-a.
	\end{equation*}
	This verifies
	\begin{equation*}
	\sum_{t=1}^{\infty}\dfrac{1}{\sqrt{t}}\|\bigtriangleup_t\|^2<\infty~\text{a.s.}
	\end{equation*}
	because $a>0$ can be arbitrarily small. Finally,  an application of the Kronecker lemma implies (\ref{bigtri}). This complete the proof.
\end{proof}

\section{Results on stochastic approximation}
\label{sec:SA}
For ease of reading,  we recall some results on stochastic approximation from  \cite{Robbins1971A} and    \cite{chen2006stochastic}.
\begin{lem}\label{lem:Robbins-Siegmund} \cite{Robbins1971A}
	Let $\{\mathcal{F}_{k}\}$ be an nondecreasing sequence of $\sigma$-algebra and $\{v_k\}$, $\{a_k\}$,
	$\{b_k\}$,and $\{\phi_k\}$ be the four nonnegative sequence adopted to $\mathcal{F}_{k}$.
	Assume that for all $k$,
	$$\mathbb{E}[v_{k+1}|\mathcal{F}_{k}]\le(1+a_k)v_k+b_k-\phi_k.$$
	If $\sum_{k=1}^\infty a_k<\infty $ and $\sum_{k=1}^\infty b_k<\infty $  almost surely. Then $\{v_k\}$ converges to a finite random variable $v_{\infty} $ and $\sum_{k=1}^\infty\phi_k<\infty $ almost surely.
\end{lem}

\begin{lem}\cite[Lemma 3.1.1]{chen2006stochastic}\label{lem:rate}
	Suppose $d\times d$-dimension matrix $F_k\rightarrow F$, $F$ is a stable matrix ,that is, every eigenvalue of $F$ has strictly negative real part. If step-size $\al_{k}$ satisfies $$\al_k>0, \al_k \xrightarrow[k \to \infty]{} 0, \sum\limits_{k=1}^{\infty} {\al_k}=\infty,$$
	and $d$-dimension vectors $\{e_k\},\{\upsilon_k\}$ satisfy the following conditions
	\begin{equation}\label{rate condition}
	\sum_{k=1}^\infty \al_{k}e_k<\infty,~\upsilon_k\rightarrow 0,
	\end{equation}
	then $\{y_k\}$ defined by the following recursion with arbitrary initial value $x_0$ tends to zero:
	\begin{equation}\label{linear reccursion}
	y_{k+1}=y_k+\al_{k}F_ky_k+\al_{k}\left(e_k+\upsilon_k\right).
	\end{equation}
\end{lem}

\begin{lem}\label{lem:asym norm} \cite[Theorem 3.3.1]{chen2006stochastic}
	Let $\{y_k\}$ be given by (\ref{linear reccursion}) with an arbitrarily given initial value. Assume the following conditions holds:
	\begin{itemize}
		\item [\rm{(i)}]$\al_k>0,\al_{k}\rightarrow0$ as $k\rightarrow\infty$, $\sum_{k=1}\al_{k}=\infty$, and
		\begin{equation*}
		\al_{k+1}^{-1}-\al_{k}^{-1}\rightarrow a\ge 0~\text{as}~k\rightarrow\infty;
		\end{equation*}
		\item [\rm{(ii)}] $F_k\rightarrow F$ and $F+\dfrac{a}{2}$ is stable;
		\item [\rm{(iii)}]
		\begin{equation*}
		\nu_k=o(\sqrt{\al_{k}}),   \quad e_k=\sum_{t=0}^\infty C_ts_{k-t},s_t=0~\text{for}~t<0,
		\end{equation*}
		where $C_t$ are $d\times d$ constant matrices with $\sum_{t=0}^\infty \|C_t\|<\infty$ and $\{s_k,\mathcal{F}_k\}$ is a martingale difference sequence of $d-$dimension satisfying the following conditions
		\begin{equation}\label{c1}
		\mathbb{E}\left[s_k|\mathcal{F}_{k-1}\right]=0,~\sup_k\mathbb{E}\left[\|s_k\|^2\big|\mathcal{F}_{k-1}\right]\le\sigma~\text{with}~\sigma~\text{being a constant,}
		\end{equation}
		\begin{equation}\label{c2}
		\lim_{k\rightarrow\infty}\mathbb{E}\left[s_ks_k^T\big|\mathcal{F}_{k-1}\right]=\lim_{k\rightarrow\infty}\mathbb{E}\left[s_ks_k^T\right]\define S_0
		\end{equation}
		and
		\begin{equation}\label{c3}
		\lim_{N\rightarrow\infty}\sup_{k}\mathbb{E}\left[\|s_k\|^21_{\{\|s_k\|>N\}}\right]=0.
		\end{equation}
		Then $\dfrac{y_k}{\sqrt{\al_{k}}}$ is asymptotically normal:
		\begin{equation*}
		\dfrac{y_k}{\sqrt{\al_{k}}}\xrightarrow[k\rightarrow\infty]{d}N(0,S),
		\end{equation*}
		where
		\begin{equation*}
		S=\int_{0}^{\infty}e^{(F+a/2I)t}\sum_{k=0}^\infty C_kS_0\sum_{k=0}^\infty C_k^Te^{(F^T+a/2I)t}\d t.
		\end{equation*}
	\end{itemize}
\end{lem}
\end{appendices}

\end{document}